\documentclass[12pt]{amsart}
\allowdisplaybreaks

\usepackage{amssymb
,amsthm
,amsmath
,amscd
,mathtools	
,mathdots
,leftidx
}
\usepackage[all]{xy}
\usepackage{url}
\usepackage{graphicx}
\usepackage{hyperref}
\usepackage[top=30truemm,bottom=35truemm,left=25truemm,right=25truemm]{geometry}
\RequirePackage[dvipsnames,usenames]{color}
\usepackage{soul,xcolor}
\usepackage[normalem]{ulem}
\usepackage{soul}
\usepackage{caption}

\makeatletter
\def\@tocline#1#2#3#4#5#6#7{\relax
  \ifnum #1>\c@tocdepth 
  \else
    \par \addpenalty\@secpenalty\addvspace{#2}%
    \begingroup \hyphenpenalty\@M
    \@ifempty{#4}{%
      \@tempdima\csname r@tocindent\number#1\endcsname\relax
    }{%
      \@tempdima#4\relax
    }%
    \parindent\z@ \leftskip#3\relax \advance\leftskip\@tempdima\relax
    \rightskip\@pnumwidth plus4em \parfillskip-\@pnumwidth
    #5\leavevmode\hskip-\@tempdima
      \ifcase #1
       \or\or \hskip 1em \or \hskip 2em \else \hskip 3em \fi%
      #6\nobreak\relax
    \hfill\hbox to\@pnumwidth{\@tocpagenum{#7}}\par
    \nobreak
    \endgroup
  \fi}
\makeatother

\newcommand{\GL}{\mathrm{GL}}

\newcommand{\GSp}{\mathrm{GSp}}
\newcommand{\Sp}{\mathrm{Sp}}
\newcommand{\F}{\mathbb{F}}
\newcommand{\GU}{\mathrm{GU}}
\newcommand{\Gm}{\mathbb{G}_{m}}
\newcommand{\ctext}[1]{\raise0.2ex\hbox{\textcircled{\scriptsize{#1}}}}

\AtBeginDocument{%
  \def\MR#1{}
}

\newcommand{\co}{\mathcal{O}}
\newcommand{\p}{\mathfrak{p}}

\newcommand{\ow}{&\quad&\text{otherwise }}
\newcommand{\wh}{&\quad&\text{when }}

\newcommand{\Z}{\mathbb{Z}}
\newcommand{\PP}{\mathbb{P}}
\DeclareMathOperator{\diag}{diag}
\DeclareMathOperator{\Ann}{Ann}
\DeclareMathOperator{\ord}{ord}
\DeclareMathOperator{\Res}{Res}
\DeclareMathOperator{\chara}{char}
\DeclareMathOperator{\row}{row}
\DeclareMathOperator{\col}{col}
\DeclareMathOperator{\Spec}{Spec}
\newcommand{\bsp}{b_{\mathrm{sp}}}
\newcommand{\Vsymp}{V^{\mathrm{symp}}_{b}}
\newcommand{\Vsympsp}{V^{\mathrm{symp}}_{b_{\mathrm{sp}}}}
\newcommand{\Lsymp}{\mathcal{L}_{b}^{\mathrm{symp}}}
\newcommand{\Lsymprat}{\la_{b}^{\mathrm{symp,rat}}}

\newcommand{\Lsympsp}{\mathcal{L}_{b_{\mathrm{sp}}}^{\mathrm{symp}}}
\newcommand{\Lsympbar}{\overline{\la_{b}^{\mathrm{symp}}}}
\newcommand{\la}{\mathcal{L}}
\newcommand{\lac}{\mathcal{L}_{\mathrm{coord}}^{b,m,r}}
\newcommand{\laczero}{\mathcal{L}_{\mathrm{coord}}^{b,0,r}}
\newcommand{\red}{\mathrm{red}}
\newcommand{\hpr}{\widetilde{h}}
\usepackage[abbrev,alphabetic]{amsrefs}

\theoremstyle{plain}
\newtheorem{theorem}[subsubsection]{Theorem}
\newtheorem{proposition}[subsubsection]{Proposition}
\newtheorem{lemma}[subsubsection]{Lemma}
\newtheorem{claim}[subsubsection]{Claim}

\theoremstyle{definition}

\newtheorem{definition}[subsubsection]{Definition}
\newtheorem{remark}[subsubsection]{Remark}
\newtheorem*{ack}{Acknowledgements}

\newtheorem*{lemma*}{Lemma}

\newcommand{\cred}{\color{black}}

\usepackage{seqsplit}
\usepackage{xstring}

\def\hsymb#1{\mbox{\strut\rlap{\smash{\Huge$#1$}}\quad}}
\def\Lsymb#1{\mbox{\strut\rlap{\smash{\Large$#1$}}\quad}}

\def\c#1#2{$\mathrm{col}_{#1\to#2}$}
\def\r#1#2{$\mathrm{row}_{#1\to#2}$}
\begin{document}

\title{On the semi-infinite Deligne--Lusztig varieties for $\GSp$}

\author{Teppei Takamatsu}
\address{Department of Mathematics (Hakubi Center), Kyoto University, Kyoto 606-8502, Japan}
\email{teppeitakamatsu.math@gmail.com}
\date{}

\begin{abstract}
{\cred
We prove that Lusztig's semi-infinite Deligne--Lusztig variety for $\mathrm{GSp}$ (and its inner form) is isomorphic, as a set with action, to an affine Deligne--Lusztig variety at infinite level, generalizing a result of Chan--Ivanov. Furthermore, we show that a component of some affine Deligne--Lusztig variety $X^0_{w_r}(b)_{\mathcal{L}}$ for $\mathrm{GSp}$ can be written, up to perfection, as a direct product of a classical Deligne--Lusztig variety with an affine space. We also study the varieties $X_h$ defined by Chan and Ivanov, and show that $X_h$ at infinite level can be realized as a subset of semi-infinite Deligne--Lusztig varieties defined using components of affine Deligne--Lusztig varieties such as $X^0_{w_r}(b)_{\mathcal{L}}$ above, even in the $\mathrm{GSp}$ case. This reinterprets previous constructions of representations from $X_h$ as instances of Lusztig's conjectural picture.}
\end{abstract}

\maketitle

\section{Introduction}
Deligne--Lusztig varieties, {\cred introduced in the celebrated paper \cite{Deligne1976}}, 
are algebraic varieties over $\overline{\F}_{q}$ whose $\ell$-adic cohomology {\cred provides a geometric construction of many representations of finite reductive groups.} 
More precisely,
{\cred given a reductive group $G$ over $\F_{q}$ and a maximal $\F_{q}$-torus $T\subset G$, the associated Deligne–Lusztig variety $X$ carries actions of $T(\F_{q})$ and $G(\F_{q})$, and from these actions one constructs a (virtual) representation of $G(\F_{q})$ from a character of $T(\F_{q})$.}
{\cred One of the main theorems of Deligne and Lusztig}
asserts that {\cred every} irreducible representation of $G(\F_{q})$ appears in some {\cred such} virtual representation.

{\cred It is natural to seek to extend this theory to}
$p$-adic fields,
and this expectation has great significance {\cred from} the perspective of the local Langlands correspondence.
To this end, 
{\cred there are two approaches to generalizing Deligne--Lusztig varieties.}
{\cred The first is the} semi-infinite Deligne--Lusztig variety in the sense of Feigin--Frenkel \cite{Feigin1990}, a subset of the semi-infinite flag manifold {\cred cut out} by the Deligne--Lusztig condition.
{\cred The second} is {\cred the} affine Deligne--Lusztig variety, {\cred defined by an analogous} condition in an affine flag variety (\cite{Rapoport2005}).
Lusztig proposed {\cred a conjectural construction of supercuspidal representations}
{\cred realizing} the local Langlands correspondence {\cred via the} homology of semi-infinite Deligne--Lusztig varieties (\cite{Lusztig1979}). 
{\cred This approach has since been developed in greater detail, especially for division algebras (\cite{Boyarchenko2012}, \cite{Chan2018}, and \cite{Chan2020}).}
See also \cite{Ivanov2022}, where the $p$-adic Deligne--Lusztig varieties are defined 
{\cred functorially by following the definition of semi-infinite Deligne--Lusztig varieties.}

{\cred Affine Deligne–Lusztig varieties have likewise been studied from the perspective of the local Langlands correspondence \cite{Ivanov2016}, \cite{Ivanov2018}, and—thanks to their relationship with the reductions of integral models of Shimura varieties—their geometry has been analyzed extensively \cites{Viehmann2008}, \cite{Viehmann2008a}, \cite{Shimada2021}, \cite{Shimada2022}, \cite{Goertz2006}, \cite{Goertz2015}, \cite{Goertz2018}, \cite{He2022}; see the surveys \cite{He2018},\cite{He2021}.}

In this paper, we study the relation between semi-infinite and affine Deligne--Lusztig varieties for $\GSp$.
In \cite{Chan2021}, Chan and Ivanov 
established such a relation for $\GL_{n}$.
More precisely, they describe semi-infinite Deligne--Lusztig varieties as an inverse limit of affine Deligne--Lusztig varieties of higher levels (namely, ``affine Deligne--Lusztig varieties {\cred at} infinite level''), and give a geometric realization of the local Langlands and Jacquet--Langlands correspondence on them.
In this paper, we will give an analogue of Chan--Ivanov's result for $\GSp$. 
Before stating precise results, we {\cred summarize} what we prove roughly:
\begin{enumerate}
\item
{\cred A semi-infinite Deligne--Lusztig variety associated with a coxeter element $w$ for $\GSp_{2n}$}
can be described {\cred as} the inverse limit of affine Deligne--Lusztig varieties of {\cred higher levels} {\cred (``at infinite level'')}. More precisely, we have
$$
(\mathop{\varprojlim}\limits_{r>m}\dot{X}^{m}_{w_r}(b))(\overline{k})
\simeq X_{w}^{(U)}(b).
$$
Here, {\cred $b$ represents} a basic $\sigma$-conjugacy class {\cred in} $\GSp_{2n} (\breve{K})$ {\cred with Kottwitz invariant $\kappa (b) = k$}, 
and $w_{r}$ (resp.\,$w$) is a certain representative of {\cred an element of the} affine Weyl group (resp.\,{\cred a coxeter element of the finite} Weyl group) of $\GSp_{2n}$ {\cred such that $w_r$ projects to $w$}.
{\cred The varieties} $X^{m}_{w_r}(b)$ and $\dot{X}^{m}_{w_r}(b)$ are {\cred the level-$I^{m}$ affine Deligne--Lusztig varieties} defined by $b$ and $w_{r}$
{\cred (see Subsection  \ref{subsection:DLset})}.
\item
{\cred 
A connected component $X^{{\cred 0}}_{w_r}(b)_{\mathcal{L}}$ of
$X^{{\cred 0}}_{w_r}(b)(\overline{k})$ is, up to perfection,
{\cred isomorphic to a product $X_{\overline{w}}^{\overline{B}}\times\mathbb{A}$ of a classical Deligne–Lusztig variety and an affine space.}
Here, $\overline{G}$ is 
\[
\overline{G} :=
\begin{cases}
\GSp_{2n} \textup{ over } \F_{q} & \textup{ if } n \textup{ is even and } k=0, \\
\GSp_{n} \textup{ over } \F_{q^{2}}  & \textup{ if } n \textup{ is even and } k=1, \\ 
\GSp_{2n} \textup{ over } \F_{q} & \textup{ if } n \textup{ is odd and } k=0,  \\
\GU_{n} \textup{ over } \F_{q} & \textup{ if } n \textup{ is odd and } k=1,
\end{cases}
\]
and $\overline{w}$ is a suitable element of the Weyl group of $\overline{G}$ (see Definitions \ref{def:findlv} and  \ref{def:findlvodd}).
Such a product description does not hold in general (cf.\ \cite{Shimada2022} and \cite{He2022}).}
\item
The scheme $X_{h}$, defined in \cite{Chan2019} (see also \cite{Chan2021a}) is also related to the semi-infinite Deligne--Lusztig variety
$X_{w}^{(U)}(b)$.
{\cred More precisely, the subset (which can be regarded as a ``component'') of $X_{w}^{(U)}$ defined by
$\mathop{\varprojlim}\limits_{r>m}\dot{X}^{m}_{w_r}(b)_{\mathcal{L}}$
(where $\dot{X}^{m}_{w_r}(b)_{\mathcal{L}}$ is a component of $X^{m}_{w_r}(b)(\overline{k})$), is isomorphic to $\mathop{\varprojlim}\limits_{h}X_h$.
In particular, $X_h$ is isomorphic to a quotient of that component.}
In \cite{Chan2021a}, they studied L-packets using varieties $X_{h}$.
On the other hand, in \cite{Lusztig1979}, Lusztig expects that we can construct supercuspidal representations by using semi-infinite Deligne--Lusztig varieties $X_{w}^{(U)} (b)$ in some sense.
{\cred Thus the constructions of representations from $X_h$ can be viewed as instances of Lusztig’s expectation in \cite{Lusztig1979}.}
\end{enumerate}
We will explain our results in detail.
Let $K$ be a non-archimedean local field, and $\breve{K}$ a completion of the maximal unramified extension of $K$.
Let $V := \breve{K}^{2n}$ be the vector space with the symplectic form $\Omega$ given by (\ref{eqn:Omega}).
We define $w_{r} \in \GSp(V)$ by
\begin{eqnarray*}
w_{r}&:=&
\left(\begin{array}{cccc|cccc}
&&&& (-1)^{n+1}\varpi^{r+k}&&&\\
\varpi^{-r}&&\hsymb{0}&&&&&\\
&\ddots&&&&\hsymb{0}&&\\
&&\varpi^{-r}&&&&&\\\hline
&&&&&\varpi^{r+k}&&\\
&&&&&&\ddots&\\
&\hsymb{0}&&&\hsymb{0}&&&\varpi^{r+k}\\
&&&\varpi^{-r}&&&&
\end{array}\right),
\end{eqnarray*}
and a representative $b \in \GSp (V)$ of a basic $\sigma$-conjugacy class with $\lambda (b) = -\varpi^{k}$ for $k=0,1$.
We also put $w := w_{0}$.
Here, $\lambda$ is the similitude character of $\GSp(V)$.
The key method in \cite{Chan2021} is parameterizing affine (resp.\,semi-infinite) Deligne--Lusztig varieties by a set $V_{b}^{\mathrm{adm}}$.
Instead of $V_{b}^{\mathrm{adm}}$, we introduce the set $\Vsymp$, defined by
\[
\Vsymp:=\{v\in V\mid \langle v,F(v)\rangle=\cdots=\langle v,F^{n-1}(v)\rangle=0,
\langle v,F^{n}(v)\rangle\neq0\}.
\]
The parameterization is given by the following theorem.

\begin{theorem}[See Theorem \ref{theorem:comparison} for the precise statements]
\label{theorem:introcomparison}
If $r+k\geq m+1$, there are the following two $J_b(K)$-equivariant commutative diagrams with bijective horizontal arrows.
\[
\xymatrix{
\{v\in \Vsymp\mid \alpha:=\langle v,F^{n}(v)\rangle \in K^{\times}\} \ar[r]^-{g_{b,0}}_-{\sim} \ar[d] & X_{w}^{(U)}(b) \ar[d]\\
\Vsymp/\breve{K}^{\times} \ar[r]^{g_{b,0}}_{\sim} &  X_{w}^{(B)}(b)
}
\]
\[
\xymatrix{
\{v\in \Vsymp\mid\frac{\sigma(\alpha)}{\alpha}\equiv1 \mod \p^{m+1}\}/\dot{\sim}_{b,m,r} \ar[r]^-{g_{b,r}}_-{\sim} \ar[d] & \dot{X}_{w_{r}}^{m}(b)(\overline{k}) \ar[d] \\
\Vsymp/\sim_{b,m,r}  \ar[r]^-{g_{b,r}}_-{\sim} & X_{w_{r}}^{m}(b)(\overline{k})
}
\]
Here, $\sim_{b,m,r}$ and $\dot{\sim}_{b,m,r}$ are some equivalence relations.
\end{theorem}

For the definition of objects on the right-hand sides, see Subsection \ref{subsection:DLset}.
By studying the relation $\sim_{b,m,r}$ in detail, we have {\cred the} desired comparison theorem
$$
(\mathop{\varprojlim}\limits_{r>m}\dot{X}^{m}_{w_r}(b))(\overline{k})\simeq X_{w}^{(U)}(b).
$$
The key part of Theorem \ref{theorem:introcomparison} is the surjectivity of {\cred the} horizontal arrows. Its proof is given by direct computation using elementary transformations in $\GSp$, as in the proof of \cite[Theorem 6.5]{Chan2021}.
However, since we work {\cred with} $\GSp$, elementary transformations are more complicated, and the proof needs to be carried out with greater care.
To get the surjectivity for the lower diagram, we need an infinite number of elementary transformations (whose infinite product converges). This is in contrast to the proof of \cite[Theorem 6.5]{Chan2021}.

Next, we will consider the structure of {\cred the} connected components of affine Deligne--Lusztig varieties.
In the following, we fix a Coxeter-type representative $b$ of a basic $\sigma$-conjugacy class (see Subsection \ref{subsection:repb}).
{\cred By using Viehmann's argument in \cite{Viehmann2008a},}
we obtain the description of a connected component of an affine Deligne--Lusztig variety in terms of $\Vsymp$.
More precisely, we can define a subset $\Lsymp \subset \Vsymp$, and the corresponding subset $X_{w_{r}}^{m}(b)_{\la}$ gives a component of $X_{w_{r}}^{m} (b)$.
Moreover, one can show that this component is contained in an affine Schubert cell (see Proposition \ref{prop:cell containment}).
Since the structure of {\cred an} affine Schubert cell is well known (see \cite[Lemma 4.7]{Chan2021} for example), now we can study the structure of this component in an explicit way.
The main structure theorem is the following.

\begin{theorem}[see Theorem \ref{theorem:restradlv}]
\label{theorem:introstradlv}
Suppose that $r+k \geq 1$.
Then we have a decomposition of $\overline{\F}_{q}$-schemes
\[
X^{0}_{w_{r}} (b)_{\la}^{\mathrm{perf}} \simeq X_{\overline{w}}^{\overline{B}, \mathrm{perf}} \times \mathbb{A}^{\mathrm{perf}}.
\]
Here, $X_{\overline{w}}^{\overline{B}}$ is a classical Deligne--Lusztig {\cred variety} associated with some reductive group $\overline{G}$ which depends on $n$ and $k$ (see Definition \ref{definition:reductive}),
$\mathbb{A}$ is an affine space over $\overline{\F}_{q}$, and $\mathrm{perf}$ {\cred denotes} the perfection.
\end{theorem}
We sketch the proof in the following.
Since we know the defining equations of $X^{0}_{w_{r}} (b)_{\la}$ in the affine Schubert cell, 
it is imaginable that the proof can be done through explicit calculations.
More precisely, to obtain a decomposition as in the theorem, it suffices to solve these equations over the {\cred coordinate} ring of a classical Deligne--Lusztig variety $X_{\overline{w}}^{\overline{B}}$.
In the simplest case, the equations and solutions are constructed as follows:
Consider the open subscheme $\overline{\mathcal{L}} \subset \mathbb{A}_{x,w}$ defined by \begin{equation}
\label{eqn:toycase}
xw^{q^2}  -wx^{q^2} \neq 0.
\end{equation}
We define the reduced locally closed subscheme $\mathcal{L}_{1}$ of $\mathbb{A}_{x,y,z,w}$ defined by (\ref{eqn:toycase}) and the equation
\begin{equation}
\label{eqn:toycase2}
xz^q-yw^q+zx^q-wy^q =0.
\end{equation}
Actually, 
{\cred in the case where $G= \GSp_{4}$ and $k=1$, a classical Deligne--Lusztig variety $X_{\overline{w}}^{\overline{B}}$ is the projectivization of $\overline{\mathcal{L}}$,}
and 
{\cred a component $X^{0}_{w_{r}} (b)_{\la}$ of an affine Deligne--Lusztig variety is the projectivization of $\mathcal{L}_{1}$ (see Subsections \ref{subsection:structureofLsymp_even} and \ref{subsection:structureofADLV}).}
Therefore, in this case, an essential part of the proof is to solve the equation (\ref{eqn:toycase2}) over $\overline{\mathcal{L}}$ up to perfection so that we have 
\[
\mathcal{L}_{1}^{\mathrm{perf}} \simeq  \overline{\mathcal{L}}^{\mathrm{perf}} \times \mathbb{A}^{1,\mathrm{perf}}.
\]
We put 
\begin{align*}
y' &:= y, \\
z' &:= x^{\frac{1}{q}} z -  w^{\frac{1}{q}}y.
\end{align*}
Then the left-hand side of (\ref{eqn:toycase2}) can be written as 
\begin{equation}
\label{eqn:toycase3}
z'^{q} + x^{q-\frac{1}{q}} z' + (x^{q-\frac{1}{q}} w^{\frac{1}{q}} - w^{q}) y'.
\end{equation}
Since we have 
\[
((x^{q-\frac{1}{q}} w^{\frac{1}{q}} - w^{q}) x^{\frac{1}{q}})^{q} 
= x^{q^{2}} w - w^{q^{2}}x,
\]
the equation (\ref{eqn:toycase3}) can be solved over $\overline{\la}^{\mathrm{perf}}$ with respect to $y'$.
This kind of direct computation only works for $\GSp_{4}$, and the general case is more subtle. See the proof of Lemma \ref{lemma:affinefibration}.

The third application is the description of $X_{h}$ appearing in \cite{Chan2021a}.
We have the following theorem, which is an analogue of \cite[Equation (7.1)]{Chan2021}.
\begin{theorem}[see Remark \ref{remark:Xhinf}]
\label{theorem:introXhinf}
The variety $X_{h}$ and components $\dot{X}^{m}_{w_r}(b)_{\la}$ of affine Deligne--Lusztig varieties {\cred coincide} at infinite level, i.e., we have
\[
\mathop{\varprojlim}\limits_{h} X_{h} \simeq \la_{b}^{\mathrm{symp,rat}} \simeq \mathop{\varprojlim}\limits_{r>m}\dot{X}^{m}_{w_r}(b)_{\la} \subset \mathop{\varprojlim}\limits_{r>m}\dot{X}^{m}_{w_r}(b) \simeq X_{w}^{(U)}(b).
\]
\end{theorem}

This theorem ensures that $X_h$ at infinite level can be regarded as a connected component of a semi-infinite Deligne--Lusztig variety.
Therefore, it reinterprets the studies of the local Langlands correspondence using $X_{h}$ as a realization of Lusztig's expectation of \cite{Lusztig1979}.
The proof is given by the direct computation based on the parameterization given by Theorem \ref{theorem:introcomparison} and Theorem \ref{theorem:restradlv}. See Subsection \ref{subsection:familyXh} for more {\cred details}.

This paper is organized as follows.
In Section \ref{section:comparison}, 
{\cred we set up the notation to state the main comparison theorem} Theorem \ref{theorem:introcomparison}. 
In Section \ref{section:pfcomparison}, we prove Theorem \ref{theorem:introcomparison} by computing elementary transformations.
In Section \ref{section:conncomp}, we describe a connected component of affine Deligne--Lusztig varieties using $\Lsymp$, which is motivated by the main comparison theorem and \cite{Viehmann2008a}'s result.
Moreover, we prove Theorem \ref{theorem:introstradlv} and Theorem \ref{theorem:introXhinf} by studying the set $\Lsymp$ in detail.

\begin{ack}
The author is deeply grateful to Naoki Imai for his deep encouragement and helpful suggestions. 
The author also thanks Shou Yoshikawa for discussing the solutions of the equation (\ref{eqn:toycase2}).
Moreover, the author wishes to express his gratitude to Charlotte Chan, Alexander B. Ivanov, Masao Oi, Yasuhiro Oki, and Ryosuke Shimada for helpful comments.
Finally, the author would like to express his sincere gratitude to the referee for carefully reviewing the manuscript and providing valuable comments.
The author was supported by JSPS KAKENHI Grant numbers JP19J22795 and JP22J00962.
\end{ack}

\subsection*{Data availability}
Data sharing is not applicable to this article as no data sets were generated or analyzed during the current study.

\subsection*{Conflict of interest}
The author has no {\cred c}onflict of interest directly relevant to the content of this article.

\section{Statement of the comparison theorem}
\label{section:comparison}

\subsection{Basic Notation}
Let $K$ be a non-archimedean local field with residue characteristic $p >0$, and 
$\breve{K}$ the completion of the maximal unramified extension of $K$.
Let $\co_{K}$ (resp.\,$\co$) be the ring of integers of $K$ (resp.\,$\breve{K}$), $\p_{K}$ (resp.\,$\p$) its maximal ideal, and $k=\F_{q}$ (resp.\,${\cred \overline{k}}$) its residue field.
We fix a uniformizer $\varpi$ of $K$. 
{\cred Let $\ord$ denote the normalized valuation on $\breve{K}$,}
and {\cred let} $\sigma\in \mathrm{Aut}(\breve{K}/K)$ {\cred be} the Frobenius morphism.

Let $V_K:= K^{2n}$, \,\,\,$V:= \breve{K}^{2n}$, and we denote the symplectic form on $V_K$ associated with 
\begin{equation}
\label{eqn:Omega}
\Omega:=
\left(\begin{array}{ccc|ccc}
&&&&&1\\
&\hsymb{0}&&&\iddots&\\
&&&1&&\\\hline
&&1&&&\\
&\iddots&&&&\\
1&&&&\hsymb{0}&
\end{array}\right)
\cdot \diag(1,-1,1,-1, \ldots, 1,-1)
. 
\end{equation}
 by $\langle\, , \rangle$. 
 {\cred In this paper, horizontal/vertical lines in block matrices indicate the central divider unless otherwise noted.}

Let $G$ be $\GSp_{2n}$ over $K$ associated with the above symplectic form, i.e.,
$$G := \{g\in \GL_{2n}\mid\exists\lambda(g)\in \mathbb{G}_m, \forall x,y\in V_K, \langle gx,gy\rangle=\lambda(g)\langle x,y\rangle \}.$$
Then we have the exact sequence
\[
\begin{CD}
1 @>>>\Sp_{2n} @>>>\GSp_{2n} @>\lambda>> \mathbb{G}_m @>>>1
\end{CD}
\]
and $\Sp_{2n}$ is a simply connected algebraic group. By \cite[2.a.2]{Pappas2008},
the Kottwitz map {\cred $\kappa_G\colon G(\breve{K}) \rightarrow \pi_1(G)_{\Gamma} \simeq \Z$} 
is {\cred given by} $\ord\circ \lambda$.

We take $[b]$, a basic $\sigma$-conjugacy class of $G(\breve{K})$, and set $k:=\kappa_G([b])\in\mathbb{Z}$.
In the following, we suppose that $k \in \{0,1 \}$ (note that we may assume this condition by multiplying $[b]$ by a scalar {\cred since $\ord \circ \lambda (\varpi^i b) =\ord (\varpi^{2i} b) = 2i + \ord (b)$}).

We can choose a representative $b \in G(\breve{K})$ of $[b]$ satisfying
{\cred
\begin{equation}
\label{eq:lambda}
\lambda(b)= - \varpi^{k} \quad \textup{and} \quad b \in G(K).
\end{equation}
}
We assume (\ref{eq:lambda}) in the {\cred remainder} of this paper. 
For the precise choice of $b$, see Subsection \ref{subsection:repb}.
We write $F \colon V\rightarrow V$ for the structure morphism of {\cred the} isocrystal corresponding to $b$, i.e., $F=b\sigma$.
We also frequently apply $F$ to matrices $M \in M_{2n} (\breve{K})$ {\cred by setting} $F(M) = b\sigma (M).$
Similarly, we also use the notation $F^{-1} = \sigma^{-1} b^{-1}$ and apply it to elements in $V$ or $M_{2n} (\breve{K}).$
 We write $J_b$ for the associated inner form of $G$ over $K$, i.e., $J_b$ is the reductive group whose $R$-valued {\cred points are} given by
$$J_b(R):= \{g\in G(R\otimes_{K}\breve{K})\mid g^{-1}b\sigma(g)=b\}$$
for a finite type algebra $R$ over $K$.
{\cred In the case when $k=1$ (resp.\ $k=0$), $J_b$ is the (unique) non-trivial inner form (resp.\ the trivial inner form) of $\GSp_{2n}$.}

\subsection{Definition of semi-infinite Deligne--Lusztig varieties and {\cred a}ffine Deligne--Lusztig varieties}
\label{subsection:DLset}

For $r\in \mathbb{Z}_{\geq0}$, define
\begin{eqnarray*}
w_{r}:=
\left(\begin{array}{cccc|cccc}
&&&& (-1)^{n+1}\varpi^{r+k}&&&\\
\varpi^{-r}&&\hsymb{0}&&&&&\\
&\ddots&&&&\hsymb{0}&&\\
&&\varpi^{-r}&&&&&\\\hline
&&&&&\varpi^{r+k}&&\\
&&&&&&\ddots&\\
&\hsymb{0}&&&\hsymb{0}&&&\varpi^{r+k}\\
&&&\varpi^{-r}&&&&
\end{array}\right).
\end{eqnarray*}
{\cred Note} that $w_r$ represents the Coxeter element of the Weyl group of $G$. For simplicity, we write $w$ for $w_{0}.$

First, we recall the definition of a semi-infinite Deligne--Lusztig variety which is a direct analogue of a classical Deligne--Lusztig variety in \cite{Deligne1976}.
Let $B$ be {\cred the} intersection of {\cred the} standard Borel subgroup (upper triangular matrices) of $\GL_{2n}$ with $\GSp_{2n}$, and $U$ its unipotent radical.
We define semi-infinite Deligne--Lusztig varieties for $\GSp$ by
\begin{eqnarray*}
X^{(U)}_{w}(b)&:=&\{g\in G(\breve{K})/U(\breve{K})\mid g^{-1}b\sigma(g)\in U(\breve{K})wU(\breve{K})\},\\
X^{(B)}_{w}(b)&:=&\{g\in G(\breve{K})/B(\breve{K})\mid g^{-1}b\sigma(g)\in B(\breve{K})wB(\breve{K})\}.
\end{eqnarray*}
A priori, these sets do not  have scheme structures.

Next, {\cred w}e define affine Deligne--Lusztig varieties of higher Iwahori level as in \cite[Subsection 3.4]{Chan2021}.
For $m \in \Z_{\geq 0}$, we define subgroups $I^{m}_{\GL}, \dot{I}^{m}_{\GL} \subset \GL(\breve{K})$ by
\[
I^{m}_{\GL}
:=
\left(\begin{array}{cc|cc}
\co^{\times}&&&\hsymb{\p}^{m+1}\\
&\ddots&&\\\hline
&&\ddots&\\
\hsymb{\p}^{m}&&&\co^{\times}
\end{array}\right)
,
\dot{I}^{m}_{\GL}
:=
\left(\begin{array}{cc|cc}
1 + \p^{m+1} &&&\hsymb{\p}^{m+1}\\
&\ddots&&\\\hline
&&\ddots&\\
\hsymb{\p}^{m}&&& 1+ \p^{m+1}
\end{array}\right).
\]

We put $I^{m} := I^{m}_{\GL} \cap G (\breve{K})$ and $\dot{I}^{m} := \dot{I}^{m}_{\GL} \cap G (\breve{K})$.
Define affine Deligne--Lusztig varieties as
\begin{eqnarray*}
X^{m}_{w_{r}}(b)(\overline{k})&:=&\{gI^{m} \in G(\breve{K})/I^m \mid g^{-1}b\sigma(g)\in I^{m} w_{r}I^{m} \},\\
\dot{X}^{m}_{w_r}(b)(\overline{k})&:=&\{g\dot{I}^{m} \in G(\breve{K})/\dot{I}^{m} \mid g^{-1}b\sigma(g)\in \dot{I}^{m} w_{r}\dot{I}^{m} \}.
\end{eqnarray*}

{\cred Note} that the above $w_{r}, I^m, \dot{I}^m$ {\cred satisfy} the condition of 
\cite[Theorem 4.9]{Chan2021}
if $r \geq m$.
Therefore, by \cite[Corollary 4.10]{Chan2021}, if $\chara K>0$, (resp. $ \chara K=0$),  the above $X^{m}_{w_{r}}$ and $\dot{X}^{m}_{w_{r}}$ can be regarded as schemes which are locally of finite type (resp.\,locally perfectly of finite type) over $\overline{k}$, which are locally closed subschemes of affine flag varieties $L(G)/L^{+}(I^m)$ and  $L(G)/L^{+}(\dot{I}^m)$.

\subsection{Comparison theorem}
Let
$$\Vsymp:=\{v\in V\mid \langle v,F(v)\rangle=\cdots=\langle v,F^{n-1}(v)\rangle=0,
\langle v,F^{n}(v)\rangle\neq0\}. $$

\begin{lemma}
$v\in \Vsymp$, then $v,F(v),\ldots.F^{2n-1}(v)$ 
{\cred forms a} basis of $V$ over $\breve{K}$.
\end{lemma}
\begin{proof}
{\cred Note} that for $v_{1}, v_{2}\in V,$ we have 
\begin{equation}
\label{eqn:Fsymplectic}
\langle F(v_{1}),F(v_{2})\rangle=\lambda(b)\sigma(\langle v_{1}, v_{2}\rangle).
\end{equation}
Therefore, for $v\in V^{\mathrm{symp}}_b$, we get 
\begin{equation}
\label{eq:FiFj}
\langle F^{i}(v),F^{j}(v)\rangle
\left\{\begin{aligned}
&=0 \wh |i-j|\leq n-1,\\
&\neq0 \wh |i-j|=n.
\end{aligned}\right.
\end{equation}
Suppose that $\sum_{i=0}^{2n-1}c_iF^{i}(v)=0$ for $c_{i} \in \breve{K}$.
Applying $\langle F^{n-1}(v),-\rangle$ on the equation, we get $c_{2n-1}=0$ by (\ref{eq:FiFj}).
Similarly, applying $\langle F^{n-2}(v),-\rangle,  \ldots \langle F^{-n}(v),-\rangle$, we get $c_{2n-2}=\cdots=c_{0}=0$. {\cred This} finishes {\cred the} proof.
\end{proof}

Next, we will define a morphism from $\Vsymp$ to $\GSp_{2n}(\breve{K})$.
\begin{definition}
We put 
 \[
 V_{k}:=\varpi^{k} (b\sigma)^{-1} = \varpi^{k} \sigma^{-1}b^{-1},
 \]
 which is a $K$-linear morphism from $V$ to $V$.
 {\cred Note that $V_k = \varpi^k F^{-1}$.}
 We also put 
$\alpha_{v}:=\langle v,F^{n}(v)\rangle \in \breve{K}^{\times}$ for $v \in \Vsymp$.
We fix an element $v\in \Vsymp$.
We put
\[
G_{1} (v):= 
(-1)^{n+1} \frac{\alpha_{v}}{\sigma^{-1}(\alpha_{v})}(V_{k}(v) - \frac{\langle V_{k}(v), F^{n}(v)\rangle}{\alpha_{v}}v).
\]
We also put
\[
G_{i+1}(v):=  \frac{\alpha_{v}}{\sigma^{-1}(\alpha_{v})} (V_{k}(G_{i}(v))-\frac{\langle V_{k}(G_{i}(v)), F^{n}(v)\rangle}{\alpha_{v}}v)
\]
for $1 \leq i \leq n-2$ inductively.
We define
\begin{eqnarray*}
g_{b,r}(v)&:=&(v,\varpi^{r}F(v),\ldots,\varpi^{(n-2)r}F^{n-2}(v), \varpi^{(n-1)r}F^{n-1}(v),\\
&\quad&\quad\varpi^{r}G_{1}(v),\varpi^{2r}G_{2}(v),\ldots,\varpi^{(n-1)r}G_{n-1}(v), \varpi^{nr}F^{n}(v)).
\end{eqnarray*}
{\cred 
Note that we have $\langle F^{i}(v), G_1 (v) \rangle =0$ for $1\leq i\leq n-2$ since $v\in \Vsymp$ (cf.\ \eqref{eqn:Fsymplectic}).
Also, clearly we have $\langle F^n (v), G_1 (v) \rangle =0$.
Moreover, we have 
\begin{eqnarray*}
\langle F^{n-1} (v), G_1 (v) \rangle &=& (-1)^{n+1} \frac{\alpha_{v}}{\sigma^{-1}(\alpha_{v})} \langle F^{n-1} (v), V_k (v) \rangle  \\
&=& (-1)^{n+1} \frac{\alpha_{v}}{\sigma^{-1}(\alpha_{v})} (-\varpi^{-k}) \sigma^{-1} (\varpi^k \langle F^{n} (v), v \rangle) = (-1)^{n+1} \alpha_v,
\end{eqnarray*}
where we use (\ref{eqn:Fsymplectic}) in the second equality.
}
{\cred Inductively, we can show that 
\[
\langle F^i (v), G_{j} (v) \rangle  =
\begin{cases}
0 & \textup{if } 0\leq i \leq n, 1 \leq j \leq n-1, \textup{ and }i+j \neq n, \\
(-1)^{i}\alpha_v  & \textup{if } 1 \leq i \leq n-1 \textup{ and }i+j =n.
\end{cases}
\]
Moreover, we have $\langle G_i (v), G_j (v) \rangle  = 0$ for $1 \leq i,j \leq n-1$ by (\ref{eqn:Fsymplectic}) and $v\in \Vsymp$.
Therefore,} we have $g_{b,r}(v) \in \GSp_{2n}(\breve{K})$ with 
$\lambda(g_{b,r}(v)) = \varpi^{nr}\alpha_{v}$.
\end{definition}

\begin{remark}
\label{remark:hb}
There are other candidates {\cred for} $g_{b,r}$. 
We take $v \in \Vsymp$, and
we put $H_{n}(v) := F^{n}(v)$.
Inductively, we define
\[
H_{i} (v) :=
\varpi^{-k} \frac{\alpha_{v}}{\sigma(\alpha_{v})}(F(H_{i+1}(v))
-\frac{\langle v, F(H_{i+1}(v)) \rangle}{\alpha_{v}} F^{n}(v)
)
\]
for $1\leq i \leq n-1$.
We put 
\begin{align*}
h_{b} (v):=& (v, \ldots, F^{n-1}(v), H_{1} (v), \ldots, H_{n} (v)) \\
&\cdot 
\diag(1, \varpi^{\lceil \frac{-k}{2} \rceil}, \ldots, \varpi^{\lceil \frac{-k(i-1)}{2} \rceil}, \ldots, \varpi^{\lceil \frac{-k(n-1)}{2} \rceil}, \varpi^{\lceil \frac{-kn}{2} \rceil + \lfloor \frac{-k(n-1)}{2} \rfloor}, \ldots, \varpi^{\lceil \frac{-kn}{2} \rceil}).
\end{align*}
Then $g'_{b} (v) \in \GSp_{2n} (\breve{K})$ as well. We use this construction in Subsection \ref{subsection:familyXh}.
\end{remark}

\begin{lemma}
\label{lemma:order}
We have
\begin{equation}
\label{eqn:gbrdlv}
F(g_{b,r}(v))=g_{b,r}(v)w_{r}A_{b,r} \textup{ with } A_{b,r}=
\left(\begin{array}{cccc|cccc}
1&&&&{b}_{1}&\cdots&b_{n-1}&a_{n}\\
&\ddots&&&&&&a_{n-1}\\
&&\ddots&&&&&\vdots\\
&&&1&&&&a_{1}\\\hline
&&&&\frac{\sigma(\alpha_{v})}{\alpha_{v}}&&&\\
&&&&&\ddots&&\\
&&&&&&\ddots&\\
&&&&&&&\frac{\sigma(\alpha_{v})}{\alpha_{v}}
\end{array}\right),
\end{equation}
where ${a}_{i}, {b}_{i}$ are {\cred elements} of $\breve{K}$ such that
{\cred
\[
\ord\, {a}_{i} \geq ir+\frac{k}{2}i, \quad \ord\, {b}_{i} \geq ir+\frac{k}{2}i.
\]
}
Note that $A_{b,r}\in G(\breve{K})$, and so $a_{i}= (-1)^{n+i} b_{i}$ for $i=1,\ldots n-1$.
In particular,
we have
$A_{b,r} \in I^{m}$
if $ \lceil r + \frac{k}{2} \rceil (= r+k) \geq m+1$.
\end{lemma}

\begin{proof}

We have
\begin{eqnarray*}
F(g_{b,r}(v)) &=&(F(v), \varpi^{r}F^{2}(v)
\ldots,
\varpi^{(n-2)r} F^{n-1}(v),
\varpi^{(n-1)r} F^{n}(v), \\
&\quad&\quad\varpi^{r} F( G_{1}(v)),
\varpi^{2r} F(G_{2}(v)),
\ldots,\varpi^{(n-1)r}F(G_{n-1}(v)), \varpi^{nr} F^{n+1}(v))
\end{eqnarray*}
and
\begin{eqnarray*}
g_{b,r}(v)w_{r} &=& 
(F(v), \varpi^{r}F^{2}(v), \ldots, \varpi^{(n-2)r}F^{n-1}(v), \varpi^{(n-1)r}F^{n}(v),  \\
&\quad&\quad (-1)^{n+1} \varpi^{r+k}v, \varpi^{2r+k}G_{1}(v), \ldots, \varpi^{(n-1)r+k} G_{n-2}(v), \varpi^{nr+k}G_{n-1} (v)).
\end{eqnarray*}
Therefore, for the $1$st, $\ldots$, $n$-th column, the equality (\ref{eqn:gbrdlv}) is obvious.

Next, we compute the $(n+1), \ldots, (2n-1)$-th column of $A_{b,r}.$
By definition, we have

\begin{equation}
\label{eqn:Grep1}
F(G_{1}) (v) = 
(-1)^{n+1} \frac{\sigma(\alpha_{v})}{\alpha_{v}} (\varpi^{k} v - \frac{
\sigma(\langle V_{k}(v), F^{n}(v) \rangle)}{\sigma(\alpha_{v})} F(v)).
\end{equation}
Similarly, we have
\begin{equation}
\label{eqn:Grep2}
F(G_{i+1}(v))
=
\frac{\sigma(\alpha_{v})}{\alpha_{v}} (\varpi^{k} G_{i}(v) - \frac{\sigma(\langle V_{k}(G_{i}(v)), F^{n}(v) \rangle)}{\sigma(\alpha_{v})} F(v)).
\end{equation}
These formulas correspond to  $(n+1), \ldots, 2n-1$-th columns of the equality (\ref{eqn:gbrdlv}). Note that we {\cred have not yet estimated} the orders of $b_{i}$.

Finally, we compute the $2n$-th column of $A_{b,r}$.
We can write
\[
F^{n+1}(v) = \sum_{i=0}^{n} c_{i} F^{i} ({\cred v}) + \sum_{j=1}^{n-1} d_{j} G_{j}(v).
\]
Applying $\langle F^{i}(v),-\rangle$ ($i=2,3,\ldots,n$) to this equation, we have $d_{1}=\cdots=d_{n-2}=c_{0}=0.$

Moreover, applying $\langle F(v),-\rangle$, we have 
$$
d_{n-1}= \frac{\sigma(\alpha_{v})}{\alpha_{v}}\varpi^{k}.
$$
Thus we have verified the form of {\cred the} $2n$-th column and {\cred that} $a_{i}=\varpi^{ir}c_{n+1-i}$ ($i=1,\ldots,n$).

Finally, we should estimate the order of $a_{i}$. Now we have the equation
\begin{equation}
\label{eqn:FGformula}
F^{n+1}(v) = \sum_{i=1}^{n} c_{i} F^{i} (v) + d_{n-1} G_{n-1} (v).
\end{equation}
Note that $G_{n-1}(v)$ can be written by a linear combination of $v, V(v), \ldots V^{n-1} (v)$ by definition.
Applying $F^{n-1}$ to {\cred the} equation (\ref{eqn:FGformula}), we have
\[
F^{2n}(v)- \sum_{i=1}^{n} \sigma^{n-1}(c_{i}) F^{i+n-1}(v)- \sigma^{n-1} (d_{n-1}) F^{n-1} (G_{n-1} (v)),
\]
where $F^{n-1}(G_{n-1} (v))$ can
be written by a linear combination of $F^{n-1} (v), \ldots, v$.
Since the slope of $b$ is $\frac{k}{2}$, by \cite[Lemma 5.2.4]{Kedlaya2005}, we get 
\[
\ord c_{n+1-i} \geq \frac{k}{2} i. 
\]
Therefore, we have
\[
\ord a_{i} = ir + \ord c_{n+1-i} \geq ir + \frac{k}{2} i.
\]
\end{proof}

\begin{remark}
\label{remark:lattice}
Let $v \in \Vsymp$.
We define the non-commutative ring $\mathcal{D}_{k}$ by
\[
\mathcal{D}_{k} := \co [F,V_{k}] / (FV_{k}=V_{k}F= \varpi^{k}, aV_{k} = V_{k} \sigma(a), Fa = \sigma (a) F)_{a\in \co}.
\]
By the proof of Lemma \ref{lemma:order} (more precisely, equations (\ref{eqn:Grep1}) and (\ref{eqn:Grep2})) 
\[
\la := \co v \oplus  \cdots \oplus \co F^{n-1}(v) \oplus \co G_{1} (v) \oplus \cdots 
\oplus \co G_{n-1} (v) \oplus  \co F^{n} (v)
\]
is contained in $\mathcal{D}_{k} v$.
Note that $\la$ is self-dual up to {\cred a} constant.
Since we have
\[
\varpi^{k} \la \subset F \la \subset \la 
\]
by Lemma \ref{lemma:order} again, we have
{\cred $\mathcal{D}_{k} v = \la.$}
\end{remark}

{\cred Now we can state the main comparison theorem, which is an analogue of \cite[Theorem 6.5]{Chan2021}.}
\begin{theorem}
\label{theorem:comparison}
Assume 
$ r+k  \geq m+1${\cred .}
Then there exist the following two $J_b(K)$-equivariant commutative diagrams with bijective horizontal arrows.
\[
\xymatrix{
\{v\in \Vsymp\mid \alpha_{v}:=\langle v,F^{n}(v)\rangle \in K^{\times}\} \ar[r]^-{g_{b,0}}_-{\sim} \ar[d] & X_{w}^{(U)}(b) \ar[d]\\
\Vsymp/\breve{K}^{\times} \ar[r]^{g_{b,0}}_{\sim} &  X_{w}^{(B)}(b)
}
\]
\[
\xymatrix{
\{v\in \Vsymp\mid\frac{\sigma(\alpha_{v})}{\alpha_{v}}\equiv1 \mod \p^{m+1}\}/\dot{\sim}_{b,m,r} \ar[r]^-{g_{b,r}}_-{\sim} \ar[d] & \dot{X}_{w_{r}}^{m}(b)(\overline{k}) \ar[d] \\
\Vsymp/\sim_{b,m,r}  \ar[r]^-{g_{b,r}}_-{\sim} & X_{w_{r}}^{m}(b)(\overline{k})
}
\]
Here, we {\cred define}
$$
v_{1}\sim_{b,m,r} (resp.\,\dot{\sim}_{b,m,r})\, v_{2}\overset{\mathrm{def}}{\iff} g_{b,r}(v_{1})I^{m} (resp.\,\dot{I}^m) =g_{b,r}(v_{2})I^{m} (resp.\,\dot{I}^m).
$$
for $v_{1}, v_{2} \in \Vsymp$.
\end{theorem}
The proof of this theorem is given in the next section.

\section{Proof of comparison theorem}
\label{section:pfcomparison}

\subsection{Proof for the semi-infinite case}
\label{subsection:pfsemiinf}
In this subsection, we verify the existence of the first diagram in Theorem \ref{theorem:comparison}.
By Lemma \ref{lemma:order}, the well-definedness and the $J_{b}(K)$-{\cred equivariance} of {\cred the} horizontal maps are obvious.
Note that for any $c \in\breve{K}$, {\cred there exists a diagonal matrix $A\in G(\breve{K})$ such that $g_{b,r}(cv)=g_{b,r}(v)A$.}
Moreover, the injectivity of the maps is obvious too. 
In this section, we will prove the surjectivity in the first diagram. 
 
To begin with, we consider the bottom map of the first diagram. Take any element $gB(\breve{K})\in X_{w}^{(B)}(b)$. Then $g^{-1}F(g)\in B(\breve{K})wB(\breve{K}).$
Replacing $g$ with a suitable representative, we may assume $g^{-1}F(g)\in wB(\breve{K}).$ 
Hence we have $g=F(g)C$ with
\begin{equation}
\label{C}
C\in
\left(\begin{array}{cccccc|ccccccc}
\breve{K}&\breve{K}^{\times}&\breve{K}&\cdots&\cdots&\breve{K}&\breve{K}&\cdots&\cdots&\breve{K}&\breve{K}&\breve{K}\\
\vdots&0&\breve{K}^{\times}&\ddots&\ddots&\vdots&\vdots&\vdots&\vdots&\vdots&\vdots&\vdots\\
\vdots&\vdots&0&\ddots&\ddots&\vdots&\vdots&\vdots&\vdots&\vdots&\vdots&\vdots\\
\vdots&\vdots&\vdots&\ddots&\ddots&\breve{K}&\vdots&\vdots&\vdots&\vdots&\vdots&\vdots\\
\vdots&\vdots&\vdots&\vdots&\ddots&\breve{K}^{\times}&\vdots&\vdots&\vdots&\vdots&\vdots&\breve{K}\\
\breve{K}&\vdots&\vdots&\vdots&\vdots&0&\vdots&\vdots&\vdots&\vdots&\vdots&\breve{K}^{\times}\\\hline
\breve{K}^{\times}&\vdots&\vdots&\vdots&\vdots&\vdots&\breve{K}&\vdots&\vdots&\vdots&\vdots&0\\
0&\vdots&\vdots&\vdots&\vdots&\vdots&\breve{K}^{\times}&\ddots&\vdots&\vdots&\vdots&\vdots\\
\vdots&\vdots&\vdots&\vdots&\vdots&\vdots&0&\ddots&\ddots&\vdots&\vdots&\vdots\\
\vdots&\vdots&\vdots&\vdots&\vdots&\vdots&\vdots&\ddots&\ddots&\breve{K}&\vdots&\vdots\\
\vdots&\vdots&\vdots&\vdots&\vdots&\vdots&\vdots&\ddots&\ddots&\breve{K}^{\times}&\breve{K}&\vdots\\
0&0&0&\cdots&\cdots&0&0&\cdots&\cdots&0&\breve{K}^{\times}&0
\end{array}\right).
\end{equation}

Note that for $P\in G(\breve{K}),$ we have
\begin{equation}
\label{eqn:sigmaconj}
gP=F(g)CP=F(gP)(\sigma(P)^{-1}CP).
\end{equation}

\begin{claim}
\label{claim:C'}
Suppose that there exists a $P \in B(\breve{K})$ such that
$$\sigma(P)^{-1}CP=C'\quad\text{with }$$
$$
C'=
\left(\begin{array}{cccc|cccc}
\ast&\ast&&\Lsymb{0}&\ast&\cdots&\ast&0\\
0&&\ddots&&&&\ast&\vdots\\
\vdots&&&\ast&&\hsymb{0}&\vdots&0\\
0&&&&&&\ast&\ast\\\hline
\ast&&&&&&&\\
&&&&\ast&&\hsymb{0}&\\
&\hsymb{0}&&&&\ddots&&\\
&&&&&&\ast&\\
\end{array}\right).
$$
Then the bottom map of the first diagram in Theorem \ref{theorem:comparison} is surjective.
\end{claim}

\begin{proof}
We set $(1,2), \ldots, (n-1,n), (n,2n)$-th entries of $C'$ as $u_{1}, \ldots, u_{n-1}, u_{n}$.
Note that $u_{1}, \ldots, u_{n} \in \breve{K}^{\times}$ automatically.
We can take a diagonal matrix $Q = \diag (q_{1}, \ldots, q_{2n}) \in G (\breve{K})$ satisfying the following:
\begin{align}
\begin{aligned}
\label{eqn:qs}
&q_{1}=1,\\
&q_{2}=u_{1}^{-1},\\
&q_{3}=(u_{2}\sigma(q_{2})^{-1})^{-1},\\
&\vdots\\
&q_{n}=(u_{n-1}\sigma(q_{n-1})^{-1})^{-1},\\
&q_{2n}=(u_{n}\sigma(q_{n})^{-1})^{-1}.
\end{aligned}
\end{align}
Then by replacing $P$ by $PQ$, we may assume 
$$
C'=
\left(\begin{array}{cccc|cccc}
\ast&1&&\Lsymb{0}&\ast&\cdots&\ast&0\\
0&&\ddots&&&&\ast&\vdots\\
\vdots&&&1&&\hsymb{0}&\vdots&0\\
0&&&&&&\ast&1\\\hline
\ast&&&&&&&\\
&&&&\ast&&\hsymb{0}&\\
&\hsymb{0}&&&&\ddots&&\\
&&&&&&\ast&\\
\end{array}\right).
$$
On the other hand, $g':= gP:=(g'_{1},\ldots,g'_{2n})$ satisfies $g'=F(g')C'$. 
By (\ref{eqn:sigmaconj}), we have
\begin{eqnarray*}
g'_{2}&=&F(g'_{1}),\\
&\vdots&\\
g'_{n}&=&F^{n-1}(g'_{1}),\\
g'_{2n}&=& F^{n}(g'_{1}).
\end{eqnarray*}
Moreover, we can show that $F(g'_{n+1})$ can be written by a $\breve{K}$-linear combination of $F(g'_{1})$ and $g'_{1}$.
Similarly, $F(g'_{n+i+1})$ can be written by a $\breve{K}$-linear combination of $F(g'_{1})$ and $g'_{n+i}$ for $i=1, \ldots, n-2$.
Since $g' \in \GSp (V)$, we can show that $g'_{1}\in \Vsymp$ and $g'=g_{b,0}(g'_{1}).$
\end{proof}

We will construct $P$ satisfying the assumption in Claim \ref{claim:C'} as a product of the following kinds of elementary matrices of $\GSp_{2n}$.
\begin{itemize}
\item
$
P=1_{2n}+ce_{i,j}+ (-1)^{i-j+1} ce_{2n+1-j, 2n+1-i}
$\\
Here, we set $1\leq i< j\leq2n$, and $j\neq 2n+1-i $. {\cred In particular}, we have $P \in B(\breve{K}).$
For these $P$, the $\sigma$-conjugation $\ast\to \sigma(P)^{-1}\ast P$ act by adding $c$ times the $i$-th column to the $j$-th column, and adding $\pm c$ times the $(2n+1-j)$-th column to the $(2n+1)$-th column, followed by adding $-\sigma(c)$ times $j$-th row to $i$-th row, and adding $\mp \sigma(c)$ times $(2n+1-i)$-th row to $(2n+1-j)$-th row. 
We denote these transformations by $\mathrm{col}_{i\to j}$, $\mathrm{col}_{2n+1-j\to 2n+1-i}$, $\mathrm{row}_{j\to i}$, and $\mathrm{row}_{2n+1-i\to 2n+1-j}.$ Moreover, by the abuse of notation, we use one of them for representing all the above transformations.
\item $
P=1_{2n}+ce_{i,2n+1-i} 
$\\
Here, we set $1\leq i \leq n$, {\cred and hence} we have $ P \in B(\breve{K}).$
The $\sigma$-conjugation by these $P$ act by adding $c$ times the $i$-th column to the $(2n+1-i)$-th column, followed by adding $-\sigma(c)$ times the $(2n+1-i)$-th row to the $i$-th row.
As above, we call these transformations \c{i}{2n+1-i}, and \r{2n+1-i}{i}. Moreover, we use one of them for representing the above $\sigma$-conjugation.
\end{itemize}

We will eliminate entries of $C$ successively {\cred using the} above $P$. While these modifications, we should check that they do not restore already modified entries and that they preserve the form of $C$ as in (\ref{C}). However, in fact, we will use only  
$$
P\in B(\breve{K})\cap wB(\breve{K})w^{-1}=
\left(\begin{array}{cccc|cccc}
\breve{K}^{\times}&0&\cdots&0&\breve{K}&\cdots&\breve{K}&0\\
&\ddots&&&&&&\breve{K}\\
&&\ddots&\hsymb{\breve{K}}&&\hsymb{\breve{K}}&&\vdots\\
&&&\ddots&&&&\breve{K}\\\hline
&&&&\ddots&\hsymb{\breve{K}}&&0\\
&&&&&\ddots&&\vdots\\
&\hsymb{0}&&&&&\ddots&0\\
&&&&&&&\breve{K}^{\times}
\end{array}\right)
$$
in the modification,
so the form of $C$ as in (\ref{C}) is preserved successively. 
Therefore it suffices to check that each modification does not restore already modified entries.

In the following, 
we will denote the $(i,j)$-th entry of matrices under transformations by $(i,j)$.
First, we will modify the $3$rd, $\ldots$, $n$-th columns of {\cred the} above $C$, as {\cred follows}.  
{\cred
\[
\begin{array}{l l}
\text{$3$rd column:} & \text{We use \c{2}{3}  to eliminate } (1,3),\\
\text{$4$th column:} & \text{We use \c{2}{4},\,\c{3}{4} to eliminate } (1,4),(2,4),\\
                             & \vdots \\
\text{$n$-th column:}& \text{We use \c{2}{n},\ldots,\c{n-1}{n}  to eliminate } (1,n),\ldots,(n-2,n).
\end{array}
\]
}
Here, for the $i$-th column, we should check that \c{j}{i} ($j=2, \ldots, i-1$) do not restore already modified entries. \c{j}{i} is together with \c{2n+1-i}{2n+1-j}, \r{i}{j}, and \r{2n+1-j}{2n+1-i}. 
Since the $j$-th column is already modified, \c{j}{i} affects only $(j-1,i)$.
The transformations \c{2n+1-i}{2n+1-j}, \r{2n+1-j}{2n+1-i} do not affect {\cred the entries we have set to zero} obviously. By the form of the $i$-th row, \r{i}{j} affects only entries on the $1$st and the $l$-th column ($l \geq i+1$), so it does not restore {\cred the entries we have set to zero} too.

Next, we will eliminate the $2n$-th column {\cred except for the} $(n, 2n)$-th entry by using \c{2}{2n}, $\ldots$, \c{n}{2n}.  
Here, we use \c{j}{2n} for $2 \leq j \leq n$, which are together with \c{1}{2n+1-j}, \r{2n}{j}, and \r{2n+1-j}{1}. As before, \c{j}{2n} affects only $(j-1,2n)$. Clearly, \c{1}{2n+1-j} is irrelevant to {\cred the entries we have set to zero}, and \r{2n}{j} only affect $(j,2n-1)$. For \r{2n+1-j}{1}, {\cred in the $(2n+1-j)$-th row,} {\cred the entries not yet set to zero} lie on $(2n-j)$, $\ldots,$ $(2n-1)$-th column (resp.\,the $1$st, $(n+1), \ldots, (2n-1)$-th column) if $j<n$ (resp.\,if $j=n$). {\cred In particular,} they are in the $1, (n+1), \ldots, (2n-1)$-th columns, so \r{2n+1-j}{1} does not affect {\cred the entries we have set to zero}.

Now we have modified the 3rd, \ldots, $n$-th, $2n$-th columns.
{\cred Note} that $(i,j)$ $(n+1\leq i\leq j\leq 2n-1)$ are already modified now. 
Indeed, we have $C \in G(\breve{K})$, so we can use
$$
\left\{
\begin{aligned}
&\langle \text{the } 2n\text{-th column vector}, \text{the } j\text{-th column vector} \rangle=0 \wh i=n+1,\\
&\langle \text{the } (2n+2-i)\text{-th column vector}, \text{the } j\text{-th column vector} \rangle=0 \ow,
\end{aligned}
\right.
$$
to show that the $(i,j)$-th entries are $0$.

Finally, we will eliminate the $(n-1)\times (n-1)$ submatrix of $C$ lying on the $2, \ldots, n$-th rows and $1, (n+1) \ldots, (2n-2)$-th columns.  It is enough to eliminate the {\cred upper-left} half triangular entries of this submatrix, because we have 
\[
\langle i\text{-th row vector}, j\text{-th row vector}\rangle=0
\]
for ($2\leq i , j \leq n$) since  $G(\breve{K})$ is {\cred closed under transpose}.

We will {\cred successively} modify the 1st, $(n+1)$-th, \ldots, $(2n-2)$-th columns as {\cred follows}.
{\cred
\[
\begin{array}{l l}
\text{$1$st column:} & \text{We use \r{n+1}{2}, \ldots, \r{n+1}{n} to eliminate } (2,1), \ldots, (n,1).\\
\text{$n+1$-th column:} & \text{We use \r{n+2}{2}, \ldots, \r{n+2}{n-1} to eliminate } (2, n+1), \ldots, (n-1, n+1).\\
                             & \vdots \\
\text{$2n-2$-th column:} & \text{We use \r{2n-1}{2} to eliminate } (2,2n-2).
\end{array}
\]
}
As before, we should check that these transformations do not affect {\cred the entries we have set to zero}.
For the 1st column, \r{n+1}{j} is together with \r{2n+1-j}{n}, \c{j}{n+1}, and \c{n}{2n+1-j} ($j=2, \ldots n$).
Since the $(n+1)$-th row is already modified, \r{n+1}{j} only affects $(1,j)$. Similarly, \r{2n+1-j}{n} affects only $(n,2n-j)$.
Therefore, these operations do not affect {\cred the entries we have set to zero}.

For the $(n+i)$-th column ($i=1, \ldots, n-2$), we use \r{n+i+1}{j} ($2\leq j \leq n-i$),
which is together with \r{2n+1-j}{n-i}, \c{j}{n+i+1} and \c{n-i}{2n+1-j} (resp.\, \c{j}{n+i+1}) if $2 \leq j < n-i$ (resp.\,$j=n-i$).
Since the $(2n+1-j)$-th row {\cred is zero except at} the $(2n+1-j, 2n+1-j-1)$-th entry, \r{2n+1-j}{n-i} affects only $(2n+1-j-1)$-th column. Since $(2n+1-j-1)> n+i$ if $j <n-i$, the transformation \r{2n+1-j}{n-i} does not affect $n+i$-th column.
Clearly, \c{j}{n+i+1} does not affect {\cred the entries we have set to zero}. Moreover, \c{n-i}{2n+1-j} does not affect {\cred such entries too} since
$2n> 2n+1-j > n+i+1$ if $j<n-i$. 

Now we complete the procedure, i.e., we have verified
that there exists $P$ satisfying the assumption of Claim \ref{claim:C'}.

Next, we consider the upper map of the first diagram in Theorem \ref{theorem:comparison}. Take any element $gU(\breve{K}) \in X_{w}^{(U)}(b),$ then the same arguments show that there exists $v\in \Vsymp$ such that $g_{b,0}(v)U(\breve{K})=gU(\breve{K}).$ Thus we get 
\[
g_{b,0}(v)^{-1}F(g_{b,0}(v))\in U(\breve{K})wU(\breve{K}).
\] 
Hence we get $\lambda(g_{b,0}(v)^{-1}F(g_{b,0}(v)))=\lambda(w),$ and by Lemma \ref{lemma:order}, we have $\sigma(\alpha)=\alpha,$ i.e., $\alpha\in K^{\times}.$ This {\cred completes the} proof for first diagram.

\subsection{Proof for the affine case}
In this subsection, we prove the second diagram in Theorem \ref{theorem:comparison}.
As before, to begin with, we will prove the surjectivity of the bottom map. Take any element $gI^{m}\in X_{w_{r}}^{m}(b) (\overline{k})$, then we get $g^{-1}F(g)\in I^{m}w_{r}I^{m},$ and we can choose a representative so that we have $g^{-1}F(g)\in w_{r}I^{m}.$ Therefore, 
we have $g= F(g) C$ with
\begin{equation}
\label{C2}
C\in 
\left(
\varpi^{-r-k}
\left(\begin{array}{c}
\p^{m+1}\\
\vdots\\
\vdots\\
\p^{m+1}\\\hline
\co^{\times}\\
\p^{m}\\
\vdots\\
\p^{m}
\end{array}\right)
\varpi^{r}
\left(\begin{array}{ccc}
\co^{\times}&&\hsymb{\p}^{m+1}\\
&\ddots\\
&&\co^{\times}\\
&&\\\hline
&&\\
&\hsymb{\p}^{m}&\\
&&\\
&&
\end{array}\right)
\left|\varpi^{-r-k}
\middle(\begin{array}{ccc}
&&\\[6pt]
&&\\[6pt]
&\hsymb{\p}^{m+1}&\\
&&\\\hline
&&\\
\co^{\times}&&\\
&\ddots&\\
\hsymb{\p}^{m}&&\co^{\times}
\end{array}\right)
\varpi^{r}
\left(\begin{array}{c}
\p^{m+1}\\[6pt]
\vdots\\[6pt]
\p^{m+1}\\
\co^{\times}\\\hline
\p^{m}\\
\vdots\\
\vdots\\
\p^{m}
\end{array}\right)
\right).
\end{equation}

\begin{claim}
\label{claim:affineC'}
Suppose that there exists $P \in I^{m}$ such that
$$
\sigma(P)^{-1}CP=C'\quad\text{with }
$$
$$
C'\in
\left(\begin{array}{cccc|cccc}
\ast&\varpi^{r}\co^{\times}&&\Lsymb{0}&\ast&\cdots&\ast&0\\
0&&\ddots&&&&\ast&\vdots\\
\vdots&&&\varpi^{r}\co^{\times}&&\hsymb{0}&\vdots&0\\
0&&&&&&\ast&\varpi^{r}\co^{\times}\\\hline
\ast&&&&&&&\\
&&&&\ast&&\hsymb{0}&\\
&\hsymb{0}&&&&\ddots&&\\
&&&&&&\ast&
\end{array}\right).
$$
Then the bottom map of the second diagram in Theorem \ref{theorem:comparison} is surjective.
\end{claim}

\begin{proof}
It can be proved by the same argument as in Claim \ref{claim:C'}.
We set $C'$'s the $(1,2), \ldots, (n-1,n), (n,2n)$-th entries as $\varpi^{r} u_1, \ldots, \varpi^{r} u_{n-1}, \varpi^{r} u_{n},$ where $u_{i}\in\co^{\times}$. 
Then we can take a diagonal matrix $Q \in I^{m}$ satisfying the equations \ref{eqn:qs}.
After replacing $P$ with $PQ$, we can show that $g':= gP = (g_{1}' \ldots, g_{2n}')$ satisfies
$g' = g_{b,r} (g_{1}')$ as in the proof of Claim \ref{claim:C'}.
\end{proof}

We will construct $P$ satisfying the assumption in Claim \ref{claim:affineC'} as a product of the following {\cred four} kinds of matrices in $I^{m}$.
\begin{itemize}
\item $
P=1_{2n}+ce_{i,j}+(-1)^{j-i+1}ce_{2n+1-j, 2n+1-i},
$
{\cred where $1\leq i<j\leq2n, j\neq 2n+1-i,$ and $c\in \p^{m}$.}\\
\item $
P=1_{2n}+ce_{i,j}+(-1)^{j-i+1}ce_{2n+1-j, 2n+1-i},
$
{\cred where $1\leq j<i\leq2n, j\neq 2n+1-i,$ and $c\in \p^{m+1}$.}
\end{itemize}
{\cred
Each such $P$ induces a $\sigma$-conjugation that is a simultaneous two-column addition transformation, followed by a simultaneous two-row addition transformation.
}
As before, we write \c{i}{j}, \c{2n+1-j}{2n+1-i}, \r{j}{i}, and \r{2n+1-i}{2n+1-j} for these elementary transformations and $\sigma$-conjugation.
\begin{itemize}
\item $
P=1_{2n}+ce_{i,2n+1-i},
$
{\cred where $1\leq i\leq n$ and $c\in \p^{m}$.}
\item $
P=1_{2n}+ce_{i,2n+1-i},
$
{\cred where $n+1\leq i\leq2n$ and $c\in \p^{m+1}$.}
\end{itemize}

Each {\cred such $P$ induces a} $\sigma$-conjugation that is a column addition transformation followed by a row addition transformation, and we write \c{i}{2n+1-i}, and \r{2n+1-i}{i} for these elementary operations and $\sigma$-conjugation.

We will eliminate the entries of $C$ by $\sigma$-conjugation {\cred using the matrices above}. 
We divide the procedure of elimination into two steps.

\subsection*{{\cred Step 1:} Eliminate lower-left entries of $C$}\mbox{}\par
In this step, we modify the following entries. 
\begin{eqnarray}
\label{eqn:step1}
&&(n+2,1),(n+3,1), \ldots, (2n,1),\\
\label{eqn:step2}
&&(n+3, n+1), \ldots (2n, n+1), (n+4, n+2), \ldots (2n, n+2), \ldots, (2n,2n-2), \\
\label{eqn:step3}
&&(2n,2), \ldots, (n+2,2), (2n-1,3), \dots, (n+2,3), \ldots, (n+2,n), \\
\label{eqn:step4}
&&(2n,3), (2n,4), \ldots, (2n, n), \ldots, (n+3,n).
\end{eqnarray}
For the above entries $(i,j)$, we put
\[
O_{i,j}  (C):= 
\begin{cases}
\ord (i,j) - (-r-k) & \textup{ if }  j \in \{1, n+1, \ldots, 2n-1 \}, \\
\ord (i,j) - r & \textup{ if } j \in \{ 2, \ldots, n, 2n\}.\\
\end{cases}
\]
We also put
\[
\gamma (C) := \min_{(i,j) \in (\ref{eqn:step1}),(\ref{eqn:step2}),(\ref{eqn:step3}),(\ref{eqn:step4})} O_{i,j}(C).
\]
We will modify the above entries so that we raise $\gamma$.
We will construct such a modification as the composition of $\sigma$-conjugation corresponding to the elementary transformation.
In this procedure, we need to check the following.
\begin{itemize}
\item{(X)}
Each $\sigma$-conjugation preserves the form as in (\ref{C2}).
\item{(Y)}
Each $\sigma$-conjugation does not restore already modified entries mod $\p^{\gamma_{i,j}(C){\cred +1}}$.
\item{(Z)}
After each $\sigma$-conjugation $C \mapsto \sigma (P)^{-1} C \sigma(P)$ , $\gamma (\sigma (P)^{-1}C \sigma(P)) \geq \gamma (C)$. 
\end{itemize}

Here, we put
\[
\gamma_{i,j} (C)=
\begin{cases}
\gamma(C) + (-r-k) & \textup{ if }  j \in \{1, n+1, \ldots, 2n-1 \},\\
\gamma(C) + r & \textup{ if } j \in \{ 2, \ldots, n, 2n\}.\\
\end{cases}
\]
Then one can show that the composition of these transformations increases $\gamma$.
Since we only use the elementary matrices $P \in I^{m} \cap w_{r} I^{m} w_{r}^{-1}$, the condition (X) is automatic. Therefore, we only check (Y) and (Z) in the following.

First, we will modify entries as in $(\ref{eqn:step1})$ as {\cred follows}.
{\cred
\[
\begin{array}{l}
\text{We use \r{n+1}{n+2} to eliminate } (n+2,1).\\
\text{We use \r{n+1}{n+3} to eliminate } (n+3,1).\\
\quad\quad    \vdots \\
\text{We use \r{n+1}{2n} to eliminate } (2n,1).
\end{array}
\]
}

Here, to eliminate the $(i,1)$-th entry ($i=n+2, \ldots 2n$), we use \r{n+1}{i}, which is together with \r{2n+1-i}{n}, \c{i}{n+1}, and \c{n}{2n+1-i}.
For $n+2 \leq i \leq 2n-1$, these transformations clearly satisfy (Y) and (Z).
Here, to prove (Z), we use $r- (-r-k) \geq 0$.
Moreover, \r{n+1}{2n}, \r{1}{n} and \c{2n}{n+1} clearly satisfy (Y) and (Z).
On the other hand, 
\c{n}{1} affects $(n+2, 1), \ldots, (2n,1)$, but since $r- (-r-k) >0$ by the assumption, it satisfies (Y) and (Z).

Next, we will modify entries as in $(\ref{eqn:step2})$ as {\cred follows}.
{\cred
\[
\begin{array}{l l}
\text{$n+1$-th column:} & \text{We use \r{n+2}{n+3}, $\ldots$, \r{n+2}{2n} to eliminate } (n+3,n+1), \ldots, (2n, n+1).\\
\text{$n+2$-th column:} & \text{We use \r{n+3}{n+4}, $\ldots,$ \r{n+3}{2n} to eliminate } (n+4,n+2), \ldots (2n, n+2).\\
                             & \vdots \\
\text{$2n-2$-th column:} & \text{We use \r{2n-1}{2n} to eliminate } (2n,2n-2).
\end{array}
\]}
Here, to eliminate the $(i,j)$-th entry $(j=n+1, \ldots, 2n-2, i= j+2, $\ldots$, 2n)$,
we use \r{j+1}{i}, which is together with \r{2n+1-i}{2n-j}, \c{i}{j+1}, and \c{2n-j}{2n+1-i}.
Clearly, \r{2n+1-i}{2n-j} satisfies (Y), and (Z).
Moreover, \c{2n-j}{2n+1-i} clearly satisfies (Y), and (Z) except for $i=2n$. 
Even for $i=2n$, \c{2n-j}{1} satisfies (Y), and (Z) since $r- (-r-k) >0$ as in the modification of $(\ref{eqn:step1})$. 
Since the $(j+1)$-th row is not modified yet, \c{i}{j+1} also satisfies (Y) and (Z).

Third, we will modify entries as in (\ref{eqn:step3}) as {\cred follows.}
{\cred
\[
\begin{array}{l}
\text{We use \c{2n-1}{2}, \ldots \c{n+1}{2} to eliminate } (2n,2), \ldots, (n+2,2).\\
\text{We use \c{2n-2}{3}, \ldots \c{n+1}{3} to eliminate } (2n-1,3), \ldots (n+2,3).\\
\quad\quad    \vdots \\
\text{We use \c{n+1}{n} to eliminate } (n+2,n).
\end{array}
\]}
Here, to eliminate the $(i,j)$-th entry ($j=2, \ldots, n, i= 2n+2-j \ldots, n+2$), we use
\c{i-1}{j}, which is together with \c{2n+1-j}{2n+2-i}, \r{j}{i-1}, and \r{2n+2-i}{2n+1-j} if $i \neq 2n +2-j$ (resp.\,if $i=2n+2-j$).
Since entries on the $(i-1)$-th column have orders $\geq -r-k+1$ except for $(i,i-1)$,
\c{i-1}{j} satisfies (Y) and (Z).
On the other hand, since $j <2n+2-i\leq n$ if $i \neq 2n+2-j$,
\c{2n+1-j}{2n+2-i} satisfies (Y) and (Z).
Moreover, since $2r+k \geq 1$ by the assumption, \r{j}{i-1} does not affect the $(s,t)$-th entries modulo $\p^{\gamma_{s,t}(C)+1}$, even for $(s,t) =(i-1,j+1)$  (resp.\,$(s,t) =(i-1,2n)$) if $j \neq n$ (resp.\,if $j=n$).
Thus \r{i-1}{j} satisfies (Y) and (Z).
Similarly, we can show \r{2n+2-i}{2n+1-j} affects only $(2n+2-j,2n+3-i)$ (resp.\,$(2n+2-j,2n)$) if $2n+2-i \neq n$ (resp.\,if $2n+2-i = n$).

Since 
\[
\langle \textup{the $(n+i)$-th row}, \textup{the $(n+j)$-th row}  \rangle = 0
\]
for $1\leq i,j \leq n$,
{\cred all the} entries \eqref{eqn:step4} have orders $\geq \gamma_{i,j} (C)+1$.
Finally, we have constructed $P_{1} \in I^{m}$ such that $\gamma (\sigma(P_{1})^{-1} C P_{1}) > \gamma (C)$. 
By repeating these construction for $C_{i} := \sigma(P_{i})^{-1} C P_{i}$ ($i=1, \ldots$), we can construct the sequence $P_{i} \in I^{m}$ ($i=1, \ldots$).
Since $P_{i}$ is a product of elementary matrices as above,
we have
\[
P_{i+1} \in (1 + \p^{\gamma(C_{i})}M_{2n} (\co)) \cap \GSp(V)
\]
for $i\geq 0$, where we put $C_{0}:=C$.
Therefore, we can show that the product 
\[
P_{\infty}:=\prod_{i=1}^{\infty}P_{i}
\]
converges in $I^{m}$.
Then the entries {\cred \eqref{eqn:step1}, \eqref{eqn:step2}, \eqref{eqn:step3}, \eqref{eqn:step4}} in $\sigma(P_{\infty})^{-1} C P_{\infty}$ are {\cred zero}, as desired.

\subsection*{{\cred Step 2:} Eliminate other entries of $C$}\mbox{}\par

Here, we will eliminate other entries so that we find $C'$ as in Claim \ref{claim:affineC'}.
We will eliminate the following entries in that order.
\begin{eqnarray}
\label{eqn:step5}
&&(2n,2n),(2n-1,2n), \ldots, (n+1,2n),\\
\label{eqn:step6}
&&(2n-1, 2n-1), \ldots (n+1, 2n-1), (2n-1, 2n-2), 
 \ldots, (n+1,n+1), \\
\label{eqn:step7}
&&(n,1), \ldots, (2,1), \\
\label{eqn:step8}
&&(2,n+1), \ldots, (n-1, n+1), (2,n+2), \ldots, (n-2,n+2), \ldots, (2, 2n-2).
\end{eqnarray}

To eliminate (\ref{eqn:step5}), we use  \c{2n-1}{2n}, \ldots, \c{n+1}{2n}, \c{1}{2n} to eliminate the $(n+1,2n)$, $\ldots, (2n,2n)$-th entries. Here, we use that $r-(-r-k) >0$ to {\cred ensure} the existence of such elementary matrices in $I^{m}$.
The transformation \c{1}{2n} (resp.\,\c{n+i}{2n}) is together with \r{2n}{1} (resp.\,\c{1}{n+1-i}, \r{2n}{n+i}, and \r{n+1-i}{1}), which preserve the form as in (\ref{C2}).
Moreover, they clearly do not restore {\cred the entries we have set to zero}.
Note that, after eliminating (\ref{eqn:step5}), the $(n+1, 2), \ldots, (n+1, 2n)$-th entries {\cred are zero} since 
\[
\langle \textup{the $i$-th column}, \textup{the $2n$-th column} \rangle = 0  
\]
for $i = 2, \ldots, n$.

To eliminate (\ref{eqn:step6}), we use \c{1}{j} to eliminate the $(n+1,j)$-th entry, and \c{i-1}{j} to eliminate the $(i, j)$-th entries for $i \geq n+1$.
The transformation \c{1}{j} is together with \c{2n+1-j}{2n}, \r{j}{1}, and \r{2n}{2n+1-j}.
They do not affect {\cred the entries we have set to zero} since the $(i,2n+1-j)$-th entries are already {\cred zero} for $i\geq n+2$.
On the other hand, \c{i-1}{j} is together with \c{2n+1-j}{2n+2-i}, \r{j}{i-1}, and \r{2n+2-i}{2n+1-j}. Clearly, \c{2n+1-j}{2n+2-i} and \r{2n+2-i}{2n+1-j} {\cred do} not restore {\cred the entries we have set to zero}.
Moreover, \r{j}{i-1} does not restore {\cred the entries we have set to zero}, since the $(j,j), \ldots, (j,2n)$-th entries are already {\cred zero}.

To eliminate (\ref{eqn:step7}), we use \r{n+1}{i} to eliminate the $(i,1)$-th entries.
They are together with transformations \r{2n+1-i}{n}, \c{i}{n+1}, and \c{n}{2n+1-i}.
They clearly do not restore {\cred any of the entries we have set to zero} since the $(n+1, 2), \ldots, (n+1, 2n)$-th entries {\cred are already zero}.

To eliminate (\ref{eqn:step8}), we use \r{j+1}{i} to eliminate the $(i,j)$-th entries.
They are together with transformations \r{2n+1-i}{2n-j}, \c{i}{j+1}, and \c{2n-j}{2n+1-i}. 
The transformation \r{2n+1-i}{2n-j} affects only the $(2n-j,2n-i)$-th entry, which {\cred has not been set to zero yet} since $2n-i \geq j+1$ if $i+j+1 \neq 2n+1$.
The transformation \c{i}{j+1} affects only the $(1,j+1), \ldots, (i-1,j+1)$-th entries, which {\cred have not been set to zero yet}.
Similarly, the transformation \c{2n-j}{2n+1-i} affects only {\cred the entries} $(1, 2n+1-i), \ldots, (n, 2n+1-i)$, which {\cred have not been set to zero yet} since $2n+1-i \geq j+2$ as above.

It is clear that after eliminating entries (\ref{eqn:step1}), \ldots, (\ref{eqn:step8}), we obtain a matrix of desired form, i.e., we have constructed $P$ satisfying the assumption in Claim \ref{claim:affineC'}.
{\cred This} finishes the proof of the surjectivity of the lower map of the {\cred second} diagram.

For the upper map of the second diagram, the proof is similar to that for the first diagram in Subsection \ref{subsection:pfsemiinf}.
{\cred This completes the proof.}

\subsection{The relation $\sim_{b,m,r}$ and $\dot{\sim}_{b,m,r}$}
\label{subsection:relationbmr}
Recall that we write $\alpha_{v}$ for $\langle v, F^{n}(v)\rangle.$

\begin{proposition}
\label{proposition:bmr}
Assume $r+ k \geq m+1$.
Let $x,y \in V_{b}^{\mathrm{symp}}$.
The following are equivalent.
\begin{enumerate}
\renewcommand{\labelenumi}{(\roman{enumi})}
\item
$x\sim_{b,m,r}y$.
\item
$x \in g_{b,r}(y)
{}^t\!
\left(\begin{array}{cccc|cccc}
\co^{\times} & \p^{m}& \cdots &\p^{m}&\p^{m}&\cdots&\cdots&\p^{m}
\end{array}\right).
$
\end{enumerate}

Moreover, if we assume
$$\sigma(\alpha_{x})/\alpha_{x}\equiv1, \sigma(\alpha_{y})/\alpha_{y}\equiv1 \mod \p^{m+1}$$ and {\cred replace} $\co^{\times}$ with $1+\p^{m+1}$, the same statement for $\dot{\sim}_{b,m,r}$ is true.
\end{proposition}

\begin{proof}
First, by Lemma \ref{lemma:order}, we can show the following (in this part, we do not need the assumption $r+k \geq m+1$).
\begin{itemize}
\item
We have
\begin{equation}
\label{eq:F1}
\varpi^{r}F(g_{b,r}(y)) = g_{b,r}(y)A,
\end{equation}
where
\[
A = 
\left(\begin{array}{cccc|ccccc}
&&&&(-1)^{n+1}\frac{\sigma(\alpha_{y})}{\alpha_{y}}\varpi^{2r+k}&&&&\\
1&&&&b_{1}&b_{2}&\cdots&b_{n-1}&a_{n}\\
&\ddots&&&&&&&\vdots\\
&&1&&&&&&a_{2}\\\hline
&&&&&\frac{\sigma(\alpha_{y})}{\alpha_{y}}\varpi^{2r+k}&&&\\
&&&&&&\ddots&&\\
&&&&&&&\ddots&\\
&&&&&&&&\frac{\sigma(\alpha_{y})}{\alpha_{y}}\varpi^{2r+k}\\
&&&1&&&&&a_{1}
\end{array}\right).
\]
Here, we use the same notation as in Lemma \ref{lemma:order}.
\item
We have
\begin{equation}
\label{eq:F2}
\varpi^{r} V_{k}(g_{b,r}(y)) = g_{b,r} (y)B,
\end{equation}
where
\begin{align*}
B&= \diag (1, \ldots, 1,\frac{\sigma^{-1}(\alpha_{y})}{\alpha_{y}}, \ldots, \frac{\sigma^{-1}(\alpha_{y})}{\alpha_{y}})\\
& \times
\left(\begin{array}{cccc|ccccc}
b_{1}'&\varpi^{2r+k}&&&b_{2}'&\cdots&b_{n-1}'&a_{n}'&\\
&&\ddots&&&&&\vdots&\\
&&&\varpi^{2r+k}&&&&\vdots&\\
&&&&&&&a_{1}'&\varpi^{2r+k}\\\hline
(-1)^{n+1} &&&&&&&&\\
&&&&1&&&&\\
&&&&&\ddots&&&\\
&&&&&&\ddots&&\\
&&&&&&&1 &
\end{array}\right).
\end{align*}
Here, as above, $a_{i}', b_{j}' \in \breve{K}$ satisfy $\ord a_{i}' \geq ir + \frac{k}{2}i$ and $\ord b_{j}' \geq jr + \frac{k}{2}j$.
\end{itemize}
Note that {\cred the implication} (i) $\Rightarrow$ (ii) is trivial.
Therefore, we will show (ii) $\Rightarrow$ (i).
We need to show the inclusion
\begin{equation}
\label{eq:desincl}
g_{b,r}(x) \in g_{b,r}(y)  I^{m}.
\end{equation}
First, we will show the $1$st, $\ldots$, $n$-th columns of (\ref{eq:desincl}).
Clearly, the $1$st column of (\ref{eq:desincl}) is no other than the assumption (ii).
Applying $\varpi^{r} F$ and using (\ref{eq:F1}) to the assumption (ii),
we have
\[
x \in g_{b,r}(y)
{}^t\!
\left(\begin{array}{ccccc|ccccc}
\p^{m+1}  & \co^{\times}
&\p^{m}
& \cdots &\p^{m}&
\p^{m}& \cdots
&\cdots&\cdots&\p^{m}
\end{array}\right),
\]
which implies the $2$nd column of (\ref{eq:desincl}).
Here, we use $2r+k \geq \lceil r+k \rceil \geq 1$.
Similarly, by applying $\varpi^{r} F$ and (\ref{eq:F1})  to the assumption (ii) $(i-1)$ times, we can verify the $i$-th column of (\ref{eq:desincl}) ($1 \leq i \leq n$). Here, we use that $\ord b_{i} \geq m+1$, which follows from the assumption $r+k \geq m+1$.

Next, we will verify the $2n$-th column of (\ref{eq:desincl}).
Since we have
\[
\varpi^{(n-1)r} F^{n-1} (x)\in
g_{b,r}(y)
{}^t\!
\left(\begin{array}{cccc|cccc}
\p^{m+1} & \cdots
&\p^{m+1}
& \co^{\times} &\p^{m}&
\cdots& \cdots
&\p^{m}
\end{array}\right),
\]
applying $\varpi^{r} F$ and using $(\ref{eq:F1})$, we have
\begin{align*}
&\varpi^{nr} F^{n} (x) \\
&\in
g_{b,r}(y)
{}^t\!
\left(\begin{array}{cccc|cccc}
\p^{m+1} & \cdots
& \cdots
&\p^{m+1}
& \p^{m+1} &
\cdots& \p^{m+1}
&\co^{\times}
\end{array}\right),
\end{align*}
which implies the $2n$-th column of $(\ref{eq:desincl})$.

Finally, we will verify the $(n+1)$-th, $\ldots$ $(2n-1)$-th column of $(\ref{eq:desincl})$.
Note that, 
by {\cred applying} $F^{-1}$ {\cred to} (\ref{eq:F1}) {\cred with} $r=0$, we have
\begin{align}
\begin{aligned}
\label{eq:Giv}
G_{1} (v)&= 
(-1)^{n+1} \frac{\alpha_{v}}{\sigma^{-1}(\alpha_{v})} \varpi^{k}F^{-1}(v) + c_{1} v, \\
G_{i} (v)&=
\frac{\alpha_{v}}{\sigma^{-1}(\alpha_{v})} (\varpi^{k} F^{-1} (G_{i-1}(v))) + c_{i} v \quad(2\leq i \leq n-1),
\end{aligned}
\end{align}
where $c_{j} \in \breve{K}$ $(1\leq j \leq n-1)$ depend on $v$ and satisfy
\begin{equation}
\label{eq:ordc}
\ord c_{j} \geq \frac{k}{2}j.
\end{equation}
In the following, we put $v:=x$.
By applying $\varpi^{r}V_{k}$ and using $ (\ref{eq:F2})$ to the assumption (ii), and using $\lceil r+ \frac{k}{2} \rceil = r+k \geq m+1$,
we have
\[
\varpi^{r} V_{k} (x) \in
g_{b,r}(y)
{}^t\!
\left(\begin{array}{cccc|cccc}
\p^{m+1} & \p^{m+1}
& \cdots
&\p^{m+1}
& \co^{\times} &
\p^{m}
& \cdots
&\p^{m}
\end{array}\right).
\]
{\cred By \eqref{eq:Giv} and \eqref{eq:ordc}, together with the assumption,}
 we have
\begin{equation}
\label{eq:G1}
\begin{split}
&\varpi^{r}G_{1} (x) 
\in
g_{b,r}(y)
{}^t\!
\left(\begin{array}{cccc|cccc}
\p^{m+1} & \p^{m+1}
& \cdots
&\p^{m+1}
& \co^{\times} &
\p^{m}
& \cdots
&\p^{m}
\end{array}\right),
\end{split}
\end{equation}
which implies the $(n+1)$-th column of (\ref{eq:desincl}).
Applying $\varpi^{r} V_{k}$ and $ (\ref{eq:F2})$ to (\ref{eq:G1}) and using (\ref{eq:ordc}) repeatedly, we can show the $i$-th column of (\ref{eq:desincl}) for $n+2 \leq i \leq 2n-1$, and {\cred this} finishes the proof.
The assertion for $\dot{\sim}_{b,m,r}$ can be shown in the same way.
\end{proof}
\begin{remark}
Suppose that $r+k \geq m+1$.
By Proposition \ref{proposition:bmr}, 
for $x, y \in V_{b}^{\mathrm{symp}}$, we have
\[
x \sim_{b,m,r+1} y \Rightarrow x \sim_{b,m,r} y.
\]
(The same statement for $\dot{\sim}_{b,m,r}$ holds true.)
Therefore, by Theorem \ref{theorem:comparison}, we have morphisms of sets
\begin{align}
\label{eqn:transition}
\begin{aligned}
&X_{w_{r+1}}^{m} (b) (\overline{k}) \rightarrow X_{w_{r}}^{m} (b) (\overline{k}), \\
&\dot{X}_{w_{r+1}}^{m}(b)(\overline{k}) \rightarrow \dot{X}_{w_{r}}^{m}(b)(\overline{k}).
\end{aligned}
\end{align}
In Remark \ref{remark:morphismadlv}, we will show that these morphisms come from morphisms of schemes.
Moreover, we have an isomorphism of sets
\[
(\mathop{\varprojlim}\limits_{r>m}\dot{X}^{m}_{w_r}(b))(\overline{k})
\simeq X_{w}^{(U)}(b)
\]
by Theorem \ref{theorem:comparison}.
Indeed, for any $(\overline{v_{r,m}})_{r,m}$ in the left-hand side (where $v_{r,m} \in \Vsymp$), 
$v_{r,m}$ converges as an element of $\Vsymp$ as follows.
Fix $(r_{0}, m_{0})$.
Let $M$ be {\cred the} minimum of order of entries of $g_{b,r_{0}} (v_{r_{0},m_{0}})$.
Then, for any $(r_{1},m_{1})$ and $(r_{2}, m_{2})$ with $r_{i} > r_{j}$ and $m_{i} > m_{j}$ for $i>j$, we have
\begin{align*}
v_{r_{2}, m_{2}} &\in g_{b, r_{1}} (v_{r_{1}, m_{1}}) 
 {}^t\!
\left(\begin{array}{cccc|cccc}
1+ \p^{m_{1}+1} & \p^{m_{1}}& \cdots &\p^{m_{1}}&\p^{m_{1}}&\cdots&\cdots&\p^{m_{1}}
\end{array}\right)\\
&\subset  v_{r_{1}, m_{1}} + \p^{M + m_{1}} M_{2n}(\co).
\end{align*}
Therefore, they converge to an element $v \in V$.
We have $v \in \Vsymp$ since $\ord (\alpha_{v})$ is preserved under $\dot{\sim}_{b,m,r}$.
One can formulate {\cred a} similar description for $X_{w}^{(B)}(b)$ as an $\Z$-quotient of the inverse limit of $X_{w_{r}}^{m}(b)(\overline{k})$, but we omit {\cred it} here.
\end{remark}

\begin{remark}
\label{remark:morphismadlv}
Actually, we can prove that {\cred the} maps (\ref{eqn:transition}) {\cred are} induced by the transition map between perfect schemes,
\begin{align}
\label{eqn:schtransition}
\begin{aligned}
&X_{w_{r+1}}^{m} (b)^{\mathrm{perf}} \rightarrow X_{w_{r}}^{m} (b)^{\mathrm{perf}}, \\
&\dot{X}_{w_{r+1}}^{m}(b)^{\mathrm{perf}} \rightarrow \dot{X}_{w_{r}}^{m}(b)^{\mathrm{perf}}.
\end{aligned}
\end{align}
To this end, we should introduce a functorial variant of Theorem \ref{theorem:comparison} and Proposition \ref{proposition:bmr}.
For {\cred a} perfect algebra $R$ over $\overline{k}$, we put $K_{R}= \mathbb{W}(R)[\frac{1}{\varpi}]$, and $V_{R} := \mathbb{W} (R)[\frac{1}{\varpi}]^{2n}$ with the symplectic form associated with $\Omega$ as in (\ref{eqn:Omega}).
Here, we put
\[
\mathbb{W} (R) :=
\begin{cases}
R \widehat{\otimes}_{\overline{k}} \co  &\textup{ if } \chara K=p,\\
W(R) \otimes_{W(\overline{k})} \co &\textup{ if } \chara K =0.
\end{cases}
\]
Then we can define 
\[
V_{b,R}^{\mathrm{symp}}
:=
\left\{
v \in V_{R} \left|
\langle v, F(v) \rangle = \cdots = \langle v, F^{n-1}(v) \rangle =0, \langle v, F^{n} (v) \rangle \in  \mathbb{W}(R) \frac{1}{\varpi}
\right\}\right..
\]
Also, we can define $g_{b,r}$ as before, and we can prove the analogue of Theorem \ref{theorem:comparison} in the following sense.
We define  $X^{m}_{r} (R)$ and $\dot{X}^{m}_{r} (R)$ by
\begin{align}
\begin{aligned}
X^{m}_{r} (R) := \{ 
g \in G(K_{R}) / I^{m} \mid g^{-1} b\sigma (g) \in I^{m} w_{r} I^{m}
\},\\
\dot{X}^{m}_{r} (R) :=
\{ 
g \in G(K_{R}) / \dot{I}^{m} \mid g^{-1} b\sigma (g) \in \dot{I}^{m} w_{r} \dot{I}^{m}
\}.
\end{aligned}
\end{align}
Then we can show that $g_{b,r}$ defines surjections
{\cred
$V_{b,R}^{\mathrm{symp}} \twoheadrightarrow X^{m}_r(R)$ and $V_{b,R}^{\mathrm{symp}} \twoheadrightarrow \dot{X}^{m}_r(R)$
}
if $r+k \geq m+1$.
Moreover, we can show the analogue of Proposition \ref{proposition:bmr} {\cred and hence}
we have transition maps
\begin{align}
\label{eqn:transitionR}
\begin{aligned}
&X^{m}_{r+1} (R)  \rightarrow X^{m}_{r} (R), \\
&\dot{X}^{m}_{r+1} (R)\rightarrow \dot{X}^{m}_{r} (R).
\end{aligned}
\end{align}
Finally, we will define the transition map (\ref{eqn:schtransition}).
For simplicity, we only define the former case.
Let $R$ be any perfect algebra over $\overline{k}$ as above.
For any $a \in X_{w_{r+1}}^{m} (b) (R)$, by \cite[Lemma 1.3.7]{Zhu2017} and \cite[proof of Lemma 1.3]{Zhu2017a}, there exists an \'{e}tale cover $R'$ such that $a |_{\Spec R'} \in X^{m}_{r+1} (R') \subset X_{w_{r+1}}^{m} (b)(R')$.
Then we can define $\iota (a|_{\Spec R'}) \in X^{m}_{r}(R')$ {\cred as the image under} (\ref{eqn:transitionR}).
By \'{e}tale descent, we can show that $\iota (a|_{\Spec R'})$ descends to an element in $X_{w_{r}}^{m} (b) (R)$.
Since this construction is functorial, it defines a morphism $X_{w_{r+1}}^{m}(b) \rightarrow X_{w_{r}}^{m} (b)$ as desired.
\end{remark}

\section{{\cred description} of {\cred the} connected components}
\label{section:conncomp}
In this section, we describe the connected components of semi-infinite (resp.\, affine) Deligne--Lusztig varieties by following the method of \cite{Viehmann2008a} (see also \cite{Viehmann2008}).
\subsection{Representatives of $b$}
\label{subsection:repb}
In the following, we define two kinds of {\cred representatives} of $b$ to describe the connected components of affine Deligne--Lusztig varieties.
We put 
$$
b_{0}
:=
\left(\begin{array}{cccc|cccc}
&&&&(-1)^{n+1}&&&\\
1&&&&&&&\\
&\ddots&&&&\\
&&1&&&\\\hline
&&&&&1&&\\
&&&&&&\ddots&\\
&&&&&&&1\\
&&&1&&&&
\end{array}\right), 
$$
and 
\[
A_{k}:=\diag(1,\varpi^{k}, \ldots, 1,\varpi^{k}).
\]
 {\cred Set $b: = b_{0} A_{k}$; we call this the Coxeter-type representative of $b$ with $\kappa (b) =k$.}
 Note that $\lambda(b) = - \varpi^{k}$ holds as before.
 
 On the other hand, we put
 
 $$
 \bsp
 :=
 \begin{cases}  
 \diag (1, \ldots, 1 | -1, \ldots, -1)
 & \textup{ if } k=0,  \\ 
 \vspace{0.2cm}
 \left(\begin{array}{ccccc}
 \left(
 \begin{array}{cc}
 0&\varpi \\
 1&0
 \end{array}
 \right)
&& \\
&\ddots & \\
&& 
\left(
 \begin{array}{cc}
 0&\varpi \\
 1&0
 \end{array}
 \right)
\end{array}\right) &\textup{ if } k=1,
\end{cases}
 $$
 and we call this representative the special-type representative of $b$ with $\kappa (b) =k$.

Let $\mathcal{A}^{\mathrm{red}}$ be the apartment of the reduced building of $\GSp_{2n}$ over $\breve{K}$ {\cred corresponding} to the maximal split torus {\cred consisting} of diagonal matrices in $\GSp_{2n}( \breve{K})$.
We can show that $b$ acts on $\mathcal{A}^{\mathrm{red}}$ {\cred and has a} unique fixed point $x$.
More precisely, $x$ is given by 

\[
\left\{
\begin{array}{ll}
\frac{1}{2} k{\alpha}^{\vee}_{2}+  \frac{1}{2} k{\alpha}^{\vee}_{3} + \cdots + \frac{m-1}{2}k{\alpha}^{\vee}_{n-2}+ \frac{m-1}{2}k{\alpha}^{\vee}_{n-1} +\frac{m}{2} k \beta^{\vee}  & \textup{if } n \textup{ is even},  \\
\frac{-1}{4}k( {\alpha}^{\vee}_{1} + {\alpha}^{\vee}_{3} + \cdots + \alpha_{n-2} + {\beta}^{\vee}) & \textup{if } n \textup{ is odd}.
\end{array}
\right.
\] 
Here, we put $m:=n/2$ if $n$ is even.
Moreover, ${\alpha}^{\vee}_{1} \ldots {\alpha}^{\vee}_{n-1}, \beta^{\vee}$ {\cred are} the usual simple coroots given by
\[
\begin{array}{cccccccccc}
&&&i&i+1&&2n+1-i&2n+2-i&&\\
\alpha^{\vee}_{i}:=&\diag(1,& \ldots,&t,&t^{-1},&\cdots &t,&t^{-1},&\cdots &1),
\end{array}
\]
\[
\begin{array}{ccccccc}
&&&n&n+1&&\\
\beta^{\vee}:=&\diag(1,& \ldots,&t,&t^{-1},&\cdots &1).
\end{array}
\]
We denote the corresponding compact open subgroup of $\GSp_{2n}(\breve{K})$ by $G_{x,0}$.
\begin{figure}
\centering
\captionsetup{justification=centering}
\[
\left(\begin{array}{cc|cc}
\co &\p & \co & \co \\
\co &\co & \p^{-1} & \co \\  \hline
\p &\p & \co &  \p\\
\co &\p & \co & \co\\
\end{array}\right), 
\left(\begin{array}{ccc|ccc}
\co &\p & \co & \p &\co & \p \\
\co &\co & \co & \co & \co &\co \\  
\co &\p & \co &  \p & \co & \p\\ \hline
\co &\co & \co & \co & \co &\co\\
\co &\p & \co &\p & \co & \p \\
\co &\co & \co & \co & \co& \co
\end{array}\right)
\]
\caption*{The shape of $G_{x,0}$ when $k=1$, for $2n=4, 6$}
\end{figure}

\subsection{Linear algebraic description of {\cred the} connected components}
In this subsection, we study the $J_{b}(K)$-action on $V^{\mathrm{symp}}_{b}$.
\begin{definition}
We put
\[
\mathcal{L}:=
\begin{cases}
\co e_{1} + \cdots +\co e_{n} + \p^{k} e_{n+1} +  \co e_{n+2} + \p^{k} e_{n+3}+ \cdots \p^{k} e_{2n-1} + \co e_{2n}  & \textup{if } n \textup{ is even,}\\
\co e_{1} + \cdots + \co e_{2n} & \textup{if } n \textup{ is odd,}
\end{cases} 
\]
where $e_{1}, \ldots, e_{2n}$ {\cred form} the standard basis of $V$.
We define
\[
\Lsymp :=
\{
v \in \Vsymp \cap \mathcal{L} \mid
\langle  v, F^{n}(v) \rangle \in \varpi^{\lfloor \frac{kn}{2} \rfloor} \co^{\times} 
\}.
\]
We also define $\Lsympsp$ in the same way.
Moreover, we define
\[
\la_{b}^{\mathrm{symp,rat}} :=
\{
v \in \Lsymp | \langle v, F^{n} (v) \rangle \in \varpi^{\lfloor \frac{kn}{2} \rfloor} \co_{K}^{\times}
\}.
\]
We also define $\Lsymprat$ in the same way.
\end{definition}

\begin{proposition}
\label{prop:component}
\begin{enumerate}
\item
We have the decomposition
\[
\Vsymp =
\bigsqcup_{j \in J_{b}(K)/ J_{b,0}} j \Lsymp.
\]
\item
We have the decomposition
\[
\Vsympsp =
\bigsqcup_{j \in J_{b_{\mathrm{sp}}}(K)/J_{b_{\mathrm{sp}},0}} j \Lsympsp.
\]
\item
There exists an element $g\in G_{x,0}$ such that 
\[
g^{-1} \bsp \sigma(g) = b.
\]
\end{enumerate}
Here, we put $J_{b,0}:=J_{b}(K) \cap G_{x,0}$ (resp.\,$J_{\bsp,0} :=J_{\bsp}(K) \cap G_{x,0}$).
\end{proposition}

To prove this proposition, we need to define the reduced version of $g_{b,r}$.

\begin{definition}
For any $v\in V$, we put
\begin{align*}
g_{b}^{\mathrm{red}}(v) := g_{b,0}(v) \cdot \diag(1, \varpi^{\lceil \frac{-k}{2} \rceil}, \ldots, \varpi^{\lceil \frac{-k(i-1)}{2} \rceil}, \ldots, \varpi^{\lceil \frac{-k(n-1)}{2} \rceil}, \varpi^{\lceil \frac{-kn}{2} \rceil + \lfloor \frac{k(n-1)}{2} \rfloor}, \ldots, \varpi^{\lceil \frac{-kn}{2} \rceil}),
\\
g_{\bsp}^{\mathrm{red}}(v) := g_{\bsp,0}(v) \cdot \diag(1, \varpi^{\lceil \frac{-k}{2} \rceil}, \ldots, \varpi^{\lceil \frac{-k(i-1)}{2} \rceil}, \ldots, \varpi^{\lceil \frac{-k(n-1)}{2} \rceil}, \varpi^{\lceil \frac{-kn}{2} \rceil + \lfloor \frac{k(n-1)}{2} \rfloor}, \ldots, \varpi^{\lceil \frac{-kn}{2} \rceil}).
\end{align*}
\end{definition}
Then {\cred one checks} that $g_{b}^{\mathrm{red}}(v) \in G_{x,0}$ (resp.\,$g_{\bsp}^{\mathrm{red}}(v) \in G_{x,0}$) for any $v \in \Lsymp$ (resp.\,$v\in \Lsympsp$).

\begin{proof}
First, we will give a short remark.
Let $v \in \Lsymp$ 
and $j \in J_{b} (K)$. 
We suppose that $jv \in \Lsymp$.
Then we have
\[
g_{b}^{\mathrm{red}}(jv) = j g_{b}^{\mathrm{red}} (v) \in G_{x,0} 
\]
and thus we have $j \in G_{x,0}.$
Therefore, to prove $(1)$ of this proposition,
we need to show that, for any $v \in \Vsymp$, there exists $j \in J_{b} (K)$ such that $jv \in \Lsymp$.
The same argument {\cred works for $\bsp$}.

Next, we deduce $(1)$ and $(3)$ from $(2)$.
Suppose that $(2)$ holds true.
Let $g \in \GSp(\breve{K})$ be an element such that
\[
g^{-1} \bsp \sigma(g) =b.
\]
Take $v \in \Lsymp$.
Then we have $gv \in \Vsympsp$.
By $(2)$, there exists $j \in J_{\bsp}(K)$ such that
\[
jg(v) \in \Lsympsp.
\]
Then we have
\[
g_{\bsp}^{\mathrm{red}} (jgv) = j g_{\bsp}^{\mathrm{red}} (gv) =jg g_{b}^{\mathrm{red}}(v) \in G_{x,0}.
\]
Now we have 
$
jg \in G_{x,0}
$
and $(jg)^{-1} \bsp \sigma(jg) = {\cred b}$. {\cred This} finishes the proof of $(3)$.

By using $(3)$ and $(2)$, we can easily show $(1)$.

Therefore, it suffices to show $(2)$.
In the following, we denote $\bsp \sigma$ by $F$.
First, we will consider the case where $n$ is even.
We put $n:=2m.$
For any $v \in \Vsympsp$, by the $\GL_{n}$-case (\cite[Lemma 6.11]{Chan2021}), there exists a $j_{0} \in J_{b_{\mathrm{sp},0}}^{(\GL_{n})}$ such that
\[
j_{0} p_{0} (v) \in \mathcal{L}_{b_{\mathrm{sp},0}}^{\mathrm{adm}}.
\]
Here, we consider the decomposition
$V = V_{0} \oplus V_{1}$,
where $V_{0}$ is a vector space spanned by $e_{1}, \ldots e_{n}$, and $V_{1}$ is a {\cred vector} space {\cred spanned} by $e_{n+1}, \ldots, e_{2n}$.
We denote the projection from $V$ to $V_{i}$ by $p_{i}$.
Moreover, $b_{\mathrm{sp},i}$ is the restriction $b_{\mathrm{sp}}|_{V_{i}}$, and
we denote the corresponding inner forms by $J_{b_{\mathrm{sp},i}}^{(\GL_{n})}$.
For the definition of $\la_{b_{\mathrm{sp},0}}^{\mathrm{adm}}$, see \cite[Definition 6.10]{Chan2021}.
Then we can show that 
\[
j
:=
\left(\begin{array}{c|c}
j_{0} & \\ \hline
& \omega ^{t}j_{0}^{-1} \omega
\end{array}\right)
\]
is contained in $J_{\bsp}(K)$. Here, we put
\[
\omega :=
\left(\begin{array}{ccccc}
&&&&1\\
&&&-1&\\
&&\iddots&&\\
&1&&&\\
-1&&&&
\end{array}\right),
\]
which is the $n \times n $ matrix.
By multiplying $j$ by some power of $\diag(1,\ldots,1, \varpi, \ldots, \varpi)$ from the left, we have the following.
\begin{itemize}
\item
$j \in J_{\bsp} (K).$
\item
$p_{0}(jv) \in  \mathcal{L}_{b_{\mathrm{sp},0}}^{\mathrm{adm}}.$
\item
$\langle jv, F^{n}(jv) \rangle \in \varpi^{m} \co^{\times}.$
\end{itemize}

By Lemma \ref{lemma:hogehoge}, it suffices to show the following claim.

\begin{claim}
\label{claim:findj}
Suppose $n$ is even as above.
Let  $v, \widetilde{v} \in \Vsympsp$ be elements satisfying the following.
\begin{enumerate}
\item
$p_{0}(v) = p_{0} (\widetilde{v})$.
\item
$\langle v, (\bsp \sigma)^{n} (v) \rangle = \langle \widetilde{v}, (\bsp \sigma)^{n} \widetilde{v} \rangle$.
\end{enumerate}
Then there exists $j \in J_{{\cred \bsp}}(K)$ such that $jv = \widetilde{v}$.
\end{claim}
Here we give the proof of Claim \ref{claim:findj} (this is essentially the same as in the proof of \cite[Lemma 5.1]{Viehmann2008a}). 
We put $v_{0}:= p_{0}(v) = p_{0}(\widetilde{v})$, $v_{1}:=p_{1}(v)$, and $\widetilde{v}_{1}:= p_{1}(\widetilde{v})$.
Then by the assumption, we have
\begin{align*}
 \langle v_{0} + \widetilde{v}_{1} - v_{1}, (\bsp\sigma)^{i} (v_{0}+\widetilde{v}_{1}-v_{1})  \rangle 
&= \langle v_{0}, (\bsp \sigma)^{i}(\widetilde{v}_{1}-v_{1}) \rangle + \langle \widetilde{v}_{1}-v_{1}, (\bsp \sigma)^{i} (v_{0}) \rangle \\
& = \langle \widetilde{v}, (b_{\mathrm{sp}} \sigma)^{i} (\widetilde{v}) \rangle - \langle v, (b_{\mathrm{sp}} \sigma)^{i} (v) \rangle
= 0
\end{align*}
for $i=1, \ldots, n$.
Therefore, for any $\phi \in \mathcal{D}_{k}$, we have
\begin{equation}
\label{eq:phi}
\langle v_{0} + \widetilde{v}_{1} - v_{1}, \phi (v_{0}+\widetilde{v}_{1}-v_{1})  \rangle =0.
\end{equation}
Note that, {\cred for dimensional reasons},
$\phi (v_{0}+\widetilde{v}_{1}-v_{1})$ can be spanned by $ (\bsp\sigma)^{i} (v_{0}+\widetilde{v}_{1}-v_{1})$ for $-n+1 \leq i \leq n$.
For any $\mathcal{D}_{k} \supset \Ann v_{0} \ni A$ and any $\phi \in \mathcal{D}_{k}$, we have
\[
\langle v_{0}, \phi A (\widetilde{v}_{1}-v_{1}) \rangle =0.
\]
Since $v_{0} \in \mathcal{L}_{b_{\mathrm{sp},0}}^{\mathrm{adm}},$ we have
\[
A(\widetilde{v}_{1}-v_{1})=0.
\]
Now we define $j \in \GL (V)$ by
\begin{itemize}
\item
$j(v_{0}) := v_{0}+ \widetilde{v}_{1} - v_{1}$,
\item
$j(\phi(v_{0})) := \phi (j(v_{0}))$,
\item
$j|_{V_{1}} := \mathrm{id}$,
\end{itemize}
for any $\phi \in \mathcal{D}_{k}$.
By the above considerations, the above $j$ is well-defined.
By definition, $j$ commutes with $\bsp \sigma$.
Moreover, for any $\phi, \phi' \in \mathcal{D}_{k}$ and any $w_{1}, w_{1}' \in V_{1}$,
 we have
 \begin{eqnarray*}
&& \langle
 j(\phi(v_{0})+w_{1}), j(\phi'(v_{0})+w_{1}')
 \rangle  \\
 &=& \langle
 \phi(v_{0}+\widetilde{v}_{1}-v_{1})+w_{1}, \phi'(v_{0}+\widetilde{v}_{1}-v_{1}) +w_{1}'
 \rangle \\
 &=& \langle \phi(v_{0}), w_{1}' \rangle + \langle w_{1}, \phi'(v_{0}) \rangle \\
 &=& \langle \phi(v_{0}) + w_{1} , \phi'(v_{0})+ w_{1}' \rangle.
 \end{eqnarray*}
 Here, in the second equality, we use the equality (\ref{eq:phi}).
Therefore, we have $j \in J_{\bsp} (K)$. {\cred This} finishes the proof of Claim \ref{claim:findj}, and the proof of Proposition \ref{prop:component} in the case where $n$ is even.

Next, we will consider the case where $n$ is odd.
If $k=0$, then we can decompose the isocrystal $V$ as $V_{0} \oplus V_{1}$ as in the case where $n$ is even, and we can proceed with the same proof.
Therefore, we may assume that $k=1$. 
We put $n:= 2n'+1$.
In this case, we can decompose the isocrystal $V$ as
\[
V= V_{0} \oplus V_{\frac{1}{2}} \oplus V_{1},
\]
where $V_{0}$ is spanned by $e_{1}, \ldots e_{n-1},$ $V_{\frac{1}{2}}$ is spanned by $e_{n}, e_{n+1}$, and $V_{1}$ is spanned by $e_{n+2}, \ldots e_{2n}$.
As in the even case, we can find $j_{0} \in J_{b_{\mathrm{sp},0}}^{(\GL_{2{\cred n'}})}$ such that
\[
j_{0} p_{0} (v) \in \mathcal{L}_{b_{\mathrm{sp},0}}^{\mathrm{adm}}.
\]
Then we can show that
\[
j
:=
\left(\begin{array}{c|cc|c}
j_{0} &&& \\ \hline
&1&& \\
&&1& \\ \hline
& && \omega ^{t}j_{0}^{-1} \omega
\end{array}\right)
\]
is contained in $J_{\bsp} (K)$.

Moreover, by multiplying $j$ from the left by some power of 
\[
\left(\begin{array}{cccccccc}
1 & &&&&&& \\ 
& \ddots & &&&&& \\
&&1 &&&&& \\
&&&&-\varpi&&& \\
&&&1&&&&\\
&&&&& \varpi &&\\
&&&&&& \ddots & \\
&&&&&&& \varpi
\end{array}\right),
\]
we have the following.

\begin{itemize}
\item
$j \in J_{\bsp} (K).$
\item
$p_{0} (j v) \in \mathcal{L}_{b_{\mathrm{sp},0}}^{\mathrm{adm}}.$
\item
$\langle jv, F^{n} (jv) \rangle \in \varpi^{m} \co^{\times}.$
\end{itemize}
By \cite[Proposition 4.3]{Viehmann2008a}, we may further assume the following;
\[
A \mathcal{D}_{k} jv_{\frac{1}{2}} =\mathcal{D}_{k} A jv_{\frac{1}{2}},
\]
where $A \in \mathcal{D}_{k}$ is a generator of $\mathrm{Ann} (jv_{0})$ as in \cite[Lemma 2.6]{Viehmann2008a}.
By Lemma \ref{lemma:hogehoge}, it suffices to show the following claim.
\begin{claim}
\label{claim:findj2}
Let  $v, \widetilde{v} \in \Vsympsp$ be elements satisfying the following.
\begin{enumerate}
\item
$p_{0}(v) = p_{0} (\widetilde{v})$.
\item
$p_{\frac{1}{2}}(v) = p_{\frac{1}{2}}(\widetilde{v})$.
\end{enumerate}
Then there exists $j\in J_{{\cred \bsp}} (K)$ such that $jv = \widetilde{v}$
\end{claim}
This claim can be proved in the same way as in Claim \ref{claim:findj}.
\end{proof}

To complete the proof, we need the following lemma.
\begin{lemma}
\label{lemma:hogehoge}
We put $F:= \bsp \sigma$ as above.
\begin{enumerate}
\item
Suppose that $n$ is even, or $n$ is odd and $k=0$.
Let $v_{0} \in \mathcal{L}_{b_{\mathrm{sp},0}}^{\mathrm{adm}}$ and $c \in \varpi^{m} \co^{\times}$.
Then there exists $v \in \Lsympsp$ such that $p_{0}(v) = v_{0}$ and 
\[
\langle v, F^{n}(v) \rangle = c. 
\]
\item
Suppose that $n$ is odd and $k=1$.
Let $v_{0} \in \mathcal{L}_{b_{\mathrm{sp},0}}^{\mathrm{adm}}$ and $v_{\frac{1}{2}} \in \mathcal{L}_{b_{\mathrm{sp}},\frac{1}{2}}^{\mathrm{adm}}$.
We further suppose that 
\[
A \mathcal{D}_{{\cred k}} v_{\frac{1}{2}} = \mathcal{D}_{{\cred k}} A v_{\frac{1}{2}},
\]
where we use the same notation as in \cite[Proposition 4.3 (2)]{Viehmann2008a}.
Then there exists $v \in \Lsympsp$ such that $p_{0}(v)=v_{0}$ and $p_{\frac{1}{2}}(v) =v_{\frac{1}{2}}$.
\end{enumerate}
\end{lemma}

\begin{proof}
The proof is essentially done by \cite[Proposition 5.3]{Viehmann2008a}.
First, we will prove the case $(1)$.
For simplicity, we suppose that $k=1$ (the case where $k=0$ can be proved by a similar argument). We put $n':= \frac{n}{2}$.
Let $v_{1} \in V_{1}$ be an arbitrary element.
We put
\[
\xi_{i} := \langle v_{1}, F^{i}(v_{0}) \rangle
\]
for $i=1, \ldots, n$.
Then we have
\begin{eqnarray*}
(\xi_{1}, \ldots, \xi_{n})
\diag (1, 1, \varpi^{-1}, \varpi^{-1}, \ldots, \varpi^{-(n'-1)}, \varpi^{-(n'-1)})
(\sigma (g_{b_{\mathrm{sp},0}}^{\mathrm{red}}(v_{0})))^{-1} \\
= (a_{2n-1}, -\varpi a_{2n}, \ldots, a_{n+1}, - \varpi a_{n+2}),
\end{eqnarray*}
where $g_{b_{\mathrm{sp},0}}^{\mathrm{red}}$ is defined as in \cite[Subsection 6.3]{Chan2021}, and we put
\[
v = ^{t}\!(a_{1}, \ldots, a_{2n}). 
\]
Therefore, the vector $v_{1}$ is determined uniquely by $\xi_{1}, \ldots, \xi_{n}$.
We define $q_{ij} \in \breve{K}$ by
\[
F^{-i} (v_{0}) = \sum_{j=1}^{n} q_{ij} F^{j} (v_{0})
\]
for $i=1, \ldots, n$. 
We will find an element $v$ of the form $v_{0} + v_{1}$.
We note that 
\[
\langle v, F^{i} (v) \rangle = \langle v_{0}+v_{1}, F^{i} (v_{0}+v_{1}) \rangle
= \xi_{i} + (-1)^{i} \varpi^{i} \sigma^{i} (\langle F^{-i} (v_{0}), v_{1} \rangle ) = \xi_{i} - \sum_{j=1}^{n} (-1)^{i} \varpi^{i} \sigma^{i}(q_{ij}) \sigma^{i} (\xi_{j}).
\]
Therefore, to give an element $v \in \Lsympsp$ satisfying the condition in (1), it suffices to give $(\xi_{1}, \ldots \xi_{n})$ satisfying the following.
\begin{itemize}
\item
\[
\ord \xi_{1} \geq 1, \ord \xi_{2} \geq 1, \ord \xi_{3} \geq 2,  \ord \xi_{4} \geq 2, \ldots, \ord \xi_{n-1} \geq n', \ord \xi_{n} \geq n'.
\]
\item
\[
\xi_{i} - \sum_{j=1}^{n} (-1)^{i}  \varpi^{i} \sigma^{i}(q_{ij}) \sigma^{i} (\xi_{j})
\begin{cases}
= 0 & (i=1, \ldots, n-1), \\
=c  & (i=n).
\end{cases}
\]
\end{itemize}
Here, we note that we have $\ord q_{ij} \geq \frac{i+j}{2}$ by \cite[Lemma 5.2.4]{Kedlaya2005}.
We can find $(\xi_{1}, \ldots \xi_{n})$ as above by solving the equation by the same argument as  in \cite[Lemma 5.4]{Viehmann2008a}.

Next, we will prove (2).
We put $n':= \frac{n-1}{2}$. 
We note that it suffices to find $v \in \Vsympsp \cap \mathcal{L}$ such that $p_{0}(v) = v_{0}$ and $p_{\frac{1}{2}} (v) = v_{\frac{1}{2}}$, since we have $v \in \Lsympsp$ in this case by \cite[Lemma 4.3]{Viehmann2008a} and Remark \ref{remark:lattice}. 
We put
\[
\xi_{i} := \langle v_{1}, F^{i} (v_{0}) \rangle.
\]
We also put
\[
h_{i} := \langle v_{\frac{1}{2}}, F^{i} (v_{\frac{1}{2}}) \rangle
\]
for $i=1, \ldots n-1$.
As before, we can put
\[
F^{-i} (v_{0}) = \sum_{j=1}^{n-1} q_{ij} F^{i} (v_{0})
\]
with $ \ord q_{ij} \geq \frac{i+j}{2}$.
If we put $v= v_{0}+v_{\frac{1}{2}}+v_{1}$, then we have
\[
\langle v, F^{i} (v) \rangle = \xi_{i} - \sum_{j=1}^{n-1} (-1)^{i} \varpi^{i} \sigma^{i} (q_{ij}) \sigma^{i} (\xi_{j}) + h_{i}.
\]
Therefore, to give a desired element $v$, it suffices to give $(\xi_{1}, \ldots, \xi_{n-1})$ satisfying the following.
\begin{itemize}
\item
\[
\ord \xi_{1} \geq 0, \ord \xi_{2} \geq 1, \ord \xi_{3} \geq 1, \ord\xi_{4} \geq 2, \ldots, \ord \xi_{n-2} \geq n'-1, \ord \xi_{n-1} \geq n'.
\]
\item
\[
\xi_{i} - \sum_{j=1}^{n-1} \varpi^{i} \sigma^{i} (q_{ij}) \sigma^{i} (\xi_{j}) = - h_{i}.
\]
\end{itemize}
Since we have $\ord h_{i} \geq \frac{i}{2}$, we can solve this equation by the same argument as in \cite[Lemma 5.4]{Viehmann2008a}.
{\cred This} finishes the proof.
\end{proof}

In the following, we suppose that $b = b_{\mathrm{sp}}$ is the special representative (by Proposition \ref{prop:component} (3), it is not essential).
By Proposition \ref{prop:component}, we can describe a connected component of affine Deligne--Lusztig varieties.

\begin{definition}
\label{defn:component of adlv}
For any $j \in J_{b}(K)/ J_{b}(\co_{K})$, we put
\[
X^{m}_{w_{r}}(b)_{j (\mathcal{L})} := g_{b,r} (j(\Lsymp) / \sim_{b,m,r}),
\]
which is a subset of $X^{m}_{w_{r}}(b)(\overline{k})$.
Here, we use the same notation as in Theorem \ref{theorem:comparison}.
We will equip this subset with a scheme structure in Proposition \ref{proposition:openclosed}.
We also define
\[
\dot{X}^{m}_{w_{r}}(b)_{j (\mathcal{L})} := g_{b,r} (j(\Lsymprat) / \dot{\sim}_{b,m,r}),
\]
which is a subset of {\cred $\dot{X}^{m}_{w_{r}}(b) (\overline{k}).$}
\end{definition}

The following is an analogue of \cite[Proposition 6.12]{Chan2021}.

\begin{proposition}
\label{proposition:openclosed}
Assume that $r+k   \geq 1$.
For any $j\in J_{b}(K)$, the set $X^{0}_{w_{r}}(b)_{j (\mathcal{L})}$ is an open and closed subset of $X^{0}_{w_{r}}(b)(\overline{k})$.
Therefore, we can equip $X^{0}_{w_{r}}(b)_{j (\mathcal{L})}$ with the scheme structure such that 
\[
X^{0}_{w_{r}}(b) = \bigcup_{j \in J_{b}(K)/ J_{b}(\co_{K})} X^{0}_{w_{r}}(b)_{j (\mathcal{L})}
\]
is a scheme-theoretic disjoint union.
\end{proposition}

\begin{proof}
We will show that there exists a constant $C$ which depends only {\cred on} $n,k,r$ such that
\begin{equation}
\label{eq:latticedescription}
X^{0}_{w_{r}}(b)_{\mathcal{L}}
=
\{
(\mathcal{M}_{i})^{2n-1}_{i=0} \in X^{0}_{w_{r}}(b)(\overline{k}) \mid \mathcal{M}_{0} \subset \mathcal{L} \textup{ and } \mathcal{M}_{0}^{\vee} = C^{-1} \mathcal{M}_{0} 
\}.
\end{equation}
More precisely, $C$ is defined by the following equation;
\[
C:= \varpi^{\ord \lambda (g_{b,r}(v)) - \ord \lambda( g_{b}^{\mathrm{red}}(v))},
\]
for some (any) $v \in \Vsymp$.  
We note that this definition of $C$ is equivalent to saying that 
\[
C= \varpi^{\ord \lambda( g_{b,r}(v) )} = \varpi^{\lfloor \frac{n}{2} \rfloor},
\]
for some (any) $v \in \Lsymp$.
We also note that, here, we regard {\cred an} element in 
\[
X^{0}_{w_{r}}(b)(\overline{k}) \subset G(\breve{K})/ I
\]
 as a lattice chain which is self-dual up to {\cred a} constant, via the correspondence
\[
g I \mapsto g(\la_{0}, \la_{1}, \ldots, \la_{2n-1}),
\]
where $\la_{0} \supset \cdots \supset \la_{2n-1}$ is a standard lattice chain, i.e., we put
\[
\la_{i} = \bigoplus_{1 \leq j \leq 2n-i} 
\co e_{j} \oplus 
\bigoplus_{2n-i+1 \leq j \leq 2n} 
\p e_{j}.
\]
We note that $\la_{0}$ is possibly different from $\la$ in contrast to \cite[Proposition 6.12]{Chan2021}.
{\cred We now prove} (\ref{eq:latticedescription}).
First, we take an element $g_{b,r} (v) I$ ($v \in \Lsymp$) from the left-hand side.
Then the corresponding lattice $\mathcal{M}_{0}$ is generated by $w_{1}, \ldots, w_{2n}$, where $w_{i}$ is the $i$-th column vector of $g_{b,r} (v)$.
Therefore, by definition of $g_{b,r}$ and $\Lsymp$, we have 
\[
\mathcal{M}_{0} \subset \mathcal{L}.
\]
Moreover, by definition of $C$ and the fact that $\ord( \lambda (g_{b}^{\mathrm{red}}(v)) = 0$ for $v \in \Lsymp$, we have
$\mathcal{M}_{0}^{\vee} = C^{-1}\mathcal{M}_{0}$.

On the other hand, we take an element $(\mathcal{M}_{i})_{i=0}^{2n-1}$ from the right-hand side of (\ref{eq:latticedescription}).
By Theorem \ref{theorem:comparison}, there exists $v\in \Vsymp$ such that
$\mathcal{M}_{i} = g_{b,r}(v) \mathcal{L}_{i}$.
Since $(g_{b,r} (v) \mathcal{L}_{0})^{\vee} = \lambda(g_{b,r} (v))^{-1} g_{b,r} (v) \mathcal{L}_{0}$,
we have
\[
\ord (\lambda(g_{b,r}(v))) =  \ord (\lambda(g_{b,r}(v)))- \ord (\lambda (g_{b}^{\mathrm{red}}(v))).
\]
Therefore, we have $v \in \Lsymp$. {\cred This} finishes the proof of (\ref{eq:latticedescription}).

Now we can prove the desired statement by the same argument as in the proof of \cite[Proposition 6.12]{Chan2021}.
\end{proof}

We define
\[
n_{0}:=
\begin{cases}
1  & \textup{if } k=0,\\
2  & \textup{if } k=1.
\end{cases}
\]
We also put $n' :=\lfloor n/ n_{0} \rfloor$ as before.

We will study the structure of affine Deligne--Lusztig varieties in this case as an analogue of \cite[Theorem 6.17]{Chan2021}.

First, we give a short remark.

\begin{remark}
\label{remark:linindep}
\begin{enumerate}
\item
Suppose that $n$ is even and $k=1$.
For any $v=^{t}\!(v_{1}, \ldots, v_{2n})\in \Lsymp$, we put
\[
v_{i,0} := v_{i} \mod \p
\]
for $1 \leq i \leq 2n$.
Note that we have $v_{n+1,0} = \cdots = v_{2n-1,0} =0.$
We define the matrix $A= (a_{i,j})_{1\leq i,j \leq n} \in M_{n} (\overline{k})$ by
\[
a_{i,j} := 
\textup{the image of }
\begin{cases}
b_{2i-1,2j-1}  &(\textup{if } 1 \leq i \leq n' \textup{ and } 1 \leq j \leq n'),\\
b_{2i-1,2j} & (\textup{if } 1\leq i \leq n' \textup{ and } n'+1 \leq j \leq 2n'),\\
b_{2i. 2j-1} &(\textup{if } n'+1 \leq i \leq 2n' \textup{ and } 1 \leq j \leq n'),\\
b_{2i, 2j} &(\textup{if } n'+1 \leq i \leq 2n' \textup{ and } n'+1 \leq j \leq 2n'),
\end{cases}
\]
i.e., the submatrix corresponding to the 1st, 3rd, \ldots, $(n-1)$-th, $(n+2)$-th, \ldots, $(2n)$-th rows and columns.
Here, we put $g_{b}^{\mathrm{red}} (v) = (b_{i,j})_{1\leq i,j \leq 2n}$
Then we have $A\in \GSp_{2n}(\overline{\F}_{q})$ with respect to the symplectic form associated with
\[
\overline{\Omega}
:=
\left(\begin{array}{ccc|ccc}
&&&&&1\\
&\hsymb{0}&&&\iddots&\\
&&&1&&\\\hline
&&-1&&&\\
&\iddots&&&&\\
-1&&&&\hsymb{0}&
\end{array}\right).
\]
By definition of $g_{b}^{\mathrm{red}}$, we have the following.
\begin{itemize}
\item
For $1 \leq j\leq n'$, we have
\[
(a_{1j}, \ldots, a_{nj}) = 
(v_{1,0}^{q^{2j}}, v_{3,0}^{q^{2j}}, \ldots, v_{n-1,0}^{q^{2j}},
v_{n+2,0}^{q^{2j}}, v_{n+4,0}^{q^{2j}}, \ldots, v_{2n,0}^{q^{2j}}).
\]
\item
For $n'+1 \leq j \leq 2n'$, the vector
$(a_{1j}, \ldots, a_{nj})$ is a $\overline{k}$-linear combination of 
\[
(v_{1,0}^{q^{-2k}}, v_{3,0}^{q^{-2k}}, \ldots, v_{n-1,0}^{q^{-2k}}, v_{n+2,0}^{q^{-2k}}, v_{n+4,0}^{q^{-2k}}, \ldots, v_{2n,0}^{q^{-2k}})_{1 \leq k \leq j-n'}
\]
and
\[
(a_{11}, \ldots, a_{n1}).
\]
\end{itemize}
Therefore, by \cite[Subsection 6.3]{Chan2021}, we can show that
\[
(v_{1,0}, v_{3,0} \ldots, v_{n-1,0}, v_{n+2,0}, v_{n+4,0} \ldots, v_{2n,0})
\]
 are linearly independent over $\F_{q^2}$.
 \item
 Suppose that $n$ is odd and $k=1$.
For any $v=^{t}\!(v_{1}, \ldots, v_{2n})\in \Lsymp$, we define the matrices
$A = (a_{i,j})_{1\leq i,j \leq n}, A'= (a_{i,j}')_{1\leq i, j \leq n}  \in M_{n} (\overline{k})$ by
\begin{eqnarray*}
a_{i,j} &:=& \textup{the image of } b_{2i-1,2j-1},\\
a_{i,j}' &:=& \textup{the image of } b_{2i,2j}.
\end{eqnarray*}
Here, we put $g_{b}^{\mathrm{red}}(v) = (b_{i,j})_{1\leq i,j\leq 2n}$
Then we have 
\[
^{t}\!A' H A = \overline{\lambda(v)} H,
\]
where we put
\[
H:= 
\left(\begin{array}{ccc}
&&1\\
&\iddots&\\
1&&
\end{array}\right).
\]
Since we have $A \in \GL_{n}(\overline{k})$,
we can show that
\[
(v_{1,0}, v_{3,0}, \ldots, v_{2n-1,0})
\]
are linearly independent over $\F_{q^2}$ by the same argument as in (1).
 
\item
Suppose that $k=0$.
We can show that
the matrix $A \in M_{2n} (\overline{k})$ defined as the image of $g_{b}^{\mathrm{red}}(v)$ for $v\in \Lsymp$ is contained in $\GSp_{2n}(\overline{\F}_{q})$ with respect to the symplectic form associated with $\Omega$.

Therefore, by the same argument as in (1) again,
\[
(v_{1,0}, v_{2,0}, \ldots, v_{2n,0})
\]
are linearly independent over $\F_{q}$
for any $v=^{t}\!(v_{1}, \ldots, v_{2n})\in \Lsymp$.
\end{enumerate}
\end{remark}

\begin{definition}
\label{defn:tau}
\begin{enumerate}
\item
Suppose that $k=1$.
We define the permutation $\tau : \{1, \ldots, 2n\} \rightarrow \{ 1, \ldots, 2n\}$ as follows.
\[
\tau(i) =
\begin{cases}
\lceil \dfrac{i}{2} \rceil & \textup{if } 1\leq i \leq n \textup{ and } \lfloor \dfrac{i}{2} \rfloor \textup{ is even},\\
n+ \lceil \dfrac{i}{2} \rceil & \textup{if } 1\leq i \leq n \textup{ and } \lfloor \dfrac{i}{2} \rfloor \textup{ is odd},\\
2n+1- \tau(2n+1-i) & \textup{if } n+1 \leq i \leq 2n.
\end{cases}
\]
We also define the function $\phi_{r} \colon \{1, \ldots, 2n\} \rightarrow K$ by
\[
\phi_{r}(i) =
\begin{cases} 
(-1)^{\tau(i)-i} \varpi^{(2\lceil \frac{i}{4} \rceil -1)r + (\lceil \frac{i}{4} \rceil -1)}  & \textup{if } 1\leq i \leq n \textup{ and } i \textup{ is even}, \\
(-1)^{\tau(i)-i}
\varpi^{(2\lceil \frac{i-1}{4} \rceil)r + \lceil \frac{i-1}{4} \rceil}  & \textup{if } 1\leq i \leq n \textup{ and } i \textup{ is odd},\\
\varpi^{nr+n'- \ord \phi_{r} (2n+1-i)} & \textup{if } n+1 \leq i \leq 2n.
\end{cases}
\]
Let $E_{i,j} \in \GL (V)$ be the matrix whose entries are $0$ except that the $(i,j)$-th entry is $1$.
We put
$
x_{r} := \sum_{i=1}^{2n}  \phi_{r}(i) E_{i, \tau(i)},
$
which is an element of $\GSp (V)$ by definition. 
Note that the order of each non-zero entry of $x_{r}$ is the same as the order of the entry lying {\cred in the same position as that of} $g_{b,r} (v)$ for $v \in \Lsymp$.

We also put
\[
\varphi_{r} (i) := \ord \phi_{r} (i).
\]
\item
Suppose that $k=0$.
We define
\[
x_{r}:= \diag(1, \varpi^{r}, \ldots, \varpi^{r(n-1)},  \varpi^{r}, \ldots, \varpi^{rn}) \in \GSp(V).
\]
As in (1), we define $\tau$ and $\phi_{r}$ by the formula
$
x_{r} = \sum_{i=1}^{2n} \phi_{r} (i) E_{i,\tau(i)}.
$
Moreover, we put 
\[
\varphi_{r} (i) := \ord \phi_{r} (i).
\]

\end{enumerate}
\end{definition}

The following is an analogue of \cite[Proposition 6.15]{Chan2021}.

\begin{proposition}
\label{prop:cell containment}
Suppose that $r + k \geq m +1$.
Then 
 $X^{m}_{w_{r}}(b)_{\mathcal{L}}$
is contained in the higher Schubert cell $Ix_{r} I / I^{m} \subset \GSp(V)/I^{m}$.
\end{proposition}

\begin{proof}
We only prove the case where $n$ is even and $k=1$ (other cases can be proved in the same way). 
We may assume that $m=0$.
Take an element  $g_{b,r} (v) I \in X^{0}_{w_{r}}(b)_{\mathcal{L}}$, where $v\in \Lsymp$.
We will show that $g_{b,r}(v) I \subset I x_{r} I$.
First, we will modify $g_{b,r}(v) = (g_{i,j})_{1 \leq i,j \leq 2n}$ by multiplying by elements of $I$ from the left-side.
We will modify {\cred the} $\tau(1), \ldots, \tau(n)$-th columns in this order.
More precisely, for the $\tau(j)$-th columns ($1 \leq j \leq n$),  we will use the row elementary transformations $\row_{j \to  i}$ to eliminate the $(i. \tau(j))$-th entries of $g_{b,r}(v)$ where
\[
2n+1- j \geq  i> j.
\]
{\cred Here, the order of elimination is as follows:}
 \[
 (2,\tau(1)), \ldots, (2n, \tau (1)), (3, \tau (2)), \ldots, (2n-1, \tau(2)), \ldots.
 \]
 A priori, it is non-trivial that we can proceed with such modifications. The main problem is the order of entries: Since we should use elementary transformations corresponding to elements in $I$, we should know the order of the $(j, \tau(j))$-th entries to execute the above modification in each step. 
 Such a problem is settled in the next paragraph.
 Once we can overcome such a problem, then the above modification is realized by the multiplication from the left by elements in $I$, and one can show that the above modifications do not affect already modified entries:
$\row_{j \to  i}$ itself does not affect {\cred the entries we have set to zero} since $(j, \tau(l))$-th entries for $1 \leq l \leq i-1$ is already {\cred zero}, and the counterpart $\row_{2n+1-i \to 2n+1-j}$ does not affect {\cred the entries we have set to zero} clearly since the elimination of $(2n+1-j,\tau(j))$-th entries takes place in the last step dealing with the $\tau(i)$-th column.  
Note that, if we {\cred could eliminate} above entries for $j=1, \ldots, \tau (l)$, then the $(i,j)$-entries with 
\begin{equation}
\label{eq:automvanish}
i > \max(2n-l, \tau^{-1} (j))
\end{equation}
are automatically {\cred eliminated} since our modifications are made in $\GSp (\breve{K})$.
In particular, after such a modification, the $(i,j)$-th entries for $ i > \tau^{-1} (j)$ and $j= \tau(1), \ldots, \tau (l)$ {\cred have been eliminated}.

In the following, we will verify that such modifications work.
By the definition of $g_{b,r}(v)$, the $(i,\tau(i))$-th entry of $g_{b,r}(v)$ equals to $\varphi_{r}(i)$ and the transformation modifying the $\tau(i)$-th column clearly preserves the order of the $(i+l, \tau(i+l))$-th entry for odd $l$.
If the order of the $(i+2, \tau (i+2))$-th entry coincides with the order of $\varphi_{r} (i+2)$ 
after modifying $\tau(i)$-th column,
then we can proceed with the elimination.
(Here, we suppose that $i \leq n-2$. Otherwise, we can proceed with the elimination without doing anything. In particular, this process is already done when $n=2$.)
We suppose that $i=2i'+1$ is odd (the even case can be proved in the same way).
Consider the matrix $\bar{g} = (\bar{g}_{s,t}) \in M_{n} (\overline{k})$ given by
\[
\bar{g}_{s,t} =
\begin{cases}
\textup{ the image of }   
\phi_{r}(t)^{-1}
g_{2s-1, \tau(t)}, &\textup{if } 1 \leq s \leq n', \\
\textup{ the image of }  
\phi_{r}(t)^{-1}
g_{2s, \tau(t)} & \textup{if } n'+1 \leq s \leq n.
\end{cases}
\]
Then by the same argument as in \cite[Lemma 6.16]{Chan2021}, we can show that the determinant of the {\cred upper-left} $(s \times s)$-minor is non-zero.
For $1 \leq i', j' \leq n$, the row elementary transformation \r{2i'-1}{2j'-1} on $g_{b,r}(v)$ corresponds to the row elementary transformation \r{i'}{j'} on $\bar{g}$ and such transformations preserves the determinant of the {\cred upper-left} $(i \times i)$-minor of $\overline{g}$ for $1 \leq i \leq n$.
Similarly, one can show that all the row elementary transformations in our process do not affect the {\cred upper-left} $(i \times i)$-minor of $\overline{g}$ for $1 \leq i \leq n$.
Therefore, after modifying  $\tau(i)$-th column, the {\cred upper-left} $(\frac{i+3}{2} \times \frac{i+3}{2})$ submatrix becomes upper triangular whose determinant is non-zero, so the $(\frac{i+3}{2}, \frac{i+3}{2})$-th entry is non-zero.
It means that the order of the $(i+2, \tau(i+2))$-th entry is the same as the order of $\varphi_{r} (i+2)$, and we can proceed with the modification.

Now we have successfully modified $\tau(1), \ldots, \tau(n)$-th columns.
Since the resulting matrix $H= (H_{i,j})$ is an element in $\GSp(V)$ (see (\ref{eq:automvanish})), we have
\begin{equation}
\label{eq:Htriangular}
H_{i,j} = 0
\quad \textup{ for } i > \tau^{-1}(j), 1\leq i,j \leq 2n. 
\end{equation}
We also have 
\begin{equation}
\label{eq:Hcontain}
H \in G_{x,0} \cdot \diag(1, 1, \ldots, \varpi^n, \varpi^n) \cdot \diag(1, \varpi^{r}, \ldots, \varpi^{(n-1)r}, \varpi^{r}, \ldots, \varpi^{nr})
\end{equation}
since {\cred $g_{b,r} (v)$ also lies in the right-hand side} and the above row transformations preserve the right-hand-side.

Finally, we will modify $H$ to show $H\in I x_{r} I$.
Basically, we use column elementary transformations to modify each row (more precisely, 1st, $\ldots$, $n$-th row) (note that the problem of the order of $(j, \tau(j))$-th entries as in the last modification does not occur since we already know the condition (\ref{eq:Htriangular}) and (\ref{eq:Hcontain}), and our modification preserves the determinant of matrices).
However, before that, we should treat exceptional entries (located on 1st, $\ldots$, $n$-th row) which cannot be eliminated by column elementary transformations.
Exceptional entries are the $(4l+3, 2l)$-th entries ($1 \leq 4l+3 \leq n$). Note that there is no exceptional entries when $n=2$.
To eliminate these entries, we {\cred use the row} elementary transformations $\row_{4l+4 \to 4l+3}$ since the order of the $(4l+3,2l)$-th entry is greater than the order of the $(4l+4, 2l)$-th entry. Note that such transformations do not affect other $(4l'+3,2l')$-th entries, and preserve the condition (\ref{eq:Htriangular}) and (\ref{eq:Hcontain}).
Let $H'$ be the resulting matrix.
We will modify $1, \ldots, n$-th rows of $H'$ in this order.
In the $i$-th row, if $i \not\equiv 3 \mod 4$ (resp.\,$i \equiv 3 \mod 4$),  first we will eliminate $(i,j)$-th entries for 
\begin{align*}
&j \in \{1, \ldots, 2n \} \setminus \{ \tau(1), \ldots, \tau(i), 2n+1- \tau(1), \ldots 2n+1- \tau (i) \} \\
(\textup{resp.\,}& j \in \{1, \ldots, 2n \} \setminus \{ \tau(1), \ldots, \tau(i), 2n+1- \tau(1), \ldots 2n+1- \tau (i), \dfrac{i-3}{2}\} ).
\end{align*}
by using $\col_{\tau(i) \to j}$.
Obviously, $\col_{\tau(i) \to j}$ itself only affects the $(i,j)$-th entry. On the other hand, the counterpart  $\col_{2n+1-j \to 2n+1- \tau(i)}$ does not affect already modified entries, since $(i, 2n+1- \tau(i))$ is not modified yet.

After that, we modify the $(i, 2n+1- \tau(i))$-th entry by using $\col_{\tau(i) \to 2n+1-\tau(i)}$, which only affects the $(i, 2n+1- \tau(i))$-th entry.  Finally, by multiplying the diagonal matrices in $I$ from the right (or left), we {\cred obtain} $x_{r}$.
{\cred This} finishes the proof of $g_{b,r} (x) \in I x_{r} I$.
\begin{figure}
\centering
\captionsetup{justification=centering}
\[
\left(\begin{array}{cccc|cccc}
\fbox{$\co$}&\p^{1+r}&\p^{1+2r}&\p^{2+3r}&\p^{1+r}&\p^{1+2r}&\p^{2+3r}&\p^{2+4r}\\
\co&\p^{r}&\p^{1+2r}&\p^{1+3r}&\fbox{$\p^{r}$}&\p^{1+2r}&\p^{1+3r}&\p^{2+4r}\\
\co&\uwave{\p^{1+r}}&\p^{1+2r}&\p^{2+3r}&\p^{1+r}&\fbox{$\p^{1+2r}$}&\p^{2+3r}&\p^{2+4r}\\
\co&\fbox{$\p^{r}$}&\p^{1+2r}&\p^{1+3r}&\p^{r}&\p^{1+2r}&\p^{1+3r}&\p^{2+4r}\\\hline
\p&\p^{1+r}&\p^{2+2r}&\p^{2+3r}&\p^{1+r}&\p^{2+2r}&\fbox{$\p^{2+3r}$}&\p^{3+4r}\\
\co&\p^{1+r}&\fbox{$\p^{1+2r}$}&\p^{2+3r}&\p^{1+r}&\p^{1+2r}&\p^{2+3r}&\p^{2+4r}\\
\p&\p^{1+r}&\p^{2+2r}&\fbox{$\p^{2+3r}$}&\p^{1+r}&\p^{2+2r}&\p^{2+3r}&\p^{3+4r}\\
\co&\p^{1+r}&\p^{1+2r}&\p^{2+3r}&\p^{1+r}&\p^{1+2r}&\p^{2+3r}&\fbox{$\p^{2+4r}$}
\end{array}\right).
\]
\caption*{The shape of the right-hand side of (\ref{eq:Hcontain}) for $\GSp_{8}$\\
The $(i, \tau (i))$-th entries for $1 \leq i \leq 8$ are enclosed in a square.\\
The exceptional entry is marked with a wave line.
}
\end{figure}
\end{proof}

\subsection{Structure of  $\Lsymp$, $n:$ even case}
\label{subsection:structureofLsymp_even}
In this subsection, we suppose that $n$ is even.

\begin{definition}
\label{def:findlv}
We put
\[
\overline{V}:= 
\begin{cases}
\mathcal{L} / \varpi \mathcal{L} & \textup{if } k=0,\\
\mathcal{L} / F(\mathcal{L}) & \textup{if } k=1.
\end{cases}
\]
There exists the natural symplectic form $\langle , \rangle$ on $\overline{V}$ (cf.\ Remark \ref{remark:linindep}). 
Here, $F\colon \overline{V} \rightarrow \overline{V}$ is the Frobenius morphism over $\F_{q^{2}}$ (i.e., the $q^{2}$-th power map).
Note that
the image of 
\[
e_{i} 
\begin{cases}
(i \equiv 1 \mod n_{0})  &(\textup{if } 1 \leq i \leq n) \\
(i \equiv 0 \mod n_{0})  &(\textup{if } n+1 \leq  i \leq 2n)
\end{cases}
\]
form a {\cred basis} of $\overline{V}$.
In the following, we regard $\overline{V}$ as $\overline{\F}_{q}^{\oplus n'}$ via the above basis.

{\cred Moreover}, we define
$\overline{w} \in \GSp({\overline{V}})$ by
\[
\overline{w} = 
\left(\begin{array}{cccc|cccc}
&&&&-1&&&\\
1&&\hsymb{0}&&&&&\\
&\ddots&&&&\hsymb{0}&&\\
&&1&&&&&\\\hline
&&&&&1&&\\
&&&&&&\ddots&\\
&\hsymb{0}&&&\hsymb{0}&&&1\\
&&&1&&&&
\end{array}\right).
\]
We also denote the upper-half Borel subgroup of $\GSp(\overline{V})$ by $\overline{B}$, and its unipotent radical by $\overline{U}$.
Note that $\overline{B}$ and $\overline{U}$ {\cred are} $\overline{\sigma}$-stable, where $\overline{\sigma}$ is the Frobenius morphism over $\F_{q^{n_{0}}}$.

We put
\[
\overline{\mathcal{L}_{b}^{\mathrm{symp}}}
 :=
\{
v \in \overline{V} \mid
\langle v, \overline{\sigma}^{i}(v) \rangle = 0 \, (1\leq i \leq n'-1), \langle v, \overline{\sigma}^{n'} (v) \rangle \neq 0
\}.
\]
\end{definition}

\begin{lemma}
\label{lemma:finitecomparison}
We define the equivalence relation $\sim$ on 
$\overline{\mathcal{L}_{b}^{\mathrm{symp}}}$
by 
\[
v\sim v' \Leftrightarrow v \in  \overline{\F}_{q}^{\times}\cdot v'.
\]
Then we have an isomorphism of schemes over $\overline{\F}_{q}$
\begin{equation}
\label{eqn:cldlv}
\overline{\mathcal{L}_{b}^{\mathrm{symp}}}
/ \sim \,
\simeq 
X_{\overline{w}}^{(\overline{B})}.
\end{equation}
Here, we equip the left-hand side of (\ref{eqn:cldlv}) with the scheme structure by the locally closed embedding 
\[
\overline{\mathcal{L}_{b}^{\mathrm{symp}}}
 / \sim \,
\hookrightarrow \mathbb{P}(\overline{V}),
\] 
and the right-hand side of (\ref{eqn:cldlv}) is the classical Deligne--Lusztig variety for $\GSp_{2n'}$ over $\F_{q^n_{0}}$.
Note that $\Lsympbar/ \sim $ is naturally identified with the subset of $\Lsympbar$ consisting of $(v_{1}, \ldots, v_{2n}) \in \Lsympbar$ with $v_{1} =1.$
We denote this subset by $\PP(\Lsympbar)$.
\end{lemma}

\begin{proof}
We only prove the case where $k=1$ (the case where $k=0$ can be proved in the same way).
We define the map $\overline{g} \colon \overline{\mathcal{L}_{b}^{\mathrm{symp}}} \rightarrow \GSp( \overline{V})$ as follows. We put
\[
\overline{G}_{1}(v) := 
- \frac{\alpha_{v}}{\overline{\sigma}^{-1}(\alpha_{v})} (\overline{\sigma}^{-1}(v)- 
\frac{\langle \overline{\sigma}^{-1} (v), \overline{\sigma}^{n'}(v) \rangle}{\alpha_{v}}v
),
\]
where we put $\alpha_{v}:= \langle v, \overline{\sigma}^{n'}(v) \rangle$.
We also put
\[
\overline{G}_{i+1}(v):= \frac{\alpha_{v}}{\overline{\sigma}^{-1}(\alpha_{v})} (\overline{\sigma}^{-1}(\overline{G}_{i}(v))-\frac{\langle \overline{\sigma}^{-1}(\overline{G}_{i}(v)), \overline{\sigma}^{n'}(v)\rangle}{\alpha_{v}}v)\,\,\,\,\, (i=1,n+1,\ldots,n-2).
\]
Then for any $v \in \la_{b,1}^{\mathrm{symp}}$, we can define $\overline{g} (v)$ by
\begin{eqnarray*}
\overline{g}(v):=(v, \overline{\sigma}(v),\ldots, \overline{\sigma}^{n'-2}(v), \overline{\sigma}^{n'-1}(v), \overline{G}_{1}(v), \overline{G}_{2}(v),\ldots, \overline{G}_{n'-1}(v), \overline{\sigma}^{n'}(v)).
\end{eqnarray*}
By the same proof as in Theorem \ref{theorem:comparison} (1), $\overline{g}(v)$ induces the desired isomorphism. By definition of $\overline{g}$, this is an isomorphism of schemes.
\end{proof}

\begin{definition}
\label{definition:lh}
\begin{enumerate}
\item
Suppose that $k=0$.
We put
\[
\la_{h} := (\co / \p^{h})^{\oplus 2n},
\]
which is a quotient of $\la$.
We also put
\[
\la_{b,h}^{\mathrm{symp}}
:= \left\{ v=(\overline{v_{1}}, \ldots \overline{v_{2n}}) \in \la_{h}
\left|
\begin{array}{l}
\langle v, F^{i}(v) \rangle = 0 \mod \p^{h}\quad (1 \leq i\leq n-1)\\
\langle v, F^{n}(v) \rangle \neq (\co/\p^{h})^{\times}\\
\end{array}
\right.
\right\}.
\]
\item
Suppose that $k=1$.
For any $h \geq 1$, we put 
\[
\la_{h} := (\co / \p^{h})^{\oplus n} \oplus \p / \p^{h+1} \oplus \co / \p^{h} \oplus \cdots \oplus \p / \p^{h+1} \oplus \co / \p^{h},
\]
which is a quotient of $\la$.
We also put
\[
\la_{b,h}^{\mathrm{symp}}
:= \left\{ v=(\overline{v_{1}}, \ldots \overline{v_{2n}}) \in \la_{h} \left|
\begin{array}{l}
\langle v, F^{i}(v) \rangle = 0 \mod \p^{h+\lceil \frac{i}{2} \rceil}\quad (1 \leq i\leq n-1) \\
\langle v, F^{n}(v) \rangle \in \varpi^{n'} (\co/\p^{h+n'})^{\times}\\
\end{array}
\right.
\right\}.
\]
\end{enumerate}
\end{definition}

\begin{remark}
\label{remark:Lh}
The above set $\la_{b,h}^{\mathrm{symp}}$ admits a natural scheme structure of finite type over $\overline{k}$ (resp.\,perfectly of finite type over $\overline{k}$) if $\chara K>0$ (resp.\,$\chara K=0$).
Indeed, such a structure for $\la_{h}$ is given in the same way as the definition of $L_{[r,s)}$ in \cite[p.1815]{Chan2021}.
Moreover, since the symplectic form $\langle v, F^{i}(v) \rangle$ defines a function on a scheme $\la_{h}$,
we can regard $\la_{b,h}^{\mathrm{symp}}$ as a locally closed subscheme of $\la_{h}$ naturally.
We put 
\[
\la'_{h} := \ker (\la_{h} \rightarrow \overline{V}).
\]
Then $\la'_{h}$ also admits an affine space (resp.\,perfection of an affine space) structure.
We can identify $\overline{V} \times_{\overline{k}} \la_{h}'$ (resp.\,$\overline{V}^{\mathrm{perf}} \times_{\overline{\F}_{q}} \la'_{h}$, where $\mathrm{perf}$ is the perfection) with $\la_{h}$ by the natural map
\[
(\overline{v}, w) \mapsto [\overline{v}] +w.
\]
Here, [-] is a lift defined by the same way as in Definition \ref{definition:Lbmr}.(3).
\end{remark}

\begin{lemma}
\label{lemma:affinefibration}
\begin{enumerate}
\item
For any $h \geq 1$, the natural morphism
\[
\la_{b,h}^{\mathrm{symp}} \rightarrow \Lsympbar
\]
is a surjection.
\item
Consider the embedding 
\[
\la_{b,h}^{\mathrm{symp}} \times_{\Lsympbar} \Lsympbar^{\mathrm{perf}}
\subset \la'_{h} \times_{\overline{\F}_{q}} \Lsympbar \times_{\Lsympbar} \Lsympbar^{\mathrm{perf}},
\] 
where the right-hand side is an affine space $\mathbb{A}_{\Lsympbar^{\mathrm{perf}} }$ over $\Lsympbar^{\mathrm{perf}}$ (resp.\,the perfection of an affine space $\mathbb{A}^{\mathrm{perf}}_{\Lsympbar^{\mathrm{perf}}}$ over $\Lsympbar^{\mathrm{perf}})$ if $\chara K > 0$ (resp.\,$\chara K =0$).
Then this embedding is given by the intersection of hyperplane sections of $\mathbb{A}_{\Lsympbar^{\mathrm{perf}}}$ 
(resp.\,$\mathbb{A}^{\mathrm{perf}}_{\Lsympbar^{\mathrm{pref}}}$).
In particular,
$
\la_{b,h}^{\mathrm{symp}} \times_{\Lsympbar} \Lsympbar^{\mathrm{perf}} 
$
is isomorphic to an affine space (resp.\,the perfection of an affine space) over $\Lsympbar^{\mathrm{perf}}$.
\end{enumerate}
\end{lemma}

\begin{proof}
We only prove the case where $k=1$. (the case where $k=0$ can be proved in the same way).
It suffices to show (2).
For simplicity, we suppose that $h=1.$
The general case can be shown similarly, by considering projections
$\la_{h} \rightarrow \la_{h-1}$ inductively (see the proof of Proposition \ref{surjectivity} for the related argument).
We define the coordinates $x_{1,0}, \ldots, x_{2n,0}$ of $\la_{b,1}^{\mathrm{symp}}$ as 
\[
x_{i,0} \textup{ is }
\textup{the image in } \overline{\F}_{q} \textup{ of } 
\begin{cases}
\varpi^{-1} \times \textup{ $i$-th component } &\textup{ if } n+1 \geq i \textup{ and } i \textup{ is odd,}\\
\textup{ $i$-th component } & \textup{ otherwise}. 
\end{cases}
\]
Then $\la_{b,1}^{\mathrm{symp}} \subset \mathbb{A}_{x_{1,0}, x_{2,0} \ldots, x_{2n-1,0}, x_{2n,0}}$ is defined by the following.
{\cred
\begin{equation}
\begin{aligned}
\label{eq:first}
&\langle v, F^{2i-1} (v) \rangle \\ 
=&\; \varpi^{i}\left(
 \sum_{\substack{1 \leq j \leq n \\ j \colon \textup{odd}}} (x_{j,0} x_{2n-j,0}^{q^{2i-1}} + x_{2n-j,0} x_{j,0}^{q^{2i-1}})  - \sum_{\substack{1 \leq j \leq n \\ j\colon \textup{even}}} (x_{j,0} x_{2n-j+2,0}^{q^{2i-1}} + x_{2n-j+2,0} x_{j,0}^{q^{2i-1}})
 \right)\\
=&\;0 \quad \textup{for $i=1, \ldots, n'$.}
\end{aligned}
\end{equation}

\begin{align}
\begin{aligned}
\label{eq:second}
\langle v, F^{2i} (v) \rangle 
=&\; \varpi^{i} \left( \sum_{\substack{1\leq j \leq n \\ j\colon \textup{odd}}} (x_{j,0} x_{2n+1-j,0}^{q^{2i}} - x_{2n+1-j,0} x_{j,0}^{q^{2i}}) \right)\\
=&\;0 \quad \textup{for $i=1, \ldots, n'-1$.}
\end{aligned}
\end{align}
\begin{align}
\begin{aligned}
\label{eq:third}
\langle v, F^{n} (v) \rangle
=&\; \varpi^{n'}
\left(  
\sum_{\substack{1\leq j \leq n, \\ j\colon \textup{odd}}} (x_{j,0} x_{2n+1-j}^{q^{n}} - x_{2n+1-j,0} x_{j,0}^{q^{n}})
\right)\\
\neq&\; 0
\end{aligned}
\end{align}
}
On the other hand,
$\Lsympbar \subset \mathbb{A}_{x_{1,0}, x_{3,0} ,\ldots, x_{n-1,0}, x_{n+2,0}, \ldots, x_{2n,0}}$ is defined by equations (\ref{eq:second}) and (\ref{eq:third}).
Therefore, we should solve the equations (\ref{eq:first}).
We define $v, v' \in W:= \overline{\F}_{q}^{\oplus 2n}$ by
\begin{eqnarray*}
v'&:=&  \sum_{i=1}^{n'} x_{2i-1,0} e_{2i-1} + \sum_{i=n'+1}^{n} x_{2i,0} e_{2i},\\
v''&:=& \sum_{i=1}^{n'} x_{2i,0} e_{2i} + \sum_{i=n'+1}^{n} x_{2i-1,0} e_{2i-1}.
\end{eqnarray*}
We also define the morphism $E \colon W \rightarrow W $ by 
\[
 \overline{\sigma} \circ
 \left(\begin{array}{ccccc}
 \left(
 \begin{array}{cc}
 0&1 \\
 1&0
 \end{array}
 \right)
&& \\
&\ddots & \\
&& 
\left(
 \begin{array}{cc}
 0&1 \\
 1&0
 \end{array}
 \right)
\end{array}\right),
\]
where $\overline{\sigma}$ is the $q$-th power map.
We define the coordinates $\eta_{1}, \ldots, \eta_{n}$ by
\[
\eta_{i} := \langle v'', E^{-(2i-1)} (v') \rangle.
\]
By definition of $\Lsympbar$ (cf.\ Definition \ref{def:findlv}), $\eta_{1}, \ldots, \eta_{n}$ give a linear coordinate transformation of $x_{2,0}, \ldots, x_{n,0}, x_{n+1,0}, \ldots, x_{2n-1,0}$ over $\co_{\Lsympbar}^{\mathrm{perf}}$.
We define the matrix $Q=(q_{i,j})_{1\leq i,j \leq n} \in \GL_{n} (\co_{\Lsympbar}^{\mathrm{perf}})$ by
\[
^t\!(E(v'), \ldots, E^{2n-1}(v')) = Q\, ^t\!(E^{-1}(v'), \ldots, E^{-(2n-1)}(v')).
\]
Note that, to define $Q$, we need to take the perfection.
We will show that there exists 
\[
u= 
\left(\begin{array}{ccc|ccc}
1&&\Lsymb{0}&&&\\
&\ddots&&&\hsymb{0}&\\
\Lsymb{\ast}&&1&&&\\\hline
&&&1&&\\
&\hsymb{0}&&&\ddots&\\
&&&\Lsymb{0}&&1
\end{array}\right) \in \GL_{n}(\co_{\Lsympbar}^{\mathrm{perf}}),
\]
such that $Q' := uQ$ is a matrix whose $(i,j)$-entries $q'_{i,j}$ ($i > n+1-j, j \geq n'$) are $0$.
Since 
\[
(v, F(v), \ldots, F^{2n-1} (v)) \in \GL_{2n}(\co)
\]
for any geometric point of $\Lsympbar$, we have $q_{1,n} \in \co_{\Lsympbar}^{\mathrm{perf} \times}$.
Therefore, by using the row transformation $u^{(1)} = (u_{i,j}^{(1)})$, we can eliminate $q_{2,n}, \ldots, q_{n,n}$. We denote the resulting matrix by $Q^{(1)} = (q_{i,j}^{(1)})$.
We have 
\[
u_{21}^{(1) } E^{1} (v') + E^{3} (v') = \sum_{j=1}^{n-1} q_{2,j}^{(1)} E^{-(2j-1)} (v').
\]
By the same reason as above, $q_{2,n-1}^{(1)} \in \co_{\Lsympbar}^{\mathrm{perf} \times}$.
Repeating these arguments, we have desired transformation $u = u^{(n-1)} \cdots u^{(1)}$.
The equations (\ref{eq:first}) are equivalent to 
\[
\eta_{i}^{q^{2i-1}} + \sum_{j=1}^{n} q_{i,j} \eta_{j} =0
\] 
for $1 \leq i \leq n'$ (cf.\ the arguments in Lemma \ref{lemma:hogehoge}).
By using the $n' \times n$-upper half part $u^{\mathrm{up}}$ of $u$, we can organize the equations as follows: 
\[
(u^{\mathrm{up}} \,
{}^t\!(\eta_{1}^{q}, \ldots, \eta_{n'}^{q^{2n'-1}}))_{i} + \sum_{j=1}^{n+1-i}  q'_{i,j} \eta_{j} =0, 
\]
where $q'_{i, n+1-i} \in \co_{\Lsympbar}^{\times}$ for $1 \leq i\leq n'$. We note that the first term of the left-hand side only contains polynomials of $\eta_{1}, \ldots \eta_{i}$.
Therefore, we can solve the $i$-th equation with respect to $\eta_{n+1-i}$ for $i=1, \ldots, n'$.
{\cred This} finishes the proof.
\end{proof}

\subsection{Structure of $\Lsymp$, $n:$ odd case}
In this subsection, we suppose that $n$ is odd.

\begin{definition}
\label{def:findlvodd}
We put 
\[
\overline{V}:= 
\begin{cases}
\mathcal{L} / \varpi \mathcal{L} & \textup{if } k=0,\\
\mathcal{L} / F(\mathcal{L}) & \textup{if } k=1.
\end{cases}
\]
as before.
Then the image of
\[
e_{i}\,\, (i \equiv 1 \mod n_{0}, 1 \leq i \leq 2n)
\]
form a basis of $\overline{V}$. 
\begin{enumerate}
\item
If $k=0$, then we can equip $\overline{V}$ with the natural symplectic form $\langle , \rangle$ associated with $\Omega$. By the same way as in Definition \ref{def:findlv}, we define 
\[
\overline{w} := 
\left(\begin{array}{cccc|cccc}
&&&&1&&&\\
1&&\hsymb{0}&&&&&\\
&\ddots&&&&\hsymb{0}&&\\
&&1&&&&&\\\hline
&&&&&1&&\\
&&&&&&\ddots&\\
&\hsymb{0}&&&\hsymb{0}&&&1\\
&&&1&&&&
\end{array}\right) \in \GSp (\overline{V}).
\]
We denote the upper-half Borel subgroup of $\GSp(\overline{V})$ by $\overline{B}$, and its unipotent radical by $\overline{U}$.
We denote the Frobenius morphism over $\F_{q}$ by $\overline{\sigma}$. Then $\overline{B}$ and  $\overline{U}$ are $\overline{\sigma}$-stable.
We also put 
\[
\overline{\mathcal{L}_{b}^{\mathrm{symp}}}
 :=
\{
v \in \overline{V} \mid
\langle v, \overline{\sigma}^{i}(v) \rangle = 0 \,\, (1\leq i \leq n'-1), \langle v, \overline{\sigma}^{n'} (v) \rangle \neq 0
\}
\]
as before.

\item
If $k=1$, we equip $\overline{V}$ with a inner product $(,)$ associated with
\[
H:= 
\left(\begin{array}{ccc}
&&1\\
&\iddots&\\
1&&
\end{array}\right).
\]
Since $H$ is a Hermitian matrix, we can consider the algebraic group $\GU_{n}$ over $\F_{q}$, whose $R$-valued point is given by
\[
\GU_{n} (R) := 
\{
(g, \lambda) \in \GL_{n} \times \Gm ( R \otimes_{\F_{q}} \F_{q^{2}}) \mid
^{t}\!c(g) H g = \lambda H
\}
\]
for any $\F_{q}$-algebra $R$ (note that $\lambda \in \Gm (R)$ naturally).
Here, $c(g)$ is the conjugation of $g$ induced by the non-trivial element of the Galois group $\F_{q^{2}}/ \F_{q}$.
For $\F_{q^{2}}$-algebra $R$, we have
\begin{eqnarray*}
\GU_{n} (R) &=&
\{
(g_{1}, g_{2}, \lambda) \in \GL_{n} \times \GL_{n} \times \Gm (R) \mid ^{t}\!(g_{2}, g_{1} ) H (g_{1},g_{2}) = (\lambda, \lambda) H
\}\\
&\simeq& 
\{
(g,\lambda) \in \GL_{n} \times \Gm (R)
\},
\end{eqnarray*}
i.e., we have $\GU_{n,\F_{q^{2}}} \simeq \GL_{n, \F_{q^{2}}} \times \mathbb{G}_{m,\F_{q^{2}}}$.
The last isomorphism is given by $(g_{1}, g_{2}, \lambda) \mapsto (g_{1}, \lambda)$.
We define 
\[
\overline{w}= 
( \left(\begin{array}{cccc|ccc}
&&&1&&&\\
1&&\hsymb{0}&&&&\\
&\ddots&&&&\hsymb{0}&\\
&&1&&&&\\\hline
&&&&1&&\\
&&&&&\ddots&\\
&\hsymb{0}&&&&&1
\end{array}\right), 1)
 \in \GU (\overline{\F}_{q}),
\]
where the grids are lined between the $\lceil \frac{n}{2} \rceil$-th row (resp.\,column) and the $\lceil \frac{n}{2} \rceil +1$-th row (resp.\,column). We use the same rule of grids in this subsection.
We also denote the first component of $\overline{w}$ by $\overline{w}_{1}.$
Note that $\overline{w}$ is $\sigma_{\GU_{n}}$-Coxeter element in the sense of \cite[Subsection 2.8]{He2022}.
Here, $\sigma_{\GU_{n}}$ is the Frobenius morphism associated with $\GU_{n}$.
We can also define the upper-half Borel subgroup of $\GU_{n}$ by $\overline{B}$, and its unipotent radical by $\overline{U}$, which is $\sigma_{\GU_{n}}$-invariant
We also denote the $q$-th power map by $\overline{\sigma}$.
We put
\[
\overline{\mathcal{L}_{b}^{\mathrm{symp}}}
 :=
\{
v \in \overline{V} \mid
( v, \overline{\sigma}^{i}(v) ) = 0 \,\, (1\leq i < n, \, \, i \colon \textup{odd}), ( v, \overline{\sigma}^{n} (v) ) \neq 0
\}.
\]
\end{enumerate}
\end{definition}

\begin{lemma}
\label{lemma:finitecomparisonodd}
We define the equivalence relation $\sim$ on 
$\overline{\mathcal{L}_{b}^{\mathrm{symp}}}$
by 
\[
v\sim v' \Leftrightarrow v \in  \overline{\F}_{q}^{\times}\cdot v'.
\]
Then we have an isomorphism of schemes over $\overline{\F}_{q}$
\[
\overline{\mathcal{L}_{b}^{\mathrm{symp}}}
/ \sim \,
\simeq 
X_{\overline{w}}^{(\overline{B})}.
\]
Here, 
$X_{\overline{w}}^{(\overline{B})}$ is the classical Deligne--Lusztig variety for $\GSp_{2n}$ (resp.\,$\GU_{n}$) if $k = 0$ (resp.\,$k=1$).
Note that we do not consider $\F_{q}$-structure for
$X_{\overline{w}}^{(\overline{B})}$.
Here, the scheme structure is defined as in Lemma \ref{lemma:finitecomparison}.
We put $\PP(\Lsympbar)$ in the same way as in Lemma \ref{lemma:finitecomparison}.
\end{lemma}

\begin{proof}
The case where $k=0$ follow from Lemma \ref{lemma:finitecomparison}.
We prove the case where $k=1$.
As in the proof of Lemma \ref{lemma:finitecomparison}, for any $v \in \overline{\mathcal{L}_{b}^{\mathrm{symp}}}$,
we put
\[
\overline{G}_{1}(v) := 
\frac{\alpha_{v}}{\overline{\sigma}^{-1}(\alpha_{v})} \overline{\sigma}^{-1}(v),
\]
where we put $\alpha_{v}:= ( v, \overline{\sigma}^{n}(v) )$.
We also put
\[
\overline{G}_{i+1}(v):= 
\begin{dcases}
\frac{\alpha_{v}}{\overline{\sigma}^{-1}(\alpha_{v})} \overline{\sigma}^{-1}(\overline{G}_{i} (v)) & \textup{if }i \textup{ is even},\\
\frac{\alpha_{v}}{\overline{\sigma}^{-1}(\alpha_{v})} \left(\overline{\sigma}^{-1}(\overline{G}_{i}(v))-\frac{( \overline{\sigma}^{-1}(\overline{G}_{i}(v)), \overline{\sigma}^{n}(v))}{\alpha_{v}}v\right) & \textup{if }i \textup{ is odd},
\end{dcases}
\]
for $1 \leq i \leq n-2$.
We put
\begin{eqnarray*}
\overline{g}_{1} (v) &=& (v, \overline{\sigma}^{2} (v), \ldots, \overline{\sigma}^{n-1} (v), \overline{G}_{2} (v) ,\ldots, \overline{G}_{n-1} (v)), \\
\overline{g}_{2} (v) &=& ( \overline{\sigma} (v), \overline{\sigma}^{3} (v), \ldots, \overline{\sigma}^{n-2} (v), \overline{G}_{1} (v), \overline{G}_{3}(v), \ldots, \overline{G}_{n-2}(v), \overline{\sigma}^{n} (v)).
\end{eqnarray*}
Then we have $^{t}\!\overline{g}_{2} (v) H \overline{g}_{1} (v) =\alpha H$. Therefore,
\[
\overline{g}(v) :=(\overline{g}_{1}(v), \alpha) = (\overline{g}_{1} (v), \overline{g}_{2} (v) , \alpha) \in \GU_{n} (\overline{\F}_{q}).
\]
We have
\begin{eqnarray*}
\overline{g}_{1} (v) \overline{w} &=& (\overline{\sigma}^{2} (v), \ldots, \overline{\sigma}^{n-1} (v), v,  \overline{G}_{2} (v) ,\ldots, \overline{G}_{n-1} (v)), \\
\sigma_{\GU_{n}}(\overline{g}_{2} (v)) &=& ( \overline{\sigma}^{2} (v), \ldots, \overline{\sigma}^{n-1} (v), \overline{\sigma}(\overline{G}_{1} (v)), \overline{\sigma}(\overline{G}_{3}(v)), \ldots, \overline{\sigma}(\overline{G}_{n-2}(v)),  \overline{\sigma}^{n+1} (v)).
\end{eqnarray*}
Since 
\[
\sigma_{\GU_{n}} (\overline{g}_{1} (v), \overline{g}_{2} (v), \alpha) = (\overline{\sigma} (\overline{g}_{2} (v)), \overline{\sigma} (\overline{g}_{1} (v)), \overline{\sigma} (\alpha)),
\]
we can show that
\[
\sigma_{\GU_{n}} (\overline{g} (v)) = \overline{g} (v) \overline{w} C (v)
\]
such that
\[
C(v) =
( \left(\begin{array}{cccc|ccc}
1&&&&&&\ast\\
&\ddots&\hsymb{0}&&&&\vdots\\
&&1&&&\hsymb{0}&\ast\\
&&&1&&&0\\\hline
&&&&\frac{\overline{\sigma}(\alpha)}{\alpha}&&\\
&&&&&\ddots&\Lsymb{0}\\
&\hsymb{0}&&&&&\frac{\overline{\sigma}(\alpha)}{\alpha}
\end{array}\right)
,\frac{\overline{\sigma}(\alpha)}{\alpha} ).
\]
Therefore, we  have a morphism
\[
\overline{\mathcal{L}_{b}^{\mathrm{symp}}} \rightarrow X_{\overline{w}}^{(\overline{B})}; v \mapsto \overline{g} (v) \overline{B}.
\]
Clearly, 
\[
\overline{\mathcal{L}_{b}^{\mathrm{symp}}}
/ \sim \,
\rightarrow X_{\overline{w}}^{(\overline{B})}
\]
is injective. Therefore, it suffices to show the surjectivity.
Let $\overline{g} \overline{B} \in X_{\overline{w}}^{(\overline{B})}$.
We may assume that 
$\overline{g}^{-1} \sigma_{\GU_{n}} (\overline{g}) \in \overline{w} \overline{B}$.
We put $\overline{g}= (\overline{g}_{1}, \overline{g}_{2}, \lambda)$.
We also put
$\overline{g}_{1}^{-1} \overline{\sigma}(\overline{g}_{2}) =C$,
where $C \in w_{1} B_{1}$.
Here, $B_{1}$ is the upper-half Borel subgroup of $\GL_{n}$.
By a similar argument to the proof of Claim \ref{claim:C'}, it suffices to show that there exists $P \in B_{1} (\overline{\F}_{q})$ such that
\[
P^{-1} C \overline{\sigma}(P'_{\lambda}) = C',
\]
with 
\[
C' :=
 \left(\begin{array}{cccc|ccc}
&&&\ast&&&0\\
\ast&&\hsymb{0}&&&&\ast\\
&\ddots&&&&\hsymb{0}&\vdots\\
&&\ast&&&&\ast\\\hline
&&&&\ast&&\\
&&&&&\ddots&\Lsymb{0}\\
&\hsymb{0}&&&&&\ast
\end{array}\right).
\]
Here, for any $\lambda \in \overline{\F}_{q}^{\times}$, $P'_{\lambda}$ is the matrix defined by
\[
^{t}\!P'_{\lambda} H P = \lambda H.
\]
To this end, we use the following kinds of elementary matrices
\[
P = 1_{n} + c e_{i,j}. 
\]
For such a $P$, we have $P_{1}' =1_{n} -c e_{n+1-j,n+1-i}$.
Therefore, the conjugation $\ast \rightarrow  P^{-1} \ast \sigma(P'_{1})$ {\cred acts by} adding $-c$ times the $j$-th row to the $i$-th row, followed by adding $- \overline{\sigma} (c)$ times the $(n+1-j)$-th column to the $(n+1-i)$-th column.
We will eliminate entries of $C$ by 3 steps. In the following, we use the same notation as in Section \ref{section:pfcomparison}.

First, we eliminate the following entries:
\begin{eqnarray*}
&&(n-1,n), (2,2), (n-2,n-1), \ldots, (n'+2, n'+3), (n',n'),\\
&&(n-2, n), (2,3), \ldots, (n'+2. n'+4), (n'-1, n')\\
&&\vdots\\
&&(n'+2, n), (2, n')
\end{eqnarray*}
Here, for the $i$-th line, we will eliminate 
$(n+1-(i+j), n+1-j), (j+1, j+i)$ $(1 \leq j \leq n'-i)$.
To eliminate $(n+1-(i+j), n+1-j),$ we use \r{n+1-j}{n+1-(i+j)}, which is together with \c{j}{i+j}.
Here, \c{j}{i+j} does not affect {\cred the entries we have set to zero} since $(j+1, i+j)$ has not {\cred been eliminated} yet.
To eliminate $(j+1, j+i),$ we use \r{j+i+1}{j+1}, which is together with \c{n-(i+j)}{n-j}.
Here, \c{n-(i+j)}{n-j} does not affect {\cred the entries we have set to zero} since $(n-(i+j), n-j)$ is not {\cred eliminated} yet.

In the next step, we eliminate the following entries:

\begin{eqnarray*}
&&(n'+1, n'+1) (n'+1,n'+2), (n',n'+2), (n', n'+3), \ldots, (3, 2n'), (2, 2n'),\\
&&(n', n'+1) (n'+1, n'+3), (n'-1, n'+2), (n',n'+4), \ldots, (4,2n'), (2, 2n'-1), \\
&&\vdots \\
&&(2, n'+1)
\end{eqnarray*}
Here, for the $i$-th line, we eliminate
$(n'+2-i, n'+1)$, 
$(n'+2-j, n'+i+j)$, and $(n'+2-(i+j),n'+1+j),
$ $(1 \leq i\leq n'-1 ,1\leq j \leq n'-i)$.
To eliminate $(n'+2-i, n'+1)$, we use \c{n'+1-i}{n'+1}, which is together with 
\r{n'+1+i}{n'+1}.
Note that \r{n'+1+i}{n'+1} affects only $(n'+1, n'+1+i)$, which is not {\cred elminated} yet. 
To eliminate $(n'+2-j, n' + i+j)$, we use
\c{n'+1-j}{n'+i+j}, which is together with
\r{n'+1+j}{n'+2-(i+j)}.
Note that \r{n'+1+j}{n'+2-(i+j)} affects only $(n'+2-(i+j), n'+1+j)$, which is not {\cred eliminated} yet.
To eliminate $(n'+2-(i+j),n'+1+j)$, we use \c{n'+1-(i+j)}{n'+1+j}, which is together with
\r{n'+1+i+j}{n'+1+j}. Here, \r{n'+1+i+j}{n'+1+j} affects only $(n'+2-(j+1),n'+i+(j+1))$, which is not {\cred eliminated} yet.

Third, we will {\cred eliminate} the following entries:
\begin{eqnarray*}
(1,n'+2), \ldots, (1,n)
\end{eqnarray*}
To eliminate $(1,i)$ $(n'+2 \leq i \leq n)$, we use \r{i}{1}, which is together with \c{n+1-i}{n}.
Here, \c{n+1-i}{n} affects only $(n+2-i, n)$, which does not need to be eliminated since $2\leq n+2-i \leq n'+1$.
{\cred This} finishes the proof.
\end{proof}

\begin{definition}
\label{definition:lhodd}
\begin{enumerate}
\item
Suppose that $k=0$.
We define $\la_{h}, \la_{b,h}^{\mathrm{symp}}$ by the same way as in Definition \ref{definition:lh} (1).
\item
Suppose that $k=1$.
We put
\[
\la_{h} := (\co / \p^{h})^{\oplus 2n},
\]
which is a quotient of $\la$.
We also put
\[
\la_{b,h}^{\mathrm{symp}}
:= \left\{ v=(\overline{v_{1}}, \ldots \overline{v_{2n}}) \in \la_{h} \left|
\begin{array}{l}
\langle v, F^{i}(v) \rangle = 0 \mod \p^{h+\lfloor \frac{i}{2} \rfloor}\quad (1 \leq i\leq n-1) \\
\langle v, F^{n}(v) \rangle \in \varpi^{n'} (\co/ \p^{h+n'})^{\times}
\end{array}
\right.
\right\}.
\]
\end{enumerate}
\end{definition}

\begin{lemma}
\label{lemma:affinefibrationodd}
\begin{enumerate}
\item
For any $h \geq 1$, the natural morphism
\[
\la_{b,h}^{\mathrm{symp}} \rightarrow \Lsympbar
\]
is a surjection.
\item
As in Lemma \ref{lemma:affinefibration},
consider the natural embedding 
\[
\la_{b,h}^{\mathrm{symp}} \times_{\Lsympbar} \Lsympbar^{\mathrm{perf}} 
\subset \mathbb{A}_{\Lsympbar^{\mathrm{perf}}} (\mathrm{resp.}\,\mathbb{A}^{\mathrm{perf}}_{\Lsympbar^{\mathrm{perf}}}),
\] 
if $\chara K >0$ (resp.\,$\chara K=0$).
Then this embedding is given by the intersection of hyperplane sections of $\mathbb{A}_{\Lsympbar^{\mathrm{perf}}}$ (resp.\,$\mathbb{A}^{\mathrm{perf}}_{\Lsympbar^{\mathrm{perf}}}$).
In particular,
$
\la_{b,h}^{\mathrm{symp}} \otimes_{\Lsympbar} \Lsympbar^{\mathrm{perf}} 
$
is isomorphic to an affine space (resp.\,perfection of an affine space) over $\Lsympbar^{\mathrm{perf}}$.
\end{enumerate}
\end{lemma}
\begin{proof}
It follows from the same proof as in Lemma \ref{lemma:affinefibration}
\end{proof}

\subsection{Structure of affine Deligne--Lusztig varieties}
\label{subsection:structureofADLV}
In this section, we prove the structure theorem for affine Deligne--Lusztig varieties.

\begin{definition}
\label{definition:reductive}
We define a reductive group $\overline{G}$ as follows:\footnote{The field of definition is not essential. Indeed, in the second case, the corresponding Deligne--Lusztig variety is the same as the Deligne--Lusztig varieties for the Weil restriction of $\overline{G}$ (see Remark \ref{remark:reweilrestriction})}
\[
\overline{G} :=
\begin{cases}
\GSp_{2n} \textup{ over } \F_{q} & \textup{ if } n \textup{ is even and } k=0, \\
\GSp_{n} \textup{ over } \F_{q^{2}}  & \textup{ if } n \textup{ is even and } k=1, \\ 
\GSp_{2n} \textup{ over } \F_{q} & \textup{ if } n \textup{ is odd and } k=0,  \\
\GU_{n} \textup{ over } \F_{q} & \textup{ if } n \textup{ is odd and } k=1.
\end{cases}
\]
We also define $\overline{w}, \overline{U}$ as in Definitions \ref{def:findlv} and  \ref{def:findlvodd}.
\end{definition}

In this section, we prove the following theorem.

\begin{theorem}
\label{theorem:restradlv}
Suppose that $r+k \geq 1$.
Let $b\in \GSp(V)$ be the special representative or the Coxeter representative with $\kappa (b) = k$.
Then we have a decomposition of $\overline{\F}_{q}$-schemes
\[
X^{0}_{w_{r}} (b)_{\la}^{\mathrm{perf}} \simeq X_{\overline{w}}^{\overline{B}, \mathrm{perf}} \times \mathbb{A}^{\mathrm{perf}}.
\]
Here, $\mathbb{A}$ is an affine space over $\overline{\F}_{q}$ and $\mathrm{perf}$ means the perfection.
\end{theorem}

\begin{remark}
\label{remark:reweilrestriction}
When $n$ is even and $k=1$, the Deligne--Lusztig variety $X_{\overline{w}}^{\overline{B}}$ is isomorphic to the Deligne--Lusztig variety of $\overline{G}_{0}:= \Res_{\F_{q^{2}} / \F_{q}} \GSp_{n, \F_{q^{2}}}$ with respect to some $\sigma_{\overline{G}_{0}}$-Coxeter element $\overline{w}_{0} \in \overline{G}_{0} (\overline{\F}_{q}) = \GSp (\overline{V}) \times \GSp(\overline{V})$ which corresponds to $\overline{w} \times 1$. 
\end{remark}

In the following of this section, we assume that $b$ is the special representative with $\kappa (b) =k$ (by Proposition \ref{prop:component}.(3), the Coxeter representative case is reduced to this case).
Since $X^{m}_{w_{r}} (b)_{\la}$ is contained in an affine Schubert cell, we can calculate the structure of $X^{m}_{w_{r}} (b)_{\la}$ directly, by using the coordinates of affine Schubert cells (\cite[Lemma 4.7]{Chan2021}).
We want to study such coordinates by using the results on the structure of $\Lsymp$.
However, since representatives of such coordinates have $\varpi$-adic extensions of finite length, our $\Lsymp$ is not suitable (note that each entry of $v \in \Lsymp$ has $\varpi$-adic expansion of infinite length).
Therefore, first, we will define a finite-length analogue of $\Lsymp$ in the following.

\begin{definition}
\label{definition:Lbmr}
\begin{enumerate}
\item
We define 
\[
\alpha(i), \beta_{m,r}(i), \gamma_{m,r} (i) \colon \{1, \ldots, 2n\} \rightarrow \Z
\]
by the following:
\begin{align*}
&(\alpha(1), \ldots, \alpha(n), |\alpha(n+1), \ldots, \alpha(2n))\\
=&
\begin{cases}
(0, \ldots, 0, | 1, 0, \ldots, 1, 0)   &\textup{ if } n \textup{ is even and } k=1,\\
(0, \ldots, 0, | 0, \ldots, 0)  &\textup{ otherwise}.
\end{cases} 
\end{align*}
\begin{align*}
&(\beta_{m,r}(1), \ldots, \beta_{m,r}(n),| \beta_{m,r}(n+1), \ldots, \beta_{m,r}(2n))\\
=&
\begin{cases}
(m+r, \ldots, m+r, | m+r, \ldots, m+r)   &\textup{ if } k=0,\\
(m+r+1, m+r, \ldots,  m+r+1, m+r, | m+r+1, \ldots, m+r+1)  &\textup{ if } n \textup{ is even and }k=1,\\
(m+r+1, m+r, \ldots, m+r+1, m+r) & \textup{ if } n \textup{ is odd and }k=1,                     
\end{cases} 
\end{align*}
\begin{align*}
&(\gamma_{m,r}(1), \ldots, \gamma_{m,r}(n),| \gamma_{m,r}(n+1), \ldots, \gamma_{m,r}(2n))\\
= & 
(m+ \varphi_{r} (1), \ldots, m+ \varphi_{r} (2n)).
\end{align*}
\item
We define $\la'_{b,m,r}$ as the projection of 
\[
\Lsymp \cap ^{t}(1, \p^{\alpha(2)}, \ldots, \p^{\alpha(2n)})
\]
via
\[
^{t}\!(1, \p^{\alpha(2)}, \ldots, \p^{\alpha(2n)}) \twoheadrightarrow 
^{t}\!(1, \p^{\alpha(2)}/ \p^{\beta_{m,r}(2)}, \ldots, \p^{\alpha(2n)}/ \p^{\beta_{m,r}(2n)}).
\]
We define $\la_{b,m,r}$ by the inverse image of $\la'_{b,m,r}$ via the projection
\[
^{t}\!(1, \p^{\alpha(2)}/ \p^{\gamma_{m,r}(2)}, \ldots, \p^{\alpha(2n)}/\p^{\gamma_{m,r}(2n)})
\twoheadrightarrow
^{t}\!(1, \p^{\alpha(2)}/ \p^{\beta_{m,r}(2)}, \ldots, \p^{\alpha(2n)}/ \p^{\beta_{m,r}(2n)}).
\]
Note that $\gamma_{m,r} (i) \geq \beta_{m,r} (i)$ for $2 \leq i \leq 2n$.
\item
We define $[\la_{b,m,r}] \subset V$ by 
\[
[\la_{b,m,r}] :=
\{
^{t}\!([v_{1}], \ldots, [v_{2n}]) \in \la | ^{t}\!(v_{1}, \ldots, v_{2n}) \in \la_{b,m,r}
\}.
\]
Here, we define the map 
\[
[-] \colon \p^{i}/ \p^{j} \rightarrow \co
\]
for $j> i \geq 0$ by 
\[
[a] :=  \sum_{l=i}^{j-1} [a_{l}]\varpi^{l}.
\]
Here, $a_{l} \in \overline{k}$ is defined by the formula
\[
\widetilde{a} := \sum_{l\geq i} [a_{l}]\varpi^{l},
\]
where $\widetilde{a} \in \p^{i}$ is any lift of $a$, and $[a_{l}] \in W(\overline{k}) \subset K$ is the Teichm\"{u}ller lift of $a_{l}$.

\item
Let $\lac \subset [\la_{b,m,r}]$ be the subset consisting of $v \in [\la_{b,m,r}]$ such that
\begin{equation}
\label{eq:deflac}
\langle v, F^{i} (v) \rangle \in \p^{(n-i)r+ m+ \lfloor \frac{nk}{2} \rfloor}
\end{equation}
hold for $1 \leq i \leq n-1$.
\end{enumerate}
\end{definition}

\begin{remark}
\label{remark:lcoord}
\begin{enumerate}
\item
The set $\lac$ have a natural $\overline{k}$-scheme structure for the same reason as in Remark \ref{remark:Lh}.
\item
When $n=1$, then we have 
\[
 \beta_{m,r} (2) = \gamma_{m,r} (2) = m+r.
\]
In this case, we have $\la'_{b,m,r} = \la_{b,m,r}$.
\item
The equation (\ref{eq:deflac}) is automatic when $n=1$ or $n=2$.
Therefore, in this case, we have
\[
\lac = [\la_{b,m,r}].
\]
Combining with the remark (1), we have
\[
\lac = [\la'_{b,m,r}]
\]
when $n=1$.
This is no other than $``U''$ used in the proof of \cite[Theorem 6.17]{Chan2021}.
\end{enumerate}
\end{remark}

\begin{definition}
We define the map 
\[
h_{1}, \ldots, h_{2n} \colon \lac \rightarrow V
\]
as follows:
Let $v \in \lac$.
First, we put 
\begin{align*}
h_{1}(v) &:= v, \\
h_{2n}(v) &:= F^{n} (v).
\end{align*}
Next, we inductively define 
\begin{align*}
h_{n+1-i} (v) := &F^{n-i} (v)  -  \frac{\langle h_{2n} (v), F^{n-i} (v) \rangle}{\langle h_{2n}(v), h_{1} (v) \rangle} h_{1}(v) - 
\frac{\langle h_{1}(v),  F^{n-i} (v) \rangle}{\langle h_{1}(v), h_{2n} (v) \rangle} h_{2n} (v)
\\
&- \sum_{n+2-i \leq j \leq n} \left(\frac{\langle h_{2n+1-j}(v), F^{n-i}(v) \rangle}{\langle h_{2n+1-j}(v), h_{j} (v)  \rangle} h_{j} (v)  
+ \frac{\langle h_{j} (v), F^{n-i}(v) \rangle}{\langle h_{j}(v), h_{2n+1-j} (v) \rangle} h_{2n+1-j}(v) \right),   \\
\end{align*}
\begin{align*}
h_{n+i} (v) := & V_{k}^{i} (v)- \frac{\langle h_{2n} (v), V_{k}^{i}(v) \rangle}{\langle h_{2n}(v), h_{1}(v) \rangle} h_{1} (v) - \frac{\langle h_{1}(v), V_{k}^{i}(v) \rangle}{\langle h_{1}(v), h_{2n} (v) \rangle}h_{2n} (v)
\\
&- \sum_{n+2-i \leq j \leq n} \left(
\frac{\langle h_{2n+1-j}(v), V_{k}^{i} (v) \rangle}{\langle h_{2n+1-j}(v), h_{j}(v) \rangle} h_{j} (v)
+ 
\frac{\langle h_{j}(v), V_{k}^{i} (v) \rangle}{\langle h_{j} (v), h_{2n+1-j}(v) \rangle} h_{2n+1-j} (v)
\right),
\end{align*}
for $1 \leq i \leq n-1$.
Note that we have
\begin{equation}
\label{equation:orthogonal}
\langle h_{i} (v) , h_{j} (v) \rangle =0
\end{equation}
when $i+j \neq 2n+1$.
\end{definition}

The next lemma is an analogue of Lemma \ref{lemma:order}.

\begin{lemma}
\begin{enumerate}
\item
For $v \in \lac$, we have
\[
 \ord \langle h_{i} (v), h_{2n+1-i} (v)  \rangle = \ord  \langle v, F^{n} (v) \rangle = \lfloor \frac{kn}{2} \rfloor.
\]
We put 
\begin{align*}
h_{b,r} (v)&:= (h_{1} (v), \varpi^{r} h_{2} (v), \ldots, \varpi^{(n-1)r} h_{n} (v), \\
&
\left.
(-1)^{n+1}\varpi^{r}
\frac{\langle h_{1}(v), h_{2n}(v) \rangle}
{\langle h_{n}(v), h_{n+1}(v)  \rangle} 
h_{n+1} (v),
(-1)^{n+2}\varpi^{2r}
\frac{\langle h_{1}(v), h_{2n}(v)\rangle}{\langle h_{n-1}(v),h_{n+2}(v) \rangle}h_{n+2}(v), \ldots, 
\varpi^{nr} h_{2n} (v)
\right).
\end{align*}
\item
We have
\[
F(h_{b,r} (v) )
= h_{b,r} (v) w_{r} A_{b,r}
\]
with
\[
A_{b,r} \in I^{m}.
\]
\end{enumerate}
\end{lemma}
\begin{proof}
Note that any $v\in [\la_{b,m,r}]$ can be written as $v= v_{0} + \varpi^{r+m} u$, where $v_{0} \in \Lsymp$ and 
\[
u \in ^{t}\!(0, \p^{\beta_{m,r}(2)-m-r}, \ldots, \p^{\beta_{m,r}(2n)-m-r}) \subset 
^{t}\!(\p^{\beta_{m,r}(1)-m-r}, \p^{\beta_{m,r}(2)-m-r}, \ldots, \p^{\beta_{m,r}(2n)-m-r})
.
\]
We have
\[
\ord \langle v, F^{n} (v) \rangle 
= \langle v_{0}, F^{n} (v_{0}) \rangle + \varpi^{r+m} ( \langle v_{0}, F^{n} (u) \rangle  + \langle u, F^{n} (v_{0}) \rangle) + \varpi^{2(r+m)} (\langle u, F^{n} (u) \rangle).
\]
Since $v_{0} \in \la$, we have
\[
\ord \langle v_{0}, F^{n} (u) \rangle 
\,(\textup{resp}.\,\langle u, F^{n} (v_{0})  \rangle \textup{ and } \langle u, F^{n} (u) \rangle) \geq \lfloor \frac{kn}{2} \rfloor +k
\]
by the direct calculation of orders.
Note that
we have $r \geq 1$ in the case where $k=0$.
Therefore, we can conclude that
\[
\langle v, F^{n}(v) \rangle = \lfloor \frac{kn}{2} \rfloor.
\]
To compute $\ord  \langle h_{i} (v), h_{2n+1-i}(v) \rangle$, it is useful to introduce the reduced version of $h_{i}$ as {\cred follows}.
We put 
\[
h_{i}^{\mathrm{red}} (v) = 
\begin{cases}
\varpi^{- \lfloor \frac{k(i-1)}{2} \rfloor} h_{i} (v)  & \textup{ if } 1 \leq i \leq n, \\
\varpi^{- \lfloor \frac{kn}{2} \rfloor + \lfloor \frac{k(2n-i)}{2} \rfloor}   & \textup{ if } n+1 \leq i \leq 2n. \\
\end{cases}
\]
It is enough to show that 
\begin{equation}
\label{equation:goal}
\ord \langle h_{i}^{\mathrm{red}}(v), h_{2n+1-i}^{\mathrm{red}}(v) \rangle = 0.
\end{equation}
For $1 \leq i \leq 2n$, we define the subset $M_{i} \subset \la$ by the formula
\[
\begin{cases}
M_{i} = \varpi^{m+r} F^{\overline{i-1}+1} \la  & \textup{ if } 1 \leq i \leq n,\\
M_{i} = \varpi^{m+r} F^{\overline{n} -\overline{2n-i}}  \la    & \textup{if} n+1 \leq i \leq 2n.\\
\end{cases}
\]
Here, for any integer $j \in \Z$, $\overline{j} \in \{0,1\}$ means its mod $2$.
By the direct computation according to the definition of $h_{i}^{\mathrm{red}}$ and (\ref{equation:orthogonal}), we can show 
\[
\begin{cases}
h_{i}^{\mathrm{red}} (v) -  \varpi^{-\lfloor \frac{i}{2} \rfloor} F^{i} (v) \in M_{i} & \textup{ if }  1\leq i \leq n,  \\
h_{i}^{\mathrm{red}} (v) - \varpi^{- \lfloor \frac{n}{2} \rfloor} F^{n} (v)  \in M_{i}   & \textup{ if } i=2n. \\
\end{cases}
\]
By the definition of  $M_{i}$ and (\ref{equation:orthogonal}), these imply (\ref{equation:goal}). {\cred This} finishes the proof of the assertion (1).

Next, we will prove the assertion (2).
This calculation is rather redundant, so we will only describe a sketch.
We only consider the case where $k=1$ and $n$ is even.
Since we only need to estimate the order of entries, we will proceed with the calculation ignoring $\co^{\times}$-multiplications.
We define $\widetilde{h}_{b,r} (v)$ by 
\begin{equation}
\label{eqn:tildeh}
\widetilde{h}_{b,r} (v) :=(\hpr_{1}, \ldots, \hpr_{2n}),
\end{equation}
where
\begin{equation}
\hpr_{i} :=
\label{eqn:tildehi}
\begin{cases}
F^{i-1} (v)
& \textup{ if } 1\leq i \leq n,\\
V_{k}^{i-n} (v)
& \textup{ if } n+1 \leq i \leq 2n-1,\\
F^{n} (v) 
& \textup{ if } i=2n.
\end{cases}
\end{equation}
We also put
\begin{align}
\begin{aligned}
\hpr^{\mathrm{red}}_{b,r} (v) = (\hpr^{\mathrm{red}}_{1}, \ldots, \hpr^{\mathrm{red}}_{2n} ) =& \hpr_{b,r} (v) \cdot \diag(1,1, \varpi^{-1}, \varpi^{-1}, \ldots, \varpi^{-(n'-1)}, \varpi^{-(n'-1)}, |  \\
&\varpi^{-1}, \varpi^{-1}, \ldots, \varpi^{-n'}, \varpi^{-n'}).
\end{aligned}
\end{align}

We define the equations $s \colon \Z^2 \to \{0,1\}$ by 
$s(i,j)$ is 1 (resp.\,0) if $i$ and $j$ are odd (resp.\,otherwise).
First,
for $1 \leq i \leq n-1$,
we can show that
\begin{eqnarray}
\begin{aligned}
\label{eqn:hredcomparison1}
h_{n+1-i}^{\red} &\in \hpr^{\mathrm{red}}_{n+1-i}
 + \p^{(n-i)r+m+ \frac{n}{2}- \lfloor \frac{i}{2} \rfloor} \hpr_{1}
 + \p^{ir +m+ \frac{n}{2} - \lfloor \frac{n-i}{2} \rfloor} \hpr_{2n} \\
&+ \sum_{j=1}^{i-1} \p^{(i-j)r + m + \frac{n}{2} - \lfloor \frac{n-i+j}{2} \rfloor - \overline{n-j} + s(j,n-i)} \hpr_{n-j+1} \\ 
&+ \sum_{j=1}^{i-1} \p^{(n-i+j)r +m + \frac{n}{2} - \lfloor \frac{i-j}{2} \rfloor} \hpr_{n+j}.\\
\end{aligned}
\end{eqnarray}
and
\begin{eqnarray}
\begin{aligned}
\label{eqn:hredcomparison2}
h_{n+i}^{\red} (v) &\in \hpr^{\mathrm{red}}_{n+i} 
 + \co v +
 \p^{(n-i)r+m+ \frac{n}{2} - \lceil \frac{i}{2} \rceil} 
 \hpr^{\mathrm{red}}_{2n}\\
 &+ \sum_{j=1}^{i-1} \p^{(n-i)r + m+ \frac{n}{2} - \lceil \frac{i}{2} \rceil} \hpr^{\mathrm{red}}_{n-j+1}\\
 &+ \sum_{j=1}^{i-1} \p^{s(n-j, i+1)} \hpr^{\mathrm{red}}_{n+j}.\\
 \end{aligned}
\end{eqnarray}
Indeed, we can show these inclusions by the induction on $i$ by using the definition of $h_{i}^{\mathrm{red}}$ and equations (\ref{eq:deflac}).
We can define the matrix $H \in \GL_{2n} (\breve{K})$ by
\[
h_{b,r} (v) = \hpr_{b,r} (v) H.
\]
Since we have
\[
w_{r}^{-1} h_{b,r} (v)^{-1} F(h_{b,r} (v))
 = w_{r}^{-1}  H^{-1} w_{r} w_{r}^{-1} \hpr_{b,r} (v)^{-1} F(\hpr_{b,r} (v)) \sigma (H),
\]
it suffices to show the following.
{\cred
\begin{eqnarray*}
\hpr_{b,r} (v)^{-1} F(\hpr_{b,r} (v)) \in I^{m}_{\GL}, \\
H \in I^{m}_{\GL} \cap w_{r} I^{m}_{\GL} w_{r}^{-1}.
\end{eqnarray*}}
First, we prove the former {\cred containment}.
By the direct computation similar to Lemma \ref{lemma:order},
we have the following.
\[
\hpr_{b,r} (v)^{-1} F (\hpr_{b,r} (v)) \in
\left(\begin{array}{ccc|ccc}
1&&&&& a_{1}\\
&\ddots&&\Lsymb{0}&&\vdots\\
&&1&&&a_{n}\\ \hline
&&&1&&a_{n+1}\\
&\hsymb{0}&&&\ddots&\vdots\\
&&&&&a_{2n}\\
\end{array}\right).
\]
Here, the diagonal entries are $1$ except for the $2n$-th row. Moreover, we have the following.
\begin{enumerate}
\item
 For $1 \leq i \leq n$, we have
 \[
\ord \delta_{i} \geq
rn+ \frac{k}{2} (n-i+1).
\]
\item
For $n+1 \leq i\leq 2n-1$. we have
\[
\ord \delta_{i} \geq (2n-i) r + (n- \frac{i}{2} )k.
\]
\item
We have $ \ord a_{2n} =0$.
\end{enumerate}
In particular, we have the former inclusion.

Next, we show the {\cred latter} {\cred containment}.
By using (\ref{eqn:hredcomparison1}) and (\ref{eqn:hredcomparison2}), we can show that
\[
H \in
\left(
\begin{array}{cccc|cccc}
1&\ast& \cdots&\ast&\ast&\cdots&\ast&0 \\
0&1&&&&&0&0 \\
\vdots&&\ddots&&&\iddots&&\vdots \\
0&\hsymb{\ast} && 1 &0&&\hsymb{\ast} &  0 \\ \hline
0 & &&0&1&&&0  \\
\vdots & \hsymb{\ast}&\iddots&&&\ddots& \hsymb{\ast}&\vdots \\
0 &0&&&&&1&0 \\
0 &\ast& \cdots&\ast& \ast& \cdots& \ast& 1 \\
\end{array}
\right).
\]
Here, since we only need to focus on the order of entries, we ignore multiplications by units.
By unwinding equations (\ref{eqn:hredcomparison1}) and (\ref{eqn:hredcomparison2}), we {\cred obtain} estimates for the orders of the entries of $H$, {\cred and thus obtain} the desired inclusion. 
{\cred This} finishes the proof.
\end{proof}

\begin{lemma}
\label{lemma:equivlac}
Let $v \in \lac$ and $u \in \Lsymp$.
We assume that we have
\[
v \in g_{b,r} (u) ^{t}\!(\co^{\times}, \p^{m}, \ldots, \p^{m}).
\]
Then we have
\[
h_{b,r} (v) I^{m} = g_{b,r} (u) I^{m}.
\]
\end{lemma}

\begin{proof}
As before, we only consider the case where $k=1$ and $n$ is even.
By the proof of Lemma \ref{lemma:order}, we have
$h_{b,r} (v) I^{m} = \hpr_{b,r} (v) H I^{m}  = \hpr_{b,r} (v) I^{m}$.
In the same argument as in the proof of Proposition \ref{proposition:bmr},
we have
$
\hpr_{b,r} (v) I^{m} = \hpr_{b,r} (u) I^{m}.
$
Here, we define $\hpr_{b,r}$ on $\Lsymp$ by the same formula as in (\ref{eqn:tildeh}) and (\ref{eqn:tildehi}).
Moreover, if we define $H' \in \GL (V)$ by
\[
h_{b,r} (u) = \hpr_{b,r}(u) H',
\]
then we have $H' \in I^{m}$.
Now we have
\[
\hpr_{b,r} (u) I^{m} = h_{b,r} (u) I^{m},
\]
and {\cred this} finishes the proof.
\end{proof}

\begin{proposition}
\label{prop:strbylac}
\begin{enumerate}
\item
The map $h_{b,r}$ defines a morphism
\[
\laczero \rightarrow X_{w_{r}}^{0} (b)_{\la}
\]
 of schemes over $\overline{\F}_{q}$
\item
The above morphism is an isomorphism.
\end{enumerate}
\end{proposition}

\begin{proof}
From Lemma \ref{lemma:order}, it is clear that assertion (1) holds true for $\overline{k}$-valued points.
Moreover by Proposition \ref{prop:cell containment},  $X_{w_{r}}^{m} (b)_{\la}$ is a locally closed subscheme of a higher affine Schubert cell $I x_{r} I /I^{m}$, which is described by \cite[Lemma 4.7]{Chan2021}.
Then the assertion (1) follows directly from the above description.
However, there seems to be subtlety in the statement of \cite[Lemma 4.7]{Chan2021} (see Remark \ref{rem:cimistake}).  
Therefore, we introduce the correct statement in the following.
We only treat the case where $k=1$ and $n$ is even since other cases can be proved similarly. 
Furthermore, we will reduce the argument to {\cred the} $\GL$ case for simplicity.
{\cred We use a general $m$ in the following argument to clarify when the condition $m=0$ is required.}
 
{\cred Let} $g^{0} (i,j)$ be the minimum order of the $(i, j)$-th entry of the right-hand side of (\ref{eq:Hcontain}), i.e.,
\[
G_{x,0} \cdot \diag(1, 1, \ldots, \varpi^{n}, \varpi^{n}) \cdot \diag (1, \varpi^{r}, \ldots, \varpi^{(n-1)r}, \varpi^{r}, \ldots, \varpi^{nr}).
\]
We use the notation $\tau$ defined in Definition \ref{defn:tau}.
For $1\leq i \neq j \leq 2n$, we put
\[
g(i,j) =
\begin{cases}
0 & \textup{ if } i>j, \\
1 & \textup{ if } j<i. \\
\end{cases}
\]
For such $(i,j)$, we also put
\[
h(i,j) :=
\begin{cases}
g^{0}(i, \tau(i)) + 1 - g^{0} (j, \tau(j)) &\textup{ if } \tau (i) < \tau (j), \\
g^{0}(i, \tau(i)) - g^{0} (j, \tau(j)) &\textup{ if } \tau(i) > \tau(j). \\
\end{cases}
\]
We define the set $I$ by
\[
I := \{
\alpha_{i,j} \mid g(i,j) < h(i,j)
\},
\]
here, $\alpha_{i,j}$ is the root of $\GL_{2n}$ corresponding to $(i,j)$-th entries.
We have
\begin{align}
\label{eqn:decomp}
\begin{aligned}
\psi: X^{m}_{w_{r}}(b)_{\mathcal{L}} \subset I x_{r} I / I^{m} &\subset I_{\GL} x_{r}  I_{\GL} / I^{m}_{\GL} \\
&\simeq
\prod_{\alpha_{i,1} \in I}
L_{[g(i,1), h(i,1))} U_{\alpha_{i,1}} 
\times
\prod_{\alpha_{i,j} \in I, j\neq 1} L_{[g(i,j),h(i,j)} U_{\alpha_{i,j}} \times I_{\GL}/I^{m}_{\GL}.
\end{aligned}
\end{align}
Here, $L_{[g(i,j),h(i,j)} U_{\alpha_{i,j}}$ is the $\overline{\F}_{q}$-perfect scheme defined in \cite[p.1815]{Chan2021}.
The last isomorphism $f$ is given by
\begin{equation}
\label{eqn:schubertisom}
f^{-1}((u_{i1}), (u_{ij}),g) = \prod_{\alpha_{i,1} \in I} [u_{i1}] \cdot \prod_{\alpha_{i,j} \in I, j \neq 1} [a_{ij}] \cdot x_{r} \cdot g I^{m}_{\GL},
\end{equation}
where the first product is with respect to any order, and the second product is taken with respect to the lexicographical order for $(j,i)$.
Moreover, [-] means the Teichm\"{u}ller lift (more precisely, the operation applying [-] in Definition \ref{definition:Lbmr} to each entry).
\begin{figure}
\centering
\captionsetup{justification=centering}
\[
\left(\begin{array}{cccc|cccc}
1 &0 & 0& 0&0&0&0&0\\
\co/ \p^{r} &1 &0&0&0&0&0&0 \\
\co/ \p^{1+2r} &\co/ \p^{1+r} & 1& \p / \p^{1+r}&0&0&0&0\\
\co/ \p^{r}  & \co/ \p & 0 &1 &0&0&0&0\\\hline
\co/ \p^{2+3r}   &\co/ \p^{2+2r} &\co/ \p^{1+r}&\co / \p^{2+2r}& 1 & \p / \p^{1+r} & 0& 0\\
\co/ \p^{1+2r} &\co/ \p^{r+2} & \co/\p & \co/\p^{1r} &0&1&0&0\\
\co / \p^{2+3r} &\co/ \p^{3+2r}&\co/ \p^{2+r}&\co/\p^{2+2r}&\co/\p&\co/\p^{1+r}&1&0\\
\co/\p^{2+4r}  &\co/\p^{2+3r}&\co/\p^{1+2r}& \co/ \p^{2+3r} & \co/\p^{r} & \co/\p^{1+2r} & \co/\p^{r} &1\\
\end{array}\right).
\]
\caption*{The shape of $(\prod_{\alpha_{i,j} \in I} L_{[g(i,j), h(i,j))} U_{\alpha_{i,j}})$ 
for $\GSp_{8}$.
}
\end{figure}

Since we only consider the perfect scheme structure on $\la_{coord}$, it suffices to show that the map $\psi \circ h_{b,r}$ defines a map of schemes over $\overline{\F}_{q}$.
By the proof of Lemma \ref{lemma:order}, we have
\[
h_{b,r} (v) I^{m}_{\GL} = \hpr_{b,r} (v) I^{m}_{\GL}.
\]
Therefore, it suffices to show that $f \circ \hpr_{b,r}$ define a morphism of schemes.
Each factor of $f(\hpr_{b,r} (v))$ can be computed as {\cred follows}:
For any $\hpr_{b,r} (v)$\, ($v\in \la_{\mathrm{coord}}$) there exist column-elementary transformations in $I$ whose composition $g \in I$ satisfies that
\[
\hpr_{b,r} (v) g^{-1} \in f( \prod_{\alpha_{i,j} \in I} L_{[g(i,j), h(i,j)}U_{\alpha_{i,j}} ).
\]
Such row elementary transformations are done by transforming the $1$st, \ldots, $(2n)$-th row in this order.
Note that, under such transformation of $\hpr_{b,r} (v)$, the order of the $(i,\tau(i))$-th entry is preserved by the same argument as in the proof of Proposition \ref{prop:cell containment}.
Since each coefficient of $\varpi$-adic expansion of $\hpr_{b,r} (v) g^{-1}$ is an algebraic function on $\la_{\mathrm{coord}},$ we have the assertion (1). 

Next, we will prove the assertion (2).
By the projection $p$ to $\prod_{\alpha_{i,1} \in I}
L_{[h(i,1), g(i,1))} U_{\alpha_{i,1}} $ followed by $\psi$, we have the following diagram:
\[
\xymatrix{
\laczero \ar[r]^-{h_{b,r}} \ar[rd] & X_{w_{r}}^{0} (b)_{\la} \ar[d]^{p} \\
&  \prod_{\alpha_{i,1} \in I} L_{[g(i,1), h(i,1))} U_{\alpha_{i,1}}
}
\]
By the definition of $h_{b,r}$ and the argument in the proof of the assertion (1),
$p\circ h_{b,r}$ is just a projection.
(here, we use the information of the order of products in (\ref{eqn:schubertisom})).
Note that, for $2 \leq i \leq 2n$, $\alpha$ and $\beta$ defined in Definition \ref{definition:Lbmr} satisfy 
\begin{equation}
\label{eqn:alphagh}
g(i,1) \leq \alpha(i) \,\textup{ and }\, \gamma_{m,r}(i)=\gamma_{0,r} (i) = h(i,1).
\end{equation}
Therefore the map $p \circ h_{b,r}$ is a natural closed immersion.
Now, it suffices to show that $h_{b,r}$ is surjective.
To this end, by Definition \ref{defn:component of adlv}  and Lemma \ref{lemma:equivlac}, it suffices to show that for any $u \in \Lsymp$, there exists an element $v \in \laczero$ such that
\begin{equation}
\label{eqn:inclusionbmr}
v \in g_{b,r} (u) ^{t}\!(\co^{\times}, \co, \ldots, \co).
\end{equation}
Let $a_{i}$ be a projection of $g_{b,r} (u) I$ to $L_{[g(i,1), h(i,1))} U_{\alpha_{i,1}}$.
We put $v= ^{t}\!(1, [a_{2}], \ldots, [a_{2n}])$. Here, $[-]$ means the canonical lift as in Definition \ref{definition:Lbmr} (3).
Then by (\ref{eqn:alphagh}),
we have
\[
v \in ^{t}\!(1, [\p^{\alpha(2)}/ \p^{\gamma_{0,r}(2)}], \ldots, [\p^{\alpha(2n)}/\p^{\gamma_{0,r}(2n)}]).
\] 
Moreover, (\ref{eqn:inclusionbmr}) holds true since $u\in \Lsymp$. 
Now it suffices to show that $v \in \laczero$.
By the equation (\ref{eqn:inclusionbmr}), clearly we have $v\in [\la_{b,0,r}]$.
It suffices to show that
the equation (\ref{eq:deflac}) holds true for $m=0$.
This part can be shown by direct computation.
\end{proof}

\begin{remark}
\label{rem:cimistake}
The statement of \cite[Lemma 4.7]{Chan2021} seems to be subtle.
In \cite[Lemma 4.7]{Chan2021}, they construct an isomorphism
\[
\prod_{\alpha \in S} L_{[f_{I}(\alpha), f(\alpha))} U_{\alpha} \times I/ I_{f} \simeq IxI/I_{f},
\]
where $f$ is a concave function such that the associated subgroup $I_{f}$ is normal in $I$, $f_{I}$ is the concave function of the Iwahori subgroup $I$, and the set $S$ is a certain set of roots (we omit the definition here). 
First, $f(\alpha)$ in the index of $L$ is unsuitable since it should depend on $x$ (see \cite[(6.19)]{Chan2021} for the correct formula for $\GL$). Indeed, if $I_{f} =I$, the left-hand side is $1$ though the right-hand side is non-trivial. 

Moreover, they define an isomorphism by
\begin{equation}
\label{eqn:arbitlift}
((a_{\alpha}), i) \mapsto \prod_{\alpha} \widetilde{a}_{\alpha} x I^{m},
\end{equation}
where $\widetilde{a}_{\alpha}$ is an arbitrary lift of $a_{\alpha}$.
However, in general, this morphism depends on the choice of lifts, and we need to fix a choice of lifts.
For example, consider
\[
g=
\left(\begin{array}{ccc}
1 &0&0  \\
\varpi&\varpi&0\\
\varpi&\varpi&\varpi^{2}
\end{array}\right)
I_{\GL_{3}}
=  
\left(\begin{array}{ccc}
1 &0&0  \\
0&\varpi&0\\
0&\varpi&\varpi^{2}
\end{array}\right)
I_{\GL_{3}}
\in I_{\GL_{3}} \diag(1,\varpi,\varpi^{2}) I_{\GL_{3}} / I_{\GL_{3}}.
\]
In this case, the correct isomorphism is given by
\[
L_{[0,1)} U_{\alpha_{2,1}} \times L_{[0,2)} U_{\alpha_{3,1}} \times (\textup{other terms}) \simeq I_{\GL_{3}} \diag(1,\varpi,\varpi^{2}) I_{\GL_{3}} / I_{\GL_{3}}.
\]
However, as above, 
if we do not fix the choice of lifts, the projection of $g$ in $L_{0,2} U_{\alpha_{3,1}}$ is not well-defined.
Therefore, to fix the isomorphism, we need to fix a lift.
As long as we modify the statement as above, then the proof can be done in the same way as in the proof of \cite[Lemma 4.7]{Chan2021}.
Note that we use the canonical lift $[-]$ (Definition \ref{definition:Lbmr}) in our paper.
Since $[0]=0$, the projection of $g$ in $L_{0,2} U_{\alpha_{3,1}}$ via our choice of isomorphism is $0$.
\end{remark}

\begin{remark}
{\cred The structure}
of $X_{w_{r}}^{m} (b)$ for $m>0$ is more subtle. 
Indeed, the argument of {\cred Proposition \ref{prop:strbylac} (2) does not apply for $m>0$}.
This is {\cred because} $X_{w_{r}}^{m}(b)_{\la} \nsubseteq I x_{r} I^{m}$ for $m>0$.
\end{remark}

Now we can prove Theorem \ref{theorem:restradlv}.

\noindent
\textit{Proof of Theorem \ref{theorem:restradlv}}. 
For simplicity, we consider the case where $\chara K =0$.
Let
\[
\la_{\mathrm{coord}}^{b,0,r} \rightarrow  \Lsympbar^{\mathrm{perf}}  
\]
be the natural projection. Note that this map factors through $\mathbb{P} (\Lsympbar^{\mathrm{perf}}) \subset \Lsympbar \subset \overline{V}$, 
which is isomorphic to $X_{\overline{w}}^{\overline{B}, \mathrm{perf}}$ by Lemma \ref{lemma:finitecomparison} and Lemma \ref{lemma:finitecomparisonodd}.
We want to show that this natural projection induces
\[
\la_{\mathrm{coord}}^{b,0,r} \simeq X_{\overline{w}}^{\overline{B}, \mathrm{perf}} \times \mathbb{A}^{\mathrm{perf}}.
\]
First, we consider the case where $n=1$ or $n=2$ for simplicity.
In this case, the embedding
$\la_{\mathrm{coord}}^{b,0,r} \subset [\la_{b,0,r}]$ is equality (Remark \ref{remark:lcoord}).
By the definition of $[\la_{b,0,r}]$, it suffices to show that
the projection $\la_{b,0,r}' \rightarrow \Lsympbar^{\mathrm{perf}}$ induces the decomposition 
\[
\la_{b,0,r}' \simeq X_{\overline{w}}^{\overline{B}, \mathrm{perf}} \times \mathbb{A}^{\mathrm{perf}}.
\]
This follows from Lemma \ref{lemma:affinefibration} and \ref{lemma:affinefibrationodd}.
In the general case, in addition to the argument of Lemma \ref{lemma:affinefibration} and \ref{lemma:affinefibrationodd}, further equations (\ref{equation:orthogonal}) need to be solved.
However, we can solve this equation by the same method as in Lemma \ref{lemma:affinefibration} and \ref{lemma:affinefibrationodd}.
{\cred This} finishes the proof.

\subsection{Family of finite type varieties $X_{h}$}
\label{subsection:familyXh}
In this section, we take $b$ as a special representative or a Coxeter representative.
\begin{definition}
\begin{enumerate}
\item
Suppose that $n$ is even and $k=1$.
We put
\[
\la^{(h)}:=  \p^{h} e_{1} \oplus \p^{h-1} e_{2} \oplus \cdots \oplus \p^{h-1} e_{n} \oplus \p^{h} e_{n+1} \oplus \p^{h} e_{n+2} \oplus \cdots \oplus \p^{h} e_{2n}.
\]
Moreover, we define the closed subscheme $X_{h}$ of $\la/ \la^{(h)}$ by
\[
X_{h}  := 
\left\{ v=(\overline{v_{1}}, \ldots \overline{v_{2n}}) \in \la / \la^{(h)} \left|
\begin{array}{l}
\langle v, F^{i}(v) \rangle = 0 \mod \p^{h+\lfloor \frac{i}{2} \rfloor}\quad (1 \leq i\leq n-1) \\
\langle v, F^{n}(v) \rangle \in \varpi^{n'} (\co_{K}/\p^{h+n'})^{\times}\\
\end{array}
\right.
\right\}.
\]
\item
Suppose that $n$ is odd and $k=1$. We put
\[
\la^{(h)} := \p^{h} e_{1} \oplus \p^{h-1} e_{2} \oplus \cdots \oplus \p^{h} e_{2n-1} \oplus \p^{h-1} e_{2n}
\]
if $k=1$.
We also put
\[
X_{h}  := 
\left\{ v=(\overline{v_{1}}, \ldots \overline{v_{2n}}) \in \la / \la^{(h)} \left|
\begin{array}{l}
\langle v, F^{i}(v) \rangle = 0 \mod \p^{h+\lfloor \frac{i}{2} \rfloor}\quad (1 \leq i\leq n-1) \\
\langle v, F^{n}(v) \rangle \in \varpi^{n'} (\co_{K}/\p^{h+n'})^{\times}\\
\end{array}
\right.
\right\}.
\]
\item 
Suppose that $k=0$.
We put 
\[
\la^{(h)} := \p^{h} e_{1} \oplus \cdots \oplus \p^{h} e_{2n}.
\]
We also put
\[
X_{h}  := 
\left\{ v=(\overline{v_{1}}, \ldots \overline{v_{2n}}) \in \la / \la^{(h)} \left|
\begin{array}{l}
\langle v, F^{i}(v) \rangle = 0 \mod \p^{h}\quad (1 \leq i\leq n-1) \\
\langle v, F^{n}(v) \rangle \in (\co_{K}/\p^{h})^{\times}\\
\end{array}
\right.
\right\}.
\]
\end{enumerate}
\end{definition}

\begin{proposition}
\label{surjectivity}
The natural map $\Lsymprat \rightarrow X_{h}$ is surjective.
\end{proposition}

\begin{proof}
We may assume that $b$ is a special representative.
It suffices to show that 
\[
X_{h+1} \rightarrow X_{h}
\]
is surjective.
For simplicity, we consider the case where $n$ is even and $k = 1$.
We also put
\[
X_{h}^{+}
:= \left\{ v=(\overline{v_{1}}, \ldots \overline{v_{2n}}) \in \la_{h} \left|
\begin{array}{l}
\langle v, F^{i}(v) \rangle = 0 \mod \p^{h+\lceil \frac{i}{2} \rceil}\quad (1 \leq i\leq n-1) \\
\langle v, F^{n}(v) \rangle \mod \p^{h+n'} \in (\co_{K}/\p^{h+n'})^{\times}\\
\end{array}
\right.
\right\}.
\]
For simplicity, we prove the surjectivity of
$
X_{2} \rightarrow X_{1}.
$
The surjectivity of $X_{1}^{+} \rightarrow X_{1}$ follows from the same argument as in the proof of Lemma \ref{lemma:affinefibration}. More precisely, we can show that
\[
X_{1}^{+} \times_{X_{1}} X_{1}^{\mathrm{perf}} \simeq \mathbb{A}^{\mathrm{perf}} \times X_{1}^{\mathrm{perf}}.
\]
Therefore, we will show the surjectivity of 
$
X_{2} \rightarrow X_{1}^{+}.
$
We also put 
\[
x_{i,j} \in \overline{\F}_{q} \textup{ is }
\textup{the image} \textup{ of }  p_{j} \textup{ of }
\begin{cases}
\varpi^{-1} \times \textup{ $i$-th component } &\textup{ if } n+1 \geq i \textup{ and } i \textup{ is odd,}\\
\textup{ $i$-th component } & \textup{ otherwise}. 
\end{cases}
\]
Here, $p_{j} \colon \co_{\breve{K}} \rightarrow \overline{\F}_{q}$ ($j \geq 0$) is the projection to the coefficient of $\varpi^{j}$ of the $\varpi$-adic expansion
(i.e., $x = \sum_{j \geq 0} [\p_{j}(x)] \varpi^{j}$).
Then 
\[
X_{2} \subset \mathbb{A}_{x_{1,0},x_{2,0}, \ldots, x_{2n-1,0}, x_{2n,0}, x_{1,1}, x_{3,1} \ldots, x_{n-1,1}, x_{n+2,1}, x_{n+4,1}, \ldots, x_{2n,1}}
\]
 is defined by the following equations {\cred \eqref{eq:firstXh}, \eqref{eq:secondXh}, \eqref{eq:secondXhprime}, \eqref{eq:thirdXh}, and \eqref{eq:thirdXhprime}}.

{\cred
\begin{align}
\begin{aligned}
\label{eq:firstXh}
&p_{i}(\langle v, F^{2i-1} (v) \rangle) \\ 
=&\; \left(
 \sum_{\substack{1 \leq j \leq n\\ j \colon \textup{odd}}} (x_{j,0} x_{2n-j,0}^{q^{2i-1}} + x_{2n-j,0} x_{j,0}^{q^{2i-1}})  - \sum_{\substack{1 \leq j \leq n\\ j\colon \textup{even}}} (x_{j,0} x_{2n-j+2,0}^{q^{2i-1}} + x_{2n-j+2,0} x_{j,0}^{q^{2i-1}})
 \right)\\
=&\;0 \quad \textup{for $i=1, \ldots, n'$.}
\end{aligned}
\end{align}
\begin{align}
\begin{aligned}
\label{eq:secondXh}
p_{i}(\langle v, F^{2i} (v) \rangle) 
&= \left( \sum_{\substack{1\leq j \leq n \\ j\colon \textup{odd}}} (x_{j,0} x_{2n+1-j,0}^{q^{2i}} - x_{2n+1-j,0} x_{j,0}^{q^{2i}}) \right)\\
&=0 \quad \textup{for $i=1, \ldots, n'-1$.}
\end{aligned}
\end{align}
\begin{align}
\begin{aligned}
\label{eq:secondXhprime}
&p_{i+1}(\langle v, F^{2i} (v) \rangle) \\
=&\; \left( \sum_{1\leq j \leq n, j\colon \textup{odd}} (x_{j,1} x_{2n+1-j,0}^{q^{2i}} - x_{2n+1-j,1} x_{j,0}^{q^{2i}} + x_{j,0} x_{2n+1-j,1}^{q^{2i}} - x_{2n+1-j,0} x_{j,1}^{q^{2i}}) \right)\\
&+  P (x_{1,0}, \ldots, x_{2n,0}) = 0 \quad \textup{for $i=1, \ldots, n'-1$.}
\end{aligned}
\end{align}
Here, $P$ is a certain polynomial.
\begin{align}
\begin{aligned}
\label{eq:thirdXh}
& p_{n'}( \langle v, F^{n} (v) \rangle) \in \F_{q}^{\times}.
\end{aligned}
\end{align}
\[
p_{n'+1} (\langle v, F^{n} (v) \rangle) \in \F_{q}, 
\]
which is equivalent to
\begin{equation}
\label{eq:thirdXhprime}
Q^{q}-Q=0,
\end{equation}
where
\begin{align}
\begin{aligned}
Q:=&\;\left(  
\sum_{\substack{1\leq j \leq n \\ j\colon \textup{odd}}} (x_{j,1} x_{2n+1-j,0}^{q^{n}} - x_{2n+1-j,1} x_{j,0}^{q^{n}} + x_{j,0} x_{2n+1-j,1}^{q^{n}} - x_{2n+1-j,0} x_{j,1}^{q^n})
\right)\\
& + P (x_{1,0}, \ldots, x_{2n,0})
\end{aligned}
\end{align}
}
with some polynomial $P$.
Note that $X_{1}^{+}$ is defined by equations (\ref{eq:firstXh}), (\ref{eq:secondXh}), and (\ref{eq:thirdXh}).
Therefore, we should solve the equation $(\ref{eq:secondXhprime})$ and $(\ref{eq:thirdXhprime})$ with respect to 
\[x_{1,1}, x_{3,1}, \ldots, x_{n-1,1}, x_{n+2,1}, \ldots, x_{2n,1}.
\]
We define $v, v' \in W:= \overline{\F}_{q}^{\oplus 2n}$ by
\begin{eqnarray*}
v'&:=& \sum_{i=1}^{n'} x_{2i-1,0} e_{2i-1} + \sum_{i=n'+1}^{n} x_{2i,0} e_{2i},\\
v''&:=& \sum_{i=1}^{n'} x_{2i-1,1} e_{2i-1} + \sum_{i=n'+1}^{n} x_{2i,1} e_{2i}.
\end{eqnarray*}
We define the coordinates $\eta_{0}, \ldots \eta_{n-1}$ by
\[
\eta_{i} := \langle v'', \overline{\sigma}^{-2i} (v') \rangle,
\]
which is a linear coordinate transformation of
\[x_{1,1}, x_{3,1}, \ldots, x_{n-1,1}, x_{n+2,1}, \ldots, x_{2n,1}.
\]
over $\co_{\Lsympbar}^{\mathrm{perf}}$.
We also define the matrix $Q = (q_{i,j})_{1\leq i,j \leq n} \in \GL_{n} (\co_{\Lsympbar}^{\mathrm{perf}})$ by
\[
^{t}\! (\overline{\sigma}^{2}(v'), \ldots, \overline{\sigma}^{2n} (v')) = Q ^{t} \! (v', \overline{\sigma}^{-2} (v') \ldots, \overline{\sigma}^{-2n-2} (v')). 
\]
By the same argument as in the proof of Lemma \ref{lemma:affinefibration},
there exits 
\[
u= 
\left(\begin{array}{ccc|ccc}
1&&\Lsymb{0}&&&\\
&\ddots&&&\hsymb{0}&\\
\Lsymb{\ast}&&1&&&\\\hline
&&&1&&\\
&\hsymb{0}&&&\ddots&\\
&&&\Lsymb{0}&&1
\end{array}\right) \in \GL_{n}(\co_{\Lsympbar}^{\mathrm{perf}}),
\]
such that $Q' := uQ$ is a matrix whose $(i,j)$-entries $q'_{i,j}$ ($i > n+1-j, j \geq n'$) are $0$.
The equations (\ref{eq:secondXhprime}) and (\ref{eq:thirdXhprime}) are equivalent to the following equations {\cred \eqref{eqn:equivfirst} and \eqref{eqn:equivsecond}}.
\begin{itemize}
\item
\[
\eta_{i}^{q^{2i}} + \sum_{j=1}^{n} q_{i,j} \eta_{j-1} +P =0
\]
for $i=1, \ldots, n'-1$. Here, $P$ is a certain polynomial of $x_{1,0}, \ldots, x_{2n,0}$.
\item
\begin{align*}
&(\eta_{n'}^{q^{n}} + \sum_{j=1}^{n} q_{n'j} \eta_{j-1} +P)^{q} \\
-&(\eta_{n'}^{q^{n}} + \sum_{j=1}^{n} q_{n'j} \eta_{j-1} +P) =0,
\end{align*}
where $P$ is a certain polynomial of $x_{1,0}, \ldots, x_{2n,0}$.
\end{itemize}
By using the $n' \times n$ upper half part $u^{\mathrm{up}}$ of $u$,
we can organize {\cred the} equations as follows:
{\cred
\begin{equation}
\label{eqn:equivfirst}
\eta_{i}^{q^{2i}} + \sum_{j=1}^{n} q_{i,j} \eta_{j-1} +P =0 \quad \textup{for $i=1, \ldots, n'-1$,}
\end{equation}
{\cred where} $P$ is a certain polynomial of $x_{1,0}, \ldots, x_{2n,0}$.
\begin{equation}
\label{eqn:equivsecond}
(\eta_{n'}^{q^{n}} + \sum_{j=1}^{n} q_{n'j} \eta_{j-1} +P)^{q} 
-(\eta_{n'}^{q^{n}} + \sum_{j=1}^{n} q_{n'j} \eta_{j-1} +P) =0,
\end{equation}
where $P$ is a certain polynomial of $x_{1,0}, \ldots, x_{2n,0}$.

By using the $n' \times n$-upper half part $u^{\mathrm{up}}$ of $u$, we can organize equations \eqref{eqn:equivfirst} and \eqref{eqn:equivsecond} as 
\begin{eqnarray*}
&&(u^{\mathrm{up}} \,
{}^t\!(\eta_{1}^{q^2}, \ldots, \eta_{n'}^{q^{2n'}}))_{i} + \sum_{j=1}^{n+1-i}  q'_{i,j} \eta_{j-1} +P=0 \quad \textup{for $i=1, \ldots, n'-1,$} \\
 &&(u^{\mathrm{up}} \,
{}^t\!(\eta_{1}^{q^{2}}, \ldots, \eta_{n'}^{q^{2n'}}))_{n'} + \sum_{j=1}^{n+1-n'}  q'_{n',j} \eta_{j-1}  + P)^{q}\\
&-&  (u^{\mathrm{up}} \,
{}^t\!(\eta_{1}^{q^{2}}, \ldots, \eta_{n'}^{q^{2n'}}))_{n'} + \sum_{j=1}^{n+1-n'}  q'_{n',j} \eta_{j-1}  + P)=0.
\end{eqnarray*}
}
The first equations can be solved with respect to $\eta_{n-i}$.
On the other hand, the second equation can be solved with respect to $\eta_{n'}$.
Therefore, this finishes the proof of surjectivity.
\end{proof}
In the following, we assume that $b$ is a Coxeter representative.

\begin{definition}
\label{definition:Gh}
We define $\mathbb{G}$ to be the smooth affine group scheme over $\F_{q}$ such that 
\[
\mathbb{G} (\overline{\F}_{q}) = G_{x,0},   \quad \mathbb{G} (\F_{q}) = G_{x,0}^{F_{b}}.
\]
Here, we put
\[
F_{b} \colon \GSp_{2n} (\breve{K}) \rightarrow \GSp_{2n} (\breve{K}); g \mapsto b\sigma(g) b^{-1}.
\]
For $h \in \Z_{\geq 1}$, we define $\mathbb{G}_{h}$ to be the smooth affine group scheme over $\F_{q}$ such that
\[
\mathbb{G}_{h}(\overline{\F}_{q}) = G_{x,0}/ G_{x,(h-1)+}, \quad  \mathbb{G}_{h} (\F_{q})= G_{x,0}^{F_{b}}/ G_{x,(h-1)+}^{F_{b}}. 
\]
Let $\mathbb{U} \subset \mathbb{G}$ (resp.\,$\mathbb{U}^{-}\subset \mathbb{G}$) be the smooth subgroup scheme whose $\overline{\F}_{q}$-points are upper (resp.\,lower) triangular unipotent matrices of $G_{x,0}$.
We denote the corresponding subgroups of $\mathbb{G}_{h}$ by $\mathbb{U}_{h}$ and $\mathbb{U}_{h}^{-}$.
\end{definition}

The following is an analogue of \cite[Proposition 7.12]{Chan2021}
\begin{proposition}
\label{proposition:Xhcomparison}
\begin{enumerate}
\item
The subgroup
$\mathbb{U}_{h}^{-} \cap F_{b}(\mathbb{U}_{h}) \subset \mathbb{G}_{h}$ consisting of matrices of the following form.
\begin{align*}
\left(
\begin{array}{cccc|cccc}
1&&&&&&&\\
\ast&1&&&&&\\
\vdots&&\ddots&&&\hsymb{0}&&\\
\ast&&&1&&&&\\ \hline
&&&&1&&&\\
&&\hsymb{0}&&&\ddots&&\\
&&&&&&1&\\
\ast&&&&\ast&\cdots&\ast&1
\end{array}
\right)
\end{align*}
\item
The subgroup
$\mathbb{U}_{h}^{-} \cap F_{b}^{-1}(\mathbb{U}_{h}^{-}) \subset \mathbb{G}_{h}$ consisting of matrices of the following form.
\begin{align*}
\left(
\begin{array}{cccc|cccc}
1&&&&&&&\\
&1&&&&&\\
&\Lsymb{\ast}&\ddots&&&\hsymb{0}&&\\
&&&1&&&&\\ \hline
0&0&\cdots&0&1&&&\\
&&&\vdots&&\ddots&&\\
&\hsymb{\ast}&&0&&\Lsymb{\ast}&1&\\
&&&0&&&&1
\end{array}
\right)
\end{align*}
\item
We have an isomorphism
\begin{align*}
X_{h} (\overline{\F}_{q}) & \simeq 
\{
g \in \mathbb{G}_{h} (\overline{\F}_{q}) | g^{-1} F_{b} (g) \in \mathbb{U}_{h}^{-} \cap F_{b} (\mathbb{U}_{h})
\}\\
& \simeq \{ g \in \mathbb{G}_{h} (\overline{\F}_{q}) | g^{-1} F_{b} (g) \in \mathbb{U}_{h}^{-} \}/ (\mathbb{U}_{h}^{-} \cap F_{b}^{-1} (\mathbb{U}_{h}^{-})).
\end{align*}
\end{enumerate}
\end{proposition}

\begin{proof}
(1) and (2) follows from the direct computation.
Therefore, we will prove (3).
The second isomorphism follows from Lemma \ref{lemma:Uh}. Therefore, we will prove the first isomorphism.
We will show that 
\[
\lambda \colon X_{h} (\overline{\F}_{q}) \rightarrow \{
g \in \mathbb{G}_{h} (\overline{\F}_{q}) | g^{-1} F_{b} (g) \in \mathbb{U}_{h}^{-} \cap F_{b} (\mathbb{U}_{h})
\}; \overline{v} \rightarrow \overline{g'_{b} (v)}
\]
gives the desired isomorphism.
Here, $g'_{b}(v)$ is defined in Remark \ref{remark:hb}.
Since any element of $X_{h}$ can be lifted to $\la_{b}^{\mathrm{symp,rat}}$ by Proposition \ref{surjectivity}, 
we can show that $\lambda$ is well-defined by the same argument as in the proof of Lemma \ref{lemma:order}.
Moreover, $\lambda$ is clearly injective.
Therefore, it suffices to show the surjectivity of $\lambda$.
We take an element $g = (g_{1}, \ldots, g_{2n})$ in the right-hand side.
By (1) and the assumption 
\[
g^{-1} F_{b} (g) \in \mathbb{U}_{h}^{-} \cap F_{b} (\mathbb{U}_{h}),
\]
we have 
\begin{align*}
&(F(g_{1}), \ldots, F(g_{2n})) b^{-1} \\
= &(g_{1}, \varpi g_{2}, \ldots, t_{i} g_{i}, \ldots, t_{n-1} g_{n-1}, t_{n} g_{2n},\\
& t_{n+1} ( (-1)^{n+1} g_{1} + \sum_{i=2, \ldots, n, 2n} \ast g_{i} ), t_{n+2} ( g_{n+1} + \ast g_{2n}), \ldots, t_{2n} (g_{2n-1} + \ast g_{2n})) b^{-1}
\end{align*}
in $\mathbb{G}_{h} (\overline{\F}_{q})$.
Here, we put
\[
t_{i} =
\begin{cases}
1  &\textup{ if $i$ is odd,}\\
\varpi &\textup{ if $i$ is even}.
\end{cases}
\]
By this equation, we can show that
$g_{1} \in X_{h}$ and $g = \lambda (g_{1})$.
\end{proof}

\begin{lemma}
\label{lemma:Uh}
The morphism 
\[
(\mathbb{U}_{h}^{-} \cap F_{b}^{-1} (\mathbb{U}_{h}^{-})) \times (\mathbb{U}_{h}^{-} \cap F_{b} (\mathbb{U}_{h})) \rightarrow \mathbb{U}_{h}^{-}; (x,g) \mapsto x^{-1}g F_{b}(x)
\]
is an isomorphism.
\end{lemma}

\begin{proof}
As in \cite[Lemma 7.13]{Chan2021}, we prove this lemma by direct computation.
First, we prove that 
\begin{equation}
\label{eq:fullconj}
(U^{-} \cap F_{b}^{-1} (U^{-})) \times (U^{-} \cap F_{b} (U)) \rightarrow U^{-}; (x,g) \mapsto x^{-1}g F_{b}(x)
\end{equation}
is an isomorphism.
Here, $U \subset  \GSp_{2n} (\breve{K})$ (resp.\,$U^{
-} \subset \GSp_{2n} (\breve{K})$)  be a subgroup consisting of upper (resp.\,lower) triangular unipotent matrices.

Take an element $A \in U^{-}$.
We want to show that there exists a unique pair 
\[
(x,g) \in (U^{-} \cap F_{b}^{-1} (U^{-})) \times (U^{-} \cap F_{b} (U)) 
\] 
such that 
\begin{equation}
\label{eq:compare}
xA = g F_{b} (x).
\end{equation}
Let$E_{i,j} \in \mathrm{M}_{2n \times 2n} (\breve{K})$ be a matrix whose $(s,t)$-component is $\delta_{i,s} \delta_{j,t}$, where $\delta$ is the Kronecker's delta.
We put
\begin{align*}
&g= 1 + \sum_{i,j} c_{i,j} E_{i,j},\\
&x = 1 + \sum_{i,j} x_{i,j} E_{i,j}, \\
&A = 1 + \sum_{i,j} a_{i,j} E_{i,j}.
\end{align*}
Then $c_{2,1}, \ldots c_{n,1}, c_{2n,1}$ (resp.\,$x_{i,j}$ ($1\leq i\leq n, 1\leq j \leq i-1$) and $x_{i,j}$ ($n+2\leq i \leq 2n, 1\leq j \leq 2n+1-i$)) determines $g$ (resp.\,$x$) since they are elements in $\GSp$.
We will compare the $(i,j)$-th entry of (\ref{eq:compare}) for $n+1 \leq i \leq 2n, 1 \leq j \leq 2n+1-i$ and $n+1 \leq i \leq 2n, n+1 \leq j\leq i-1$.
{\cred Note that, assuming both sides of (\ref{eq:compare}) lie in $\mathbb{U}^{-}(\overline{\F}_{q})$, the equality of all the entries above implies the full matrix equality (\ref{eq:compare}).}
We put 
\[
x'_{i,j} := \frac{t_{i}}{t_{j}} \sigma (x_{i,j}),
\]
where $t_{i}$ is as in the proof of Proposition \ref{proposition:Xhcomparison}.
We have
\[
F_{b} (x)
=
\left(
\begin{array}{ccccc|ccccc}
1&&&&&&&&&\\
0&1&&&&&&&&\\
0&x'_{2,1}& \ddots&&&&&&&\\
\vdots & \vdots&\ddots&\ddots&&&&&&\\
0&x'_{n-1,1}& \cdots & x'_{n-1,n-2} &1 &&&&&\\ \hline
-x'_{n+2,n+1} & x'_{n+2,1} & \cdots & x'_{n+2,n-2} & x'_{n+2,n-1} &1 &&&&\\
-x'_{n+3,n+1} & x'_{n+3,1} & \cdots & x'_{n+3,n-2} & x'_{n+3,n-1} & x'_{n+3,n+2} &\ddots &&&\\
\vdots & \vdots & \ddots & \vdots & \vdots & \vdots & \ddots & \ddots &&\\
-x'_{2n,n+1} & x'_{2n,1} & \cdots & x'_{2n,n-2} & x'_{2n,n-1} & x'_{2n,n+2} & \cdots & x'_{2n,2n-1} & 1 &\\
0 & x'_{n,1} & \cdots & x'_{n,n-2} & x'_{n,n-1} & 0 & \cdots & 0 & 0 &1 
\end{array}
\right).
\]
First, we compare the $n+1$-th column of (\ref{eq:compare}).
we have
\begin{align}
\label{eq:one}
\begin{aligned}
-x'_{n+2,n+1} &= a_{n+1,1}, \\
x'_{n+2,1} &= a_{n+1,2}, \\
&\vdots \\
x'_{n+2,n-1} &= a_{n+1,n}.
\end{aligned}
\end{align}
which determines the $n+2$-th row of $x$ uniquely.
Next, we compare the $n+1+j$-th row of (\ref{eq:compare}) ($1\leq j \leq n-2$).
We will show that the $n+2+j$-th row of $x$ are uniquely determined inductively.
Seeing the $(n+1+j,1)$-th entry, we have
\begin{align}
\label{eq:two}
-x'_{n+2+j,n+1} = x_{n+1+j,1} + \sum_{k=2}^{n-1} x_{n+1+j,k} a_{k,1} + \sum_{k=n+1}^{n+j} x_{n+1+j,k} a_{k,1} + a_{n+1+j,1}.
\end{align}
Seeing the $(n+1+j,l)$-th entry ($2 \leq  l \leq n-j$), we have
\begin{align}
\label{eq:three}
x'_{n+2+j, l-1}
= x_{n+1+j,l} + \sum_{k=l+1}^{n-1} x_{n+1+j,k} a_{k,l} + \sum_{k=n+1}^{n+j} x_{n+1+j,k} a_{k,l} + a_{n+1+j,l}.
\end{align}
Seeing the $(n+1+j,l)$-th entry ($n+1 \leq l \leq n+j$), we have
\begin{align}
\label{eq:four}
x'_{n+2+j, l+1} =x_{n+1+j,l} + \sum_{k=n+1}^{n+j} x_{n+1+j,k} a_{k,l} +a_{n+1+j,l}.
\end{align}
Inductively, these equations determine $x_{n+2+j,l}$ for $l=1, \ldots, n-j-1, n+1. \ldots, n+1+j$ uniquely.
The quality
\begin{align}
\label{eq:five}
\langle n+1+j\textup{-th row of } gF_{b} (x), n+1+l\textup{-th row of }gF_{b} (x)  \rangle =0
\end{align}
for $l=0, \ldots, j-1$ uniquely determines $x_{n+2+j, l}$ for $l= n-j, \ldots, n-1$.
Therefore, the $n+2+j$-th row of $x$ are uniquely determined.
Finally, we will see the $2n$-th row of (\ref{eq:compare}).
By seeing the $(2n,1)$-th entry and $(2n,n+j)$-th entries ($1\leq j \leq n$), we have
\begin{align}
\label{eq:six}
\begin{aligned}
&c_{2n,1} = x_{2n,1} + \sum_{k=2}^{n-1} x_{2n,k} a_{k,1} + \sum_{k=n+1}^{2n-1} x_{2n,k} a_{k,1} + a_{2n,1},\\ 
& (-1)^{n+1-j} c_{n+1-j,1} =x_{2n,n+j} + \sum_{k=n+j+1}^{2n-1} x_{2n,k} a_{k,n+j} + a_{2n,n+j},
\end{aligned}
\end{align}
which determines $c_{2,1}, \ldots, c_{n,1}, c_{2n,1}$ uniquely.
By the construction, these $x_{i,j}$, $c_{i,j}$ determines an element $x, g$ uniquely.
Therefore, we have the isomorphism (\ref{eq:fullconj}).

By $(\ref{eq:one}), \ldots, (\ref{eq:six})$, $A\in G_{x,0}$ corresponds to $x\in G_{x,0}$ and $g \in G_{x,0}$. Therefore, we have
\begin{equation}
\label{eq:Gxconj}
(\mathbb{U}^{-} \cap F_{b}^{-1} (\mathbb{U}^{-})) \times (\mathbb{U}^{-} \cap F_{b} (\mathbb{U})) \rightarrow \mathbb{U}^{-}; (x,g) \mapsto x^{-1}g F_{b}(x).
\end{equation}
Moreover, by estimating the order for $(\ref{eq:one}), \ldots, (\ref{eq:six})$, the isomorphism $(\ref{eq:Gxconj})$ induces the desired isomorphism.
\end{proof} 

\begin{remark}
\label{remark:Xhinf}
The variety $X_{h}$ here is isomorphic to the variety $X_{r}$ defined in \cite[Subsection 6.1]{Chan2021a} by Proposition \ref{proposition:Xhcomparison}.
Chan--Oi studied the Deligne--Lusztig induction by these varieties.
On the other hand, our $X_{h}$ satisfies that
\[
\mathop{\varprojlim}\limits_{h} X_{h} \simeq \la_{b}^{\mathrm{symp,rat}} \simeq \mathop{\varprojlim}\limits_{r>m}\dot{X}^{m}_{w_r}(b)_{\la} \subset \mathop{\varprojlim}\limits_{r>m}\dot{X}^{m}_{w_r}(b) \simeq X_{w}^{(U)}(b),
\]
i.e., the left-hand side can be regarded as a component of $X_{w}^{(U)} (b)$.
In this sense, we can say that Proposition \ref{proposition:Xhcomparison} translates \cite{Chan2021a}'s results into the realization of Lusztig's expectation in \cite{Lusztig1979}.
\end{remark}

\bibliographystyle{skalpha}
\bibliography{myref8.bib}

@Article{Chan2021,
  author   = {Charlotte Chan and Alexander B. Ivanov},
  journal  = {Math. Ann.},
  title    = {Affine {D}eligne-{L}usztig varieties at infinite level},
  year     = {2021},
  issn     = {0025-5831},
  number   = {3-4},
  pages    = {1801--1890},
  volume   = {380},
  doi      = {10.1007/s00208-020-02092-4},
  fjournal = {Mathematische Annalen},
  mrclass  = {11G25 (14F20 20G25)},
  mrnumber = {4297200},
  url      = {https://doi.org/10.1007/s00208-020-02092-4},
}

@Article{Deligne1976,
  author     = {Deligne, P. and Lusztig, G.},
  journal    = {Ann. of Math. (2)},
  title      = {Representations of reductive groups over finite fields},
  year       = {1976},
  issn       = {0003-486X},
  number     = {1},
  pages      = {103--161},
  volume     = {103},
  doi        = {10.2307/1971021},
  fjournal   = {Annals of Mathematics. Second Series},
  mrclass    = {20G05 (14M15)},
  mrnumber   = {393266},
  mrreviewer = {S. I. Gel\cprime fand},
  url        = {https://doi.org/10.2307/1971021},
}

@Article{Kedlaya2005,
  author     = {Kedlaya, Kiran S.},
  journal    = {Doc. Math.},
  title      = {Slope filtrations revisited},
  year       = {2005},
  issn       = {1431-0635},
  pages      = {447--525},
  volume     = {10},
  fjournal   = {Documenta Mathematica},
  mrclass    = {14F30 (11G25 12H25 14G22)},
  mrnumber   = {2184462},
  mrreviewer = {Adolfo Quir\'{o}s},
}

@Article{Pappas2008,
  author     = {Pappas, G. and Rapoport, M.},
  journal    = {Adv. Math.},
  title      = {Twisted loop groups and their affine flag varieties},
  year       = {2008},
  issn       = {0001-8708},
  note       = {With an appendix by T. Haines and Rapoport},
  number     = {1},
  pages      = {118--198},
  volume     = {219},
  doi        = {10.1016/j.aim.2008.04.006},
  fjournal   = {Advances in Mathematics},
  mrclass    = {22E67 (14M15 17B67 20G25)},
  mrnumber   = {2435422},
  mrreviewer = {Dmitry A. Timash\"{e}v},
  url        = {https://doi.org/10.1016/j.aim.2008.04.006},
}

@InCollection{Zhu2017,
  author     = {Zhu, Xinwen},
  booktitle  = {Geometry of moduli spaces and representation theory},
  publisher  = {Amer. Math. Soc., Providence, RI},
  title      = {An introduction to affine {G}rassmannians and the geometric {S}atake equivalence},
  year       = {2017},
  pages      = {59--154},
  series     = {IAS/Park City Math. Ser.},
  volume     = {24},
  mrclass    = {14M15 (14D24 20F65 22E57)},
  mrnumber   = {3752460},
  mrreviewer = {Felipe Zald\'{\i}var},
}

@Article{Zhu2017a,
  author     = {Zhu, Xinwen},
  journal    = {Ann. of Math. (2)},
  title      = {Affine {G}rassmannians and the geometric {S}atake in mixed characteristic},
  year       = {2017},
  issn       = {0003-486X},
  number     = {2},
  pages      = {403--492},
  volume     = {185},
  doi        = {10.4007/annals.2017.185.2.2},
  fjournal   = {Annals of Mathematics. Second Series},
  mrclass    = {14D24 (14L35 14M15 20G25)},
  mrnumber   = {3612002},
  mrreviewer = {Rolf Berndt},
  url        = {https://doi.org/10.4007/annals.2017.185.2.2},
}

@Article{Viehmann2008,
  author     = {Viehmann, Eva},
  journal    = {J. Algebraic Geom.},
  title      = {Moduli spaces of {$p$}-divisible groups},
  year       = {2008},
  issn       = {1056-3911},
  number     = {2},
  pages      = {341--374},
  volume     = {17},
  doi        = {10.1090/S1056-3911-07-00480-8},
  fjournal   = {Journal of Algebraic Geometry},
  mrclass    = {14L05 (14D20)},
  mrnumber   = {2369090},
  mrreviewer = {Noriko Yui},
  url        = {https://doi.org/10.1090/S1056-3911-07-00480-8},
}

@Article{Viehmann2008a,
  author     = {Viehmann, Eva},
  journal    = {Doc. Math.},
  title      = {The global structure of moduli spaces of polarized {$p$}-divisible groups},
  year       = {2008},
  issn       = {1431-0635},
  pages      = {825--852},
  volume     = {13},
  fjournal   = {Documenta Mathematica},
  mrclass    = {14L05 (14G35 14K10)},
  mrnumber   = {2466182},
  mrreviewer = {Alan Koch},
}

@Article{Feigin1990,
  author     = {Fe\u{\i}gin, Boris L. and Frenkel, Edward V.},
  journal    = {Comm. Math. Phys.},
  title      = {Affine {K}ac-{M}oody algebras and semi-infinite flag manifolds},
  year       = {1990},
  issn       = {0010-3616},
  number     = {1},
  pages      = {161--189},
  volume     = {128},
  fjournal   = {Communications in Mathematical Physics},
  mrclass    = {17B67 (14M15 17B55)},
  mrnumber   = {1042449},
  mrreviewer = {Chong Ying Dong},
  url        = {http://projecteuclid.org/euclid.cmp/1104180309},
}

@article {Chan2021a,
    AUTHOR = {Chan, Charlotte and Oi, Masao},
     TITLE = {Geometric {$L$}-packets of {H}owe-unramified toral
              supercuspidal representations},
   JOURNAL = {J. Eur. Math. Soc. (JEMS)},
  FJOURNAL = {Journal of the European Mathematical Society (JEMS)},
    VOLUME = {27},
      YEAR = {2025},
    NUMBER = {4},
     PAGES = {1465--1526},
      ISSN = {1435-9855,1435-9863},
   MRCLASS = {22E50 (11F70 11S37)},
  MRNUMBER = {4875610},
       DOI = {10.4171/jems/1396},
       URL = {https://doi.org/10.4171/jems/1396},
}

@InCollection{Rapoport2005,
  author     = {Rapoport, Michael},
  title      = {A guide to the reduction modulo {$p$} of {S}himura varieties},
  year       = {2005},
  note       = {Automorphic forms. I},
  number     = {298},
  pages      = {271--318},
  fjournal   = {Ast\'{e}risque},
  issn       = {0303-1179},
  journal    = {Ast\'{e}risque},
  mrclass    = {11G18 (11G40 14G35)},
  mrnumber   = {2141705},
  mrreviewer = {Ulrich G\"{o}rtz},
}

@article {Shimada2022,
    AUTHOR = {Shimada, Ryosuke},
     TITLE = {On some simple geometric structure of affine
              {D}eligne-{L}usztig varieties for {${\rm GL}_n$}},
   JOURNAL = {Manuscripta Math.},
  FJOURNAL = {Manuscripta Mathematica},
    VOLUME = {173},
      YEAR = {2024},
    NUMBER = {3-4},
     PAGES = {977--1001},
      ISSN = {0025-2611,1432-1785},
   MRCLASS = {20G25 (14G35)},
  MRNUMBER = {4704763},
MRREVIEWER = {Saeed\ Tafazolian},
       DOI = {10.1007/s00229-023-01489-0},
       URL = {https://doi.org/10.1007/s00229-023-01489-0},
}

@article {Shimada2021,
    AUTHOR = {Shimada, Ryosuke},
     TITLE = {Geometric structure of affine {D}eligne-{L}usztig varieties
              for {$\rm GL_3$}},
   JOURNAL = {J. Algebra},
  FJOURNAL = {Journal of Algebra},
    VOLUME = {623},
      YEAR = {2023},
     PAGES = {86--126},
      ISSN = {0021-8693,1090-266X},
   MRCLASS = {14M15 (20G15)},
  MRNUMBER = {4554715},
MRREVIEWER = {Felipe\ Zald\'ivar},
       DOI = {10.1016/j.jalgebra.2023.02.015},
       URL = {https://doi.org/10.1016/j.jalgebra.2023.02.015},
}

@InProceedings{Lusztig1979,
  author     = {Lusztig, G.},
  booktitle  = {Automorphic forms, representations and {$L$}-functions ({P}roc. {S}ympos. {P}ure {M}ath., {O}regon {S}tate {U}niv., {C}orvallis, {O}re., 1977), {P}art 1},
  title      = {Some remarks on the supercuspidal representations of {$p$}-adic semisimple groups},
  year       = {1979},
  pages      = {171--175},
  publisher  = {Amer. Math. Soc., Providence, R.I.},
  series     = {Proc. Sympos. Pure Math., XXXIII},
  mrclass    = {22E50 (20G05)},
  mrnumber   = {546595},
  mrreviewer = {Allan J. Silberger},
}

@Misc{Boyarchenko2012,
  author    = {Boyarchenko, Mitya},
  title     = {Deligne-Lusztig constructions for unipotent and p-adic groups},
  year      = {2012},
  copyright = {arXiv.org perpetual, non-exclusive license},
  doi       = {10.48550/ARXIV.1207.5876},
  publisher = {arXiv},
  url       = {https://arxiv.org/abs/1207.5876},
}

@Article{Chan2018,
  author     = {Chan, Charlotte},
  journal    = {Selecta Math. (N.S.)},
  title      = {Deligne-{L}usztig constructions for division algebras and the local {L}anglands correspondence, {II}},
  year       = {2018},
  issn       = {1022-1824},
  number     = {4},
  pages      = {3175--3216},
  volume     = {24},
  doi        = {10.1007/s00029-018-0410-6},
  fjournal   = {Selecta Mathematica. New Series},
  mrclass    = {22E50 (11S37 20G05)},
  mrnumber   = {3848018},
  mrreviewer = {Maria Sabitova},
  url        = {https://doi.org/10.1007/s00029-018-0410-6},
}

@Article{Chan2020,
  author     = {Chan, Charlotte},
  journal    = {J. Reine Angew. Math.},
  title      = {The cohomology of semi-infinite {D}eligne-{L}usztig varieties},
  year       = {2020},
  issn       = {0075-4102},
  pages      = {93--147},
  volume     = {768},
  doi        = {10.1515/crelle-2019-0039},
  fjournal   = {Journal f\"{u}r die Reine und Angewandte Mathematik. [Crelle's Journal]},
  mrclass    = {11F70 (20G05 22E50)},
  mrnumber   = {4168688},
  mrreviewer = {Caihua Luo},
  url        = {https://doi.org/10.1515/crelle-2019-0039},
}

@Article{Ivanov2016,
  author     = {Ivanov, Alexander B.},
  journal    = {Adv. Math.},
  title      = {Affine {D}eligne-{L}usztig varieties of higher level and the local {L}anglands correspondence for {$GL_2$}},
  year       = {2016},
  issn       = {0001-8708},
  pages      = {640--686},
  volume     = {299},
  doi        = {10.1016/j.aim.2016.05.019},
  fjournal   = {Advances in Mathematics},
  mrclass    = {22E50 (14D24 20G25 20G44)},
  mrnumber   = {3519479},
  mrreviewer = {Anne-Marie H. Aubert},
  url        = {https://doi.org/10.1016/j.aim.2016.05.019},
}

@Article{Ivanov2018,
  author     = {Ivanov, Alexander B.},
  journal    = {Math. Z.},
  title      = {Ramified automorphic induction and zero-dimensional affine {D}eligne-{L}usztig varieties},
  year       = {2018},
  issn       = {0025-5874},
  number     = {1-2},
  pages      = {439--490},
  volume     = {288},
  doi        = {10.1007/s00209-017-1896-x},
  fjournal   = {Mathematische Zeitschrift},
  mrclass    = {11S37 (11F70 14M15)},
  mrnumber   = {3774421},
  mrreviewer = {Jun Yu},
  url        = {https://doi.org/10.1007/s00209-017-1896-x},
}

@Article{Goertz2006,
  author     = {G\"{o}rtz, Ulrich and Haines, Thomas J. and Kottwitz, Robert E. and Reuman, Daniel C.},
  journal    = {Ann. Sci. \'{E}cole Norm. Sup. (4)},
  title      = {Dimensions of some affine {D}eligne-{L}usztig varieties},
  year       = {2006},
  issn       = {0012-9593},
  number     = {3},
  pages      = {467--511},
  volume     = {39},
  doi        = {10.1016/j.ansens.2005.12.004},
  fjournal   = {Annales Scientifiques de l'\'{E}cole Normale Sup\'{e}rieure. Quatri\`eme S\'{e}rie},
  mrclass    = {14L15 (14M15 20G25)},
  mrnumber   = {2265676},
  mrreviewer = {Alan Koch},
  url        = {https://doi.org/10.1016/j.ansens.2005.12.004},
}

@Article{Goertz2015,
  author     = {G\"{o}rtz, Ulrich and He, Xuhua},
  journal    = {Camb. J. Math.},
  title      = {Basic loci of {C}oxeter type in {S}himura varieties},
  year       = {2015},
  issn       = {2168-0930},
  number     = {3},
  pages      = {323--353},
  volume     = {3},
  doi        = {10.4310/CJM.2015.v3.n3.a2},
  fjournal   = {Cambridge Journal of Mathematics},
  mrclass    = {11G18 (14G35 14M15 20G25)},
  mrnumber   = {3393024},
  mrreviewer = {Liang Xiao},
  url        = {https://doi.org/10.4310/CJM.2015.v3.n3.a2},
}

@Article{Goertz2018,
  author   = {G\"{o}rtz, Ulrich and He, Xuhua},
  journal  = {Camb. J. Math.},
  title    = {Erratum to: {B}asic loci of {C}oxeter type in {S}himura varieties [ {MR}3393024]},
  year     = {2018},
  issn     = {2168-0930},
  number   = {1},
  pages    = {89--92},
  volume   = {6},
  doi      = {10.4310/CJM.2018.v6.n1.e4},
  fjournal = {Cambridge Journal of Mathematics},
  mrclass  = {11G18 (14G35 14M15 20G25)},
  mrnumber = {3786099},
  url      = {https://doi.org/10.4310/CJM.2018.v6.n1.e4},
}

@article {He2022,
    AUTHOR = {He, Xuhua and Nie, Sian and Yu, Qingchao},
     TITLE = {Affine {D}eligne-{L}usztig varieties with finite {C}oxeter
              parts},
   JOURNAL = {Algebra Number Theory},
  FJOURNAL = {Algebra \& Number Theory},
    VOLUME = {18},
      YEAR = {2024},
    NUMBER = {9},
     PAGES = {1681--1714},
      ISSN = {1937-0652,1944-7833},
   MRCLASS = {11G25 (20F55 20G25)},
  MRNUMBER = {4856606},
MRREVIEWER = {Lei\ Yang},
       DOI = {10.2140/ant.2024.18.1681},
       URL = {https://doi.org/10.2140/ant.2024.18.1681},
}

@InProceedings{He2018,
  author     = {He, Xuhua},
  booktitle  = {Proceedings of the {I}nternational {C}ongress of {M}athematicians---{R}io de {J}aneiro 2018. {V}ol. {II}. {I}nvited lectures},
  title      = {Some results on affine {D}eligne-{L}usztig varieties},
  year       = {2018},
  pages      = {1345--1365},
  publisher  = {World Sci. Publ., Hackensack, NJ},
  mrclass    = {14L05 (14G35 20G25)},
  mrnumber   = {3966812},
  mrreviewer = {Volker J. Heiermann},
}

@Article{He2021,
  author     = {He, Xuhua},
  journal    = {Forum Math. Pi},
  title      = {Cordial elements and dimensions of affine {D}eligne-{L}usztig varieties},
  year       = {2021},
  pages      = {Paper No. e9, 15},
  volume     = {9},
  doi        = {10.1017/fmp.2021.10},
  fjournal   = {Forum of Mathematics. Pi},
  mrclass    = {11G25 (20G25)},
  mrnumber   = {4312326},
  mrreviewer = {Olivier Dudas},
  url        = {https://doi.org/10.1017/fmp.2021.10},
}

@Article{Ivanov2022,
  author      = {Ivanov, Alexander B.},
  journal     = {J. Reine Angew. Math.},
  title       = {Arc-descent for the perfect loop functor and p-adic {D}eligne-{L}usztig spaces},
  year        = {2022},
  doi         = {doi:10.1515/crelle-2022-0060},
  lastchecked = {2022-11-20},
  url         = {https://doi.org/10.1515/crelle-2022-0060},
}

@article {Chan2019,
    AUTHOR = {Chan, Charlotte and Ivanov, Alexander},
     TITLE = {Cohomological representations of parahoric subgroups},
   JOURNAL = {Represent. Theory},
  FJOURNAL = {Representation Theory. An Electronic Journal of the American
              Mathematical Society},
    VOLUME = {25},
      YEAR = {2021},
     PAGES = {1--26},
      ISSN = {1088-4165},
   MRCLASS = {20G25 (14L15)},
  MRNUMBER = {4197070},
       DOI = {10.1090/ert/557},
       URL = {https://doi.org/10.1090/ert/557},
}

\end{document}